\documentclass[11pt]{amsart}
\usepackage{amscd,amssymb}
\usepackage{amsthm,amsmath,amssymb}
\usepackage[matrix,arrow]{xy}

\theoremstyle{definition}
\newtheorem{example}[equation]{Example}
\newtheorem{definition}[equation]{Definition}
\newtheorem{theorem}[equation]{Theorem}
\newtheorem{lemma}[equation]{Lemma}
\newtheorem{corollary}[equation]{Corollary}

\newtheorem{question}[equation]{Question}
\newtheorem*{question*}{Question}
\newtheorem{problem}[equation]{Problem}
\newtheorem*{problem*}{Problem}

\theoremstyle{remark}
\newtheorem{remark}[equation]{Remark}

\makeatletter\@addtoreset{equation}{section} \makeatother

\author{Ivan Cheltsov}

\title{Two local inequalities}

\address{\begin{tabbing}
\hspace*{28 em}\=\kill
School of Mathematics\\
University of Edinburgh\\
Edinburgh EH9 3JZ, UK\\
\\
\texttt{I.Cheltsov@ed.ac.uk}
\end{tabbing}}

\dedicatory{to the~memory of Vasily Iskovskikh (1939--2009)}

\thanks{The author would like to thank I.\,Dolgachev, V.\,Golyshev, D.\,Kosta, Yu.\,Prokhorov and K.\,Shramov
for useful comments and conversations. The would like to thank
T.\,Dokchitser for the~proof of Lemma~\ref{lemma:A5-6}.}

\begin{document}

\begin{abstract}
Many results in algebraic geometry can be proved using local
inequalities that relate singularities of divisors on algebraic
varieties and their multiplicities. For example, the~famous
theorem of Iskovskikh and Manin about the~nonrationality of every
smooth quartic threefold can be derived from one such inequality.
We prove two new local inequalities for divisors on smooth
surfaces. Using orbifold multiplicities, we can apply one of these
inequalities to divisors on two-dimensional orbifolds. We consider
some applications of the obtained inequalities.
\end{abstract}

\maketitle

\tableofcontents

We assume that all varieties are projective, normal, and defined
over $\mathbb{C}$.

\section{Introduction}
\label{section:intro}

Let $S$ be a~surface, let $O$ be a~smooth point of $S$, let
$\Delta_{1}$ and $\Delta_{2}$ be curves on $S$~such~that
$$
O\in\Delta_{1}\cap\Delta_{2},
$$
both $\Delta_{1}$ and $\Delta_{2}$ are irreducible and reduced,
both $\Delta_{1}$ and $\Delta_{2}$ are smooth at $O$, and
$\Delta_{1}$ intersects the~curve $\Delta_{2}$ transversally at
the~point $O$. Let $a_{1}$ and $a_{2}$ be rational~numbers.

\begin{theorem}
\label{theorem:baby-adjunction} Let $D$ be an effective
$\mathbb{Q}$-divisor on the~surface $S$ such that the~log pair
$$
\Big(S,\ D+a_{1}\Delta_{1}+a_{2}\Delta_{2}\Big)
$$
is not log canonical at $O$. Suppose that
$\Delta_{1}\not\subseteq\mathrm{Supp}(D)\not\supseteq\Delta_{2}$,
$a_{1}\geqslant 0$, $a_{2}\geqslant 0$. Then
$$
\mathrm{mult}_{O}\Big(D\cdot\Delta_{1}\Big)>1-a_{2}\ \text{or}\ \mathrm{mult}_{O}\Big(D\cdot\Delta_{2}\Big)>1-a_{1}.%
$$
\end{theorem}

\begin{proof}
If $a_{1}>1$ and $a_{2}>1$, then we are done. So we may assume
that $a_{1}\leqslant 1$. Then
$$
\Big(S,\ D+\Delta_{1}+a_{2}\Delta_{2}\Big)
$$
is not log canonical~at~$O$.~Then it follows from
\cite[Theorem~17.6]{Ko91} (see Lemma~\ref{lemma:adjunction}) that
$$
\mathrm{mult}_{O}\Big(\big(D+a_{2}\Delta_{2}\big)\cdot\Delta_{1}\Big)>1,%
$$
which implies the~required inequality.
\end{proof}

The following analogue of Theorem~\ref{theorem:baby-adjunction} is
implicitly obtained in \cite{Ko09}.

\begin{theorem}
\label{theorem:Dimitra} Let $D$ be an effective
$\mathbb{Q}$-divisor on the~surface $S$ such that the~log pair
$$
\Big(S,\ D+a_{1}\Delta_{1}+a_{2}\Delta_{2}\Big)
$$
is not log canonical at $O$. Suppose that
$\Delta_{1}\not\subseteq\mathrm{Supp}(D)\not\supseteq\Delta_{2}$,
$a_{1}\geqslant 0$, $a_{2}\geqslant 0$. Then
$$
\mathrm{mult}_{O}\Big(D\cdot\Delta_{1}\Big)>2a_{1}-a_{2}\ \text{or}\ \mathrm{mult}_{O}\Big(D\cdot\Delta_{2}\Big)>\frac{3}{2}a_{2}-a_{1}%
$$
if the~inequality $a_{1}+a_{2}/2\leqslant 1$ holds.
\end{theorem}

We prove the~following generalization of
Theorems~\ref{theorem:baby-adjunction} and \ref{theorem:Dimitra}
in Section~\ref{section:inequality-2}.

\begin{theorem}
\label{theorem:I} Let $D$ be an effective $\mathbb{Q}$-divisor on
the~surface $S$ such that the~log pair
$$
\Big(S,\ D+a_{1}\Delta_{1}+a_{2}\Delta_{2}\Big)
$$
is not log canonical at $O$. Suppose that
$\Delta_{1}\not\subseteq\mathrm{Supp}(D)\not\supseteq\Delta_{2}$,
$a_{1}\geqslant 0$, $a_{2}\geqslant 0$. Then
$$
\mathrm{mult}_{O}\Big(D\cdot\Delta_{1}\Big)>M+Aa_{1}-a_{2}\ \text{or}\ \mathrm{mult}_{O}\Big(D\cdot\Delta_{2}\Big)>N+Ba_{2}-a_{1}%
$$
if  $\alpha a_{1}+\beta a_{2}\leqslant 1$, where
$A,B,M,N,\alpha,\beta$ are non-negative rational numbers such that
\begin{itemize}
\item the~inequalities $A(B-1)\geqslant 1\geqslant\mathrm{max}(M,N)$ hold,%
\item the~inequalities $\alpha(A+M-1)\geqslant A^{2}(B+N-1)\beta$ and $\alpha(1-M)+A\beta\geqslant A$ hold,%
\item either the~inequality $2M+AN\leqslant 2$ holds or
$$
\alpha\big(B+1-MB-N\big)+\beta\big(A+1-AN-M\big)\geqslant AB-1.
$$%
\end{itemize}
\end{theorem}

We consider several applications of Theorem~\ref{theorem:I} in
Section~\ref{section:del-Pezzo-orbifolds}.

\begin{remark}
\label{remark:Dimitra} The proof of Theorem~\ref{theorem:I} is
almost identical to~the~proof of Theorem~\ref{theorem:Dimitra}.
\end{remark}

The following result is obtained in \cite{Co00} (cf.
\cite[Lemma~3.3]{Isk01}).

\begin{theorem}
\label{theorem:Corti} Let $\mathcal{M}$ be a~linear system on
the~surface $S$ that does not have fixed components, and let
$\epsilon$ be a~positive rational number. Suppose that the~log
pair
$$
\Big(S,\ \epsilon\mathcal{M}+a_{1}\Delta_{1}+a_{2}\Delta_{2}\Big)
$$
is not Kawamata log terminal at the~point $O$. Let $M_{1}$ and
$M_{2}$ be general curves in $\mathcal{M}$. Then
\begin{equation}
\label{equation:Corti} \mathrm{mult}_{O}\Big(M_{1}\cdot
M_{2}\Big)\geqslant \left\{\aligned
&\frac{4\big(1-a_{1}\big)\big(1-a_{2}\big)}{\epsilon^{2}}\ \text{if}\ a_{1}\geqslant 0\ \text{or}\ a_{2}\geqslant 0,\\
&\frac{4\big(1-a_{1}-a_{2}\big)}{\epsilon^{2}}\ \text{if}\ a_{1}\leqslant 0\ \text{and}\ a_{2}\leqslant 0.\\
\endaligned
\right.
\end{equation}
and if the~inequality~\ref{equation:Corti} is an~equality, then
$$
\mathrm{mult}_{O}\big(\mathcal{M}\big)=\frac{2(a_{1}-1)}{\epsilon},
$$
the~log pair
$(S,\epsilon\mathcal{M}+a_{1}\Delta_{1}+a_{2}\Delta_{2})$ is log
canonical, and $a_{1}=a_{2}\geqslant 0$.%
\end{theorem}

\begin{proof}
Let us show how to prove the~inequality~\ref{equation:Corti} (see
the~proof of \cite[Lemma~3.3]{Isk01}).

Let $\phi\colon\bar{S}\to S$ be a~blow up of the~point $O$, let
$E$ be the~exceptional curve of $\phi$. Then
$$
K_{\bar{S}}+\epsilon\bar{\mathcal{M}}+a_{1}\bar{\Delta}_{1}+a_{2}\bar{\Delta}_{2}+\Big(\epsilon\mathrm{mult}_{O}\big(\mathcal{M}\big)+a_{1}+a_{2}-1\Big)E\equiv\phi^{*}\Big(K_{S}+\epsilon\mathcal{M}+a_{1}\Delta_{1}+a_{2}\Delta_{2}\Big),%
$$
where $\bar{\mathcal{M}}$, $\bar{\Delta}_{1}$, $\bar{\Delta}_{2}$
are proper transforms of
 $\mathcal{M}$, $\Delta_{1}$,
$\Delta_{2}$ on the~surface $\bar{S}$, respectively.

Let $M_{1}$ and $M_{2}$ be general curves in the~linear system
$\mathcal{M}$, let $\bar{M}_{1}$ and $\bar{M}_{2}$ be the~proper
transforms of the~curves $\bar{M}_{1}$ and $\bar{M}_{2}$ on
the~surface $\bar{S}$, respectively.

Suppose that $\epsilon\mathrm{mult}_{O}(\mathcal{M})\geqslant
2-a_{1}-a_{2}$. Then
$$
\mathrm{mult}_{O}\Big(M_{1}\cdot
M_{2}\Big)\geqslant\mathrm{mult}^{2}_{O}(\mathcal{M})
\geqslant\frac{\big(2-a_{1}-a_{2}\big)^{2}}{\epsilon^{2}}\geqslant
\left\{\aligned
&\frac{4\big(1-a_{1}\big)\big(1-a_{2}\big)}{\epsilon^{2}}\ \text{if}\ a_{1}\geqslant 0\ \text{or}\ a_{2}\geqslant 0,\\
&\frac{4\big(1-a_{1}-a_{2}\big)}{\epsilon^{2}}\ \text{if}\ a_{1}\leqslant 0\ \text{and}\ a_{2}\leqslant 0,\\
\endaligned
\right.
$$
which implies the~inequality~~\ref{equation:Corti} in the~case
when $\epsilon\mathrm{mult}_{O}(\mathcal{M})\geqslant
2-a_{1}-a_{2}$.

To complete the~proof we may assume that
$\epsilon\mathrm{mult}_{O}(\mathcal{M})<2-a_{1}-a_{2}$. Then
the~log pair
$$
\Bigg(\bar{S},\ \epsilon\bar{\mathcal{M}}+a_{1}\bar{\Delta}_{1}+a_{2}\bar{\Delta}_{2}+\Big(\epsilon\mathrm{mult}_{O}\big(\mathcal{M}\big)+a_{1}+a_{2}-1\Big)E\Bigg)%
$$
is not Kawamata log terminal at some point $Q\in E$. Then
\begin{equation}
\label{equation:multiplicity}
\mathrm{mult}_{O}\Big(M_{1}\cdot M_{2}\Big)\geqslant\mathrm{mult}^{2}_{O}\big(\mathcal{M}\big)+\mathrm{mult}_{Q}\Big(\bar{M}_{1}\cdot \bar{M}_{2}\Big).%
\end{equation}

To complete the~proof, we must consider the~following possible
cases:
\begin{itemize}
\item $Q\not\in \bar{\Delta}_{1}\cup\bar{\Delta}_{2}$,%
\item $Q=E\cap\bar{\Delta}_{1}$,%
\item $Q=E\cap\bar{\Delta}_{2}$.%
\end{itemize}

Suppose that $Q\not\in \bar{\Delta}_{1}\cup\bar{\Delta}_{2}$. Then
the~log pair
$$
\Big(\bar{S},\epsilon\bar{\mathcal{M}}+\big(\epsilon\mathrm{mult}_{O}\big(\mathcal{M}\big)+a_{1}+a_{2}-1\big)E\Big)%
$$
is not Kawamata log terminal at the~point $Q$. Thus, we may assume
that
$$
\mathrm{mult}_{Q}\Big(\bar{M}_{1}\cdot
\bar{M}_{2}\Big)\geqslant\frac{4\big(2-\epsilon\mathrm{mult}_{O}\big(\mathcal{M}\big)-a_{1}-a_{2}\big)}{\epsilon^{2}}
$$
by induction. Thus, it follows from
the~inequality~\ref{equation:multiplicity} that
$$
\mathrm{mult}_{O}\Big(M_{1}\cdot M_{2}\Big)\geqslant\mathrm{mult}^{2}_{O}\big(\mathcal{M}\big)+\frac{4\big(2-\epsilon\mathrm{mult}_{O}\big(\mathcal{M}\big)-a_{1}-a_{2}\big)}{\epsilon^{2}}\geqslant\frac{4\big(1-a_{1}-a_{2}\big)}{\epsilon^{2}},%
$$
which implies the~required inequality if $a_{1}\leqslant 0$ and
$a_{2}\leqslant 0$. Hence, it follows from the~equality
$$
\frac{4\big(1-a_{1}-a_{2}\big)}{\epsilon^{2}}=\frac{4\big(1-a_{1}\big)\big(1-a_{2}\big)}{\epsilon^{2}}-\frac{a_{1}a_{2}}{\epsilon^{2}}%
$$
that we may assume that $a_{1}>0$ and $a_{2}>0$. Then
$$
\mathrm{mult}_{O}\Big(M_{1}\cdot M_{2}\Big)\geqslant\mathrm{mult}^{2}_{O}\big(\mathcal{M}\big)+\frac{4\big(2-\epsilon\mathrm{mult}_{O}\big(\mathcal{M}\big)-a_{1}-a_{2}\big)}{\epsilon^{2}}>\frac{4\big(1-a_{1}\big)\big(1-a_{2}\big)}{\epsilon^{2}}+\big(a_{1}-a_{2}\big)^{2},%
$$
because $\epsilon\mathrm{mult}_{O}(\mathcal{M})<2-a_{1}-a_{2}$.
This completes the~proof in the~case when $Q\not\in
\bar{\Delta}_{1}\cup\bar{\Delta}_{2}$.

Without loss of generality, we may assume that
$Q=E\cap\bar{\Delta}_{1}$. Then the~log pair
$$
\Big(\bar{S},\ \epsilon\bar{\mathcal{M}}+a_{1}\bar{\Delta}_{1}+\big(\epsilon\mathrm{mult}_{O}\big(\mathcal{M}\big)+a_{1}+a_{2}-1\big)E\Big)%
$$
is not Kawamata log terminal at the~point $Q$. Thus, we may assume
that
$$
\mathrm{mult}_{Q}\Big(\bar{M}_{1}\cdot \bar{M}_{2}\Big)\geqslant
\left\{\aligned
&\frac{4\big(1-a_{1}\big)\big(2-\epsilon\mathrm{mult}_{O}\big(\mathcal{M}\big)-a_{1}-a_{2}\big)}{\epsilon^{2}}\ \text{if}\ a_{1}\geqslant 0\ \text{or}\ \epsilon\mathrm{mult}_{O}\big(\mathcal{M}\big)+a_{1}+a_{2}\geqslant 1,\\
&\frac{4\big(2-\epsilon\mathrm{mult}_{O}\big(\mathcal{M}\big)-2a_{1}-a_{2}\big)}{\epsilon^{2}}\ \text{if}\ a_{1}\leqslant 0\ \text{and}\ \epsilon\mathrm{mult}_{O}\big(\mathcal{M}\big)+a_{1}+a_{2}\leqslant 1,\\
\endaligned
\right.
$$
by induction. To complete the~proof, we must consider
the~following possible cases:
\begin{itemize}
\item $a_{2}\geqslant 0$ and either $a_{1}\geqslant 0$ or $\epsilon\mathrm{mult}_{O}(\mathcal{M})+a_{1}+a_{2}\geqslant 1$,%
\item $a_{2}\geqslant 0$ and $a_{1}\leqslant 0$ and $\epsilon\mathrm{mult}_{O}(\mathcal{M})+a_{1}+a_{2}\leqslant 1$,%
\item $a_{2}\leqslant 0$ and either $a_{1}\geqslant 0$ or $\epsilon\mathrm{mult}_{O}(\mathcal{M})+a_{1}+a_{2}\geqslant 1$,%
\item $a_{2}\leqslant 0$ and $a_{1}\leqslant 0$ and $\epsilon\mathrm{mult}_{O}(\mathcal{M})+a_{1}+a_{2}\leqslant 1$.%
\end{itemize}

Suppose that $a_{2}\geqslant 0$ and either $a_{1}\geqslant 0$ or
$\epsilon\mathrm{mult}_{O}(\mathcal{M})+a_{1}+a_{2}\geqslant 1$.
Then
$$
\mathrm{mult}_{Q}\Big(\bar{M}_{1}\cdot \bar{M}_{2}\Big)\geqslant\frac{4\big(1-a_{1}\big)\big(2-\epsilon\mathrm{mult}_{O}\big(\mathcal{M}\big)-a_{1}-a_{2}\big)}{\epsilon^{2}},%
$$
and it follows from the~inequality~\ref{equation:multiplicity}
that
$$
\mathrm{mult}_{O}\Big(M_{1}\cdot M_{2}\Big)\geqslant\mathrm{mult}^{2}_{O}\big(\mathcal{M}\big)+\frac{4\big(1-a_{1}\big)\big(2-\epsilon\mathrm{mult}_{O}\big(\mathcal{M}\big)-a_{1}-a_{2}\big)}{\epsilon^{2}}\geqslant\frac{4\big(1-a_{1}\big)\big(1-a_{2}\big)}{\epsilon^{2}}+\frac{4a_{1}^{2}}{\epsilon^{2}},%
$$
which completes the~proof in the~case when $a_{2}\geqslant 0$ and
either $a_{1}\geqslant 0$ or
$\epsilon\mathrm{mult}_{O}(\mathcal{M})+a_{1}+a_{2}\geqslant 1$.

Suppose that $a_{2}\geqslant 0$ and $a_{1}\leqslant 0$ and
$\epsilon\mathrm{mult}_{O}(\mathcal{M})+a_{1}+a_{2}\leqslant 1$.
Then
$$
\mathrm{mult}_{Q}\Big(\bar{M}_{1}\cdot \bar{M}_{2}\Big)\geqslant\frac{4\big(2-\epsilon\mathrm{mult}_{O}\big(\mathcal{M}\big)-2a_{1}-a_{2}\big)}{\epsilon^{2}},%
$$
and it follows from the~inequality~\ref{equation:multiplicity}
that
$$
\mathrm{mult}_{O}\Big(M_{1}\cdot
M_{2}\Big)\geqslant\mathrm{mult}^{2}_{O}\big(\mathcal{M}\big)+\frac{4\big(2-\epsilon\mathrm{mult}_{O}\big(\mathcal{M}\big)-2a_{1}-a_{2}\big)}{\epsilon^{2}}
\geqslant\frac{4\big(1-a_{1}\big)\big(1-a_{2}\big)}{\epsilon^{2}}-\frac{4a_{1}\big(1+a_{2}\big)}{\epsilon^{2}},%
$$
which completes the~proof in the~case when $a_{2}\geqslant 0$ and
$a_{1}\leqslant 0$ and
$\epsilon\mathrm{mult}_{O}(\mathcal{M})+a_{1}+a_{2}\leqslant 1$.

Suppose that $a_{2}\leqslant 0$ and $a_{1}\leqslant 0$ and
$\epsilon\mathrm{mult}_{O}(\mathcal{M})+a_{1}+a_{2}\leqslant 1$.
Then
$$
\mathrm{mult}_{Q}\Big(\bar{M}_{1}\cdot \bar{M}_{2}\Big)\geqslant\frac{4\big(2-\epsilon\mathrm{mult}_{O}\big(\mathcal{M}\big)-2a_{1}-a_{2}\big)}{\epsilon^{2}},%
$$
and it follows from the~inequality~\ref{equation:multiplicity}
that
$$
\mathrm{mult}_{O}\Big(M_{1}\cdot
M_{2}\Big)\geqslant\mathrm{mult}^{2}_{O}\big(\mathcal{M}\big)+\frac{4\big(2-\epsilon\mathrm{mult}_{O}\big(\mathcal{M}\big)-2a_{1}-a_{2}\big)}{\epsilon^{2}}
\geqslant\frac{4\big(1-a_{1}-a_{2}\big)}{\epsilon^{2}}-\frac{4a_{1}}{\epsilon^{2}},%
$$
which completes the~proof in the~case when $a_{2}\leqslant 0$ and
$a_{1}\leqslant 0$ and
$\epsilon\mathrm{mult}_{O}(\mathcal{M})+a_{1}+a_{2}\leqslant 1$.

Thus, we may assume that $a_{2}\leqslant 0$ and either
$a_{1}\geqslant 0$ or
$\epsilon\mathrm{mult}_{O}(\mathcal{M})+a_{1}+a_{2}\geqslant 1$.
Then
$$
\mathrm{mult}_{Q}\Big(\bar{M}_{1}\cdot \bar{M}_{2}\Big)\geqslant\frac{4\big(1-a_{1}\big)\big(2-\epsilon\mathrm{mult}_{O}\big(\mathcal{M}\big)-a_{1}-a_{2}\big)}{\epsilon^{2}}.%
$$
and it follows from the~inequality~\ref{equation:multiplicity}
that
$$
\mathrm{mult}_{O}\Big(M_{1}\cdot M_{2}\Big)\geqslant\mathrm{mult}^{2}_{O}\big(\mathcal{M}\big)+\frac{4\big(1-a_{1}\big)\big(2-\epsilon\mathrm{mult}_{O}\big(\mathcal{M}\big)-a_{1}-a_{2}\big)}{\epsilon^{2}}\geqslant\frac{4-4\big(1-a_{1}\big)\big(a_{1}+a_{2}\big)}{\epsilon^{2}}.%
$$

In the~case when $a_{1}\geqslant 0$, we have
$$
\mathrm{mult}_{O}\Big(M_{1}\cdot
M_{2}\Big)\geqslant\frac{4-4\big(1-a_{1}\big)\big(a_{1}+a_{2}\big)}{\epsilon^{2}}=
\frac{4\big(1-a_{1}\big)\big(1-a_{2}\big)}{\epsilon^{2}}+\frac{4a_{1}^{2}}{\epsilon^{2}},%
$$
which completes the~proof. In the~case when $a_{1}\leqslant 0$, we
have
$$
\mathrm{mult}_{O}\Big(M_{1}\cdot
M_{2}\Big)\geqslant\frac{4-4\big(1-a_{1}\big)\big(a_{1}+a_{2}\big)}{\epsilon^{2}}=
\frac{4\big(1-a_{1}-a_{2}\big)}{\epsilon^{2}}+\frac{4a_{1}\big(a_{1}+a_{2}\big)}{\epsilon^{2}},%
$$
which completes the~proof.
\end{proof}

We prove the~following analogue of Theorem~\ref{theorem:Corti} in
Section~\ref{section:inequality-1}.

\begin{theorem}
\label{theorem:II} Let $\mathcal{M}$ be a~linear system on
the~surface $S$ that does not have fixed components, and let
$\epsilon$ be a~positive rational number. Suppose that the~log
pair
$$
\Big(S,\ \epsilon\mathcal{M}+a_{1}\Delta_{1}\Big)
$$
is not terminal at the~point $O$. Let $M_{1}$ and $M_{2}$ be
general curves in $\mathcal{M}$. Then
\begin{equation}
\label{equation:Cheltsov} \mathrm{mult}_{O}\Big(M_{1}\cdot
M_{2}\Big)\geqslant \left\{\aligned
&\frac{1-2a_{1}}{\epsilon^{2}}\ \text{if}\ a_{1}\geqslant -1\slash 2,\\
&\frac{-4a_{1}}{\epsilon^{2}}\ \text {if}\ a_{1}\leqslant -1\slash 2,\\
\endaligned
\right.
\end{equation}
and if the~inequality~\ref{equation:Cheltsov} is, actually, an
equality, then $(S,\epsilon\mathcal{M}+a_{1}\Delta_{1})$ is
canonical and
\begin{itemize}
\item either $-a_{1}\in\mathbb{N}$ and $\mathrm{mult}_{O}(\mathcal{M})=2/\epsilon$,%
\item or $a_{1}=0$ and $\mathrm{mult}_{O}(\mathcal{M})=1/\epsilon$.%
\end{itemize}
\end{theorem}

We consider one application of Theorem~\ref{theorem:II} in
Section~\ref{section:icosahedron}.

\section{Preliminaries}%
\label{section:preliminaries}

Let $\phi,\psi$ and $\upsilon_{1},\ldots,\upsilon_{\gamma}$ be
non-zero polynomials in $\mathbb{C}[z_1,\ldots, z_n]$ such that
the~locus
$$
\upsilon_{1}=\cdots=\upsilon_{\gamma}=0\subset\mathbb{C}^{n}
$$
has dimension at most $n-2$, let $a,b$ and $c$ be non-negative
rational numbers. Put
$$
\Omega=\frac{\big|\psi\big|^{b}}{\big|\phi\big|^{a}\Big(\big|\upsilon_{1}\big|+\cdots+\big|\upsilon_{\gamma}\big|\Big)^{c}},
$$
and let $O\in\mathbb{C}^{n}$ be the~origin.

\begin{question}
\label{question:stupid} When $\Omega$ is square-integrable near
the~point $O$?
\end{question}

The answer to Question~\ref{question:stupid} is given in
Example~\ref{example:OMEGA}. By putting
$$
c_{0}\big(\Omega\big)=\mathrm{sup}\left\{\epsilon\in\mathbb{Q}\
\Big|\ \text{the~function}\ \Omega^{2\epsilon}~\text{is locally
integrable in near
$O\in\mathbb{C}^n$}\right\}\in\mathbb{Q}_{\geqslant 0}
\cup\{+\infty\},%
$$
we see that $\Omega$ is locally integrable near the~point $O$
$\iff$ $c_{0}(\Omega)>1$.

\begin{example}
\label{remark:sum-of-powers} Let $m_{1},\ldots,m_{n}$ be positive
integers. Then
$$
\mathrm{min}\left(1,\ \sum_{i=1}^{n} \frac{1}{m_{i}}\right)=c_{0}\left(\frac{1}{\big|\sum_{i=1}^{n}z_{i}^{m_{i}}\big|}\right)\geqslant c_{0}\left(\frac{1}{\big|\prod_{i=1}^{n}z_{i}^{m_{i}}\big|}\right)=\mathrm{min}\left(\frac{1}{m_{1}},\frac{1}{m_{2}},\ldots,\frac{1}{m_{n}}\right).%
$$
\end{example}

Let $X$ be a~variety with rational singularities. Let us consider
 a~formal linear combination
$$
B_{X}=\sum_{i=1}^{r}a_{i}B_{i}+\sum_{i=1}^{l}c_{i}\mathcal{M}_{i},
$$
where $a_{i}$ and $c_{i}$ are rational numbers, $B_{i}$ is a~prime
Weil divisor on   $X$, and $\mathcal{M}_{i}$ is a~linear system on
the~variety $X$ that has no fixed components. We assume that
$B_{i}\neq B_{j}$ and
$\mathcal{M}_{i}\neq\mathcal{M}_{j}$~for~$i\neq j$.

\begin{remark}
\label{remark:Q-divisor} Let $k$ be a big enough natural number.
For every $i\in\{1,\ldots,l\}$ and $j\in\{1,\ldots,k\}$, let
$M_{i}^{j}$ be a~sufficiently general element in the~linear system
$\mathcal{M}_{i}$. Then replacing every $\mathcal{M}_{i}$ by
$$
\frac{M_{i}^{1}+M_{i}^{2}+M_{i}^{3}+M_{i}^{4}+\cdots+M_{i}^{k}}{k},
$$
we can always consider $B_{X}$ as a~$\mathbb{Q}$-divisor on
the~variety $X$.
\end{remark}

Suppose that $K_{X}+B_{X}$ is a~$\mathbb{Q}$-Cartier divisor.

\begin{definition}
\label{definition:mobile-effective} We say that $B_{X}$ is a~
boundary of the~log pair $(X,B_{X})$. We say that
\begin{itemize}
\item $B_{X}$ is mobile if $a_{i}=0$ for every $i\in\{1,\ldots,r\}$,%
\item $B_{X}$ is effective if $a_{i}\geqslant 0$ for every $i\in\{1,\ldots,r\}$ and $c_{j}\geqslant 0$ for every $j\in\{1,\ldots,l\}$.%
\end{itemize}
\end{definition}

Let $\pi\colon\bar{X}\to X$ be a~birational morphism such that
$\bar{X}$ is smooth. Put
$$
B_{\bar{X}}=\sum_{i=1}^{r}a_{i}\bar{B}_{i}+\sum_{i=1}^{l}c_{i}\bar{\mathcal{M}}_{i},
$$
where $\bar{B}_{i}$ is the~proper transform of the~divisor $B_{i}$
on the~variety $\bar{X}$, and $\bar{\mathcal{M}}_{i}$ is
the~proper transform of the~linear system $\bar{\mathcal{M}}_{i}$
on the~variety $\bar{X}$. Then
$$
K_{\bar{X}}+B_{\bar{X}}\equiv\pi^{*}\Big(K_{X}+B_{X}\Big)+\sum_{i=1}^{m}d_{i}E_{i},
$$
where $d_{i}\in\mathbb{Q}$, and $E_{i}$ is an exceptional divisor
of the~morphism $\pi$. Suppose that
$$
\left(\bigcup_{i=1}^{r}\bar{B}_{i}\right)\bigcup\left(\bigcup_{i=1}^{m}E_{i}\right)
$$
is a~divisor with simple normal crossing, and
$\bar{\mathcal{M}}_{i}$ is base point free for every
$i\in\{1,\ldots,l\}$. Put
$$
B^{\bar{X}}=B_{\bar{X}}-\sum_{i=1}^{m}d_{i}E_{i},
$$
and take $\epsilon\in\mathbb{Q}$ such that
$1\geqslant\epsilon\geqslant 0$. Then $(\bar{X}, B^{\bar{X}})$ is
called the~log pull back of $(X,B_{X})$.

\begin{definition}
\label{definition:log-canonical-singularities} The log pair $(X,
B_{X})$ is $\epsilon$-log canonical (respectively, $\epsilon$-log
terminal) if
\begin{itemize}
\item the~inequality $a_{i}\leqslant 1-\epsilon$ holds
(respectively, the~inequality $a_{i}<1-\epsilon$
holds),%
\item the~inequality $d_{j}\geqslant  -1+\epsilon$ holds (respectively, the~inequality $d_{j}>-1+\epsilon$ holds),%
\end{itemize}
for every $i\in\{1,\ldots,r\}$ and for every $j\in\{1,\ldots,m\}$.
\end{definition}

We say that $(X, B_{X})$ has log canonical singularities
(respectively, log terminal singularities) if the~log pair $(X,
B_{X})$ is $0$-log canonical (respectively, $0$-log terminal).

\begin{remark}
\label{remark:convexity} Let $P$ be a~point in $X$, and let
$\Delta$ be an effective  divisor
$$
\Delta=\sum_{i=1}^{r}\epsilon_{i}B_{i}\equiv B_{X},%
$$
where $\epsilon_{i}$ is a~non-negative rational number. Suppose
that
\begin{itemize}
\item the~boundary $B_{X}$ is effective,
\item the~divisor $\Delta$ is a~$\mathbb{Q}$-Cartier divisor,%
\item the~log pair  $(X, \Delta)$ is log canonical in the~point $P\in X$,%
\end{itemize}
but the~log pair $(X, B_{X})$ is not log canonical in the~point
$P\in X$. Put
$$
\alpha=\mathrm{min}\left\{\frac{a_{i}}{\epsilon_{i}}\ \Big\vert\ \epsilon_{i}\ne 0\right\},%
$$
where $\alpha$ is well defined, because there is $\epsilon_{i}\ne
0$. Then $\alpha<1$, the~log pair
$$
\left(X,\ \sum_{i=1}^{r}\frac{a_{i}-\alpha\epsilon_{i}}{1-\alpha}B_{i}+\sum_{i=1}^{l}c_{i}\mathcal{M}_{i}\right)%
$$
is not log canonical at the~point $P\in X$, the~equivalence
$$
\sum_{i=1}^{r}\frac{a_{i}-\alpha\epsilon_{i}}{1-\alpha}B_{i}+\sum_{i=1}^{l}c_{i}\mathcal{M}_{i}\equiv B_{X}\equiv\Delta%
$$
holds, and at least one irreducible component of
$\mathrm{Supp}(\Delta)$ is not contained in
$$
\mathrm{Supp}\left(\sum_{i=1}^{r}\frac{a_{i}-\alpha\epsilon_{i}}{1-\alpha}B_{i}\right).
$$
\end{remark}

Let $D_{X}$ be another boundary on the~variety $X$ such that
the~divisor
$$
K_{X}+B_{X}+D_{X}
$$
is a~$\mathbb{Q}$-Cartier divisor. Let $Z\subseteq X$ be a~closed
subvariety.

\begin{definition}
\label{definition:lct-local} The $\epsilon$-log canonical
threshold of the~boundary $D_{X}$ along $Z$ is
$$
\mathrm{c}^{\epsilon}_{Z}\big(X,B_{X},D_{X}\big)=\mathrm{sup}\left\{\lambda\in\mathbb{Q}\
\Big|\ \text{the~pair}\
 \Big(X, B_{X}+\lambda D_{X}\Big)\ \text{is $\epsilon$-log canonical along}~Z\right\}\in\mathbb{Q}\cup\big\{\pm\infty\big\}.%
$$
\end{definition}

The role played by the~number
$\mathrm{c}^{\epsilon}_{Z}(X,B_{X},D_{X})$ is very important for
$\epsilon=0$. So we put
$$
\mathrm{c}_{Z}\big(X,B_{X},D_{X}\big)=\mathrm{c}^{0}_{Z}\big(X,B_{X},D_{X}\big),
$$
and we put
$\mathrm{c}^{\epsilon}(X,B_{X},D_{X})=\mathrm{c}^{\epsilon}_{X}(X,B_{X},D_{X})$
and $\mathrm{c}(X,B_{X},D_{X})=\mathrm{c}^{0}_{X}(X,B_{X},D_{X})$.
Note that
$$
\mathrm{c}^{\epsilon}_Z\big(X,B_{X},D_{X}\big)=-\infty
$$
if the~log pair $(X,B_{X})$ is not $\epsilon$-log canonical along
$Z$ and $D_{X}=0$. If $B_{X}=0$, then we put
$$
\mathrm{c}^{\epsilon}_{Z}\big(X,D_{X}\big)=\mathrm{c}^{\epsilon}_Z\big(X,B_{X},D_{X}\big),%
$$
and we put
$\mathrm{c}^{\epsilon}(X,D_{X})=\mathrm{c}^{\epsilon}_X(X,D_{X})$,
$\mathrm{c}_{Z}(X,D_{X})=\mathrm{c}_Z(X,B_{X},D_{X})$,
$\mathrm{c}(X,D_{X})=\mathrm{c}_X(X,D_{X})$.

\begin{example}
\label{example:OMEGA} It follows from \cite{Ko97} that
$$
c_{O}\big(\Omega\big)=c_{O}\Big(\mathbb{C}^{n},\ a\big(\psi=0\big)-b\big(\phi=0\big)+c\mathcal{B}\Big),%
$$
and the~following conditions are equivalent:
\begin{itemize}
\item the~function $\Omega$ is square-integrable near $O\in\mathbb{C}^{n}$,%

\item the~singularities of the~log pair
$$
\left(\mathbb{C}^{n},\
a\big(\phi=0\big)-b\big(\phi=0\big)+\frac{c}{k}\sum_{i=1}^{k}\Big(\sum_{j=1}^{\gamma}\lambda_{j}^{i}\upsilon_{j}=0\Big)\right)
$$
are log terminal near $O\in\mathbb{C}^{n}$, where
$[\lambda_{1}^{i}:\cdots:\lambda_{\gamma}^{i}]$ is a~general point
in $\mathbb{P}^{\gamma-1}$, and $k$ is a sufficiently big natural number,%

\item the~singularities of the~log pair
$$
\Big(\mathbb{C}^{n},\ a\big(\phi=0\big)-b\big(\phi=0\big)+c\mathcal{B}\Big)%
$$
are log terminal near $O\in\mathbb{C}^{n}$, where $\mathcal{B}$ is
 a~linear system generated by
$\upsilon_{1}=0,\ldots,\upsilon_{r}=0$.
\end{itemize}
\end{example}

We say that the~log pair $(X, B_{X})$ has canonical singularities
(respectively, terminal sin\-gu\-la\-rities) if the~log pair $(X,
B_{X})$ is $1$-log canonical (respectively, $1$-log terminal).

\begin{remark}
\label{remark:1-log-canonical} Suppose that $B_{X}$ is effective
and $(X,B_{X})$ is canonical. Then $a_{1}=\cdots=a_{r}=0$.
\end{remark}

One can show that
Definition~\ref{definition:log-canonical-singularities} is
independent on the~choice of $\pi$. Put
$$
\mathrm{LCS}_{\epsilon}\Big(X, B_{X}\Big)=\left(\bigcup_{a_{i}\geqslant  1-\epsilon}B_{i}\right)\bigcup\left(\bigcup_{d_{i}\leqslant -1+\epsilon}\pi\big(E_{i}\big)\right)\subsetneq X,%
$$
put $\mathrm{LCS}(X, B_{X})=\mathrm{LCS}_{0}(X, B_{X})$ and
$\mathrm{CS}(X, B_{X})=\mathrm{LCS}_{1}(X, B_{X})$. We say that
the~subsets
$$
\mathrm{LCS}_{\epsilon}\Big(X, B_{X}\Big),\ \mathrm{LCS}\Big(X, B_{X}\Big),\ \mathrm{CS}\Big(X, B_{X}\Big)%
$$
are the~loci of $\epsilon$-log canonical, log canonical, canonical
singularities of $(X, B_{X})$, respectively.

\begin{definition}
\label{definition:log-canonical-center} A proper irreducible
subvariety $Y\subsetneq X$ is said to be a~center of
$\epsilon$-log canonical singularities of the~log pair $(X,
B_{X})$ if for some choice of the~resolution $\pi\colon\bar{X}\to
X$, one has
\begin{itemize}
\item either the~inequality $a_{i}\geqslant  1-\epsilon$ holds and $Y=B_{i}$ for some $i\in\{1,\ldots,r\}$,%
\item or the~inequality $d_{i}\leqslant -1+\epsilon$ holds and $Y=\pi(E_{i})$ for some $i\in\{1,\ldots,m\}$.%
\end{itemize}
\end{definition}

Let $\mathbb{LCS}_{\epsilon}(X, B_{X})$ be the~set of all centers
of $\epsilon$-log canonical singularities of $(X, B_{X})$. Then
$$
Y\in\mathbb{LCS}_{\epsilon}\Big(X, B_{X}\Big)\Longrightarrow Y\subseteq\mathrm{LCS}_{\epsilon}\Big(X, B_{X}\Big)%
$$
and $\mathbb{LCS}_{\epsilon}(X,
B_{X})=\varnothing\iff\mathrm{LCS}_{\epsilon}(X,
B_{X})=\varnothing$ $\iff$ the~log pair $(X,B_{X})$ is
$\epsilon$-log terminal.

\begin{remark}
\label{remark:hyperplane-reduction} Let $\mathcal{H}$ be a~linear
system on $X$ that has  no base points, let $H$ be a~sufficiently
general divisor in the~linear system $\mathcal{H}$, and let
$Y\subsetneq X$ be an irreducible subvariety. Put
$$
Y\cap H=\sum_{i=1}^{k}Z_{i},
$$
where $Z_{i}\subset H$ is an irreducible subvariety. It follows
from Definition~\ref{definition:log-canonical-center} (cf.
Theorem~\ref{theorem:adjunction})~that
$$
Y\in\mathbb{LCS}_{\epsilon}\Big(X, B_{X}\Big)\iff \Big\{Z_{1},\ldots,Z_{k}\Big\}\subseteq \mathbb{LCS}_{\epsilon}\Big(H, B_{X}\Big\vert_{H}\Big).%
$$
\end{remark}

We put $\mathbb{LCS}(X, B_{X})=\mathbb{LCS}_{0}(X, B_{X})$ and
$\mathbb{CS}(X, B_{X})=\mathbb{LCS}_{1}(X, B_{X})$. We say that
\begin{itemize}
\item elements in $\mathbb{LCS}(X, B_{X})$ are centers of log canonical singularities of the~log pair $(X, B_{X})$,%
\item elements in $\mathbb{CS}(X, B_{X})$ are centers of canonical singularities of the~log pair $(X, B_{X})$.%
\end{itemize}

\begin{example}
\label{example:smooth-point-and-log-pull-back} Let $\chi\colon
\hat{X}\to X$ be a~blow up of a~smooth point $P\in X$. Put
$$
B_{\hat{X}}=\sum_{i=1}^{r}a_{i}\hat{B}_{i}+\sum_{i=1}^{l}c_{i}\hat{\mathcal{M}}_{i},
$$
where $\hat{B}_{i}$ is a~proper transform of the~divisor $B_{i}$
on the~variety $\hat{X}$, and $\hat{\mathcal{M}}_{i}$ is a~proper
transform of the~linear system $\mathcal{M}_{i}$ on the~variety
$\hat{X}$. Then
$$
K_{\hat{X}}+B_{\hat{X}}\equiv\pi^{*}\Big(K_{X}+B_{X}\Big)+\Big(\mathrm{dim}\big(X\big)-1-\mathrm{mult}_{P}\big(B_{X}\big)\Big)E,
$$
where $E$ is the~exceptional divisor of $\chi$, and
$\mathrm{mult}_{P}(B_{X})\in\mathbb{Q}$. Then
\begin{itemize}
\item the~log pair $(X, B_{X})$ is $\epsilon$-log canonical near
$P\in X$ $\iff$ the~log pair
$$
\Big(\hat{X},\
B_{\hat{X}}+\Big(\mathrm{mult}_{P}\big(B_{X}\big)-\mathrm{dim}\big(X\big)+1\Big)E\Big)
$$
is $\epsilon$-log canonical in a~neighborhood of $E$, which
implies that
$$
\mathrm{mult}_{P}\big(B_{X}\big)\geqslant
\mathrm{dim}(X)-\epsilon\Longrightarrow
P\in\mathbb{LCS}_{\epsilon}\Big(X, B_{X}\Big),
$$%
\item if the~boundary $B_{X}$ is effective, then
$$
\mathrm{mult}_{P}\big(B_{X}\big)<1-\epsilon\Longrightarrow P\not\in\mathrm{LCS}_{\epsilon}\Big(X, B_{X}\Big),%
$$
\item if the~boundary $B_{X}$ is effective and mobile, then
$$
\mathrm{mult}_{P}\big(B_{X}\big)<1\Longrightarrow
P\not\in\mathrm{CS}\Big(X,B_{X}\Big),
$$%
\item if $\mathrm{dim}(X)=2$, and the~boundary $B_{X}$ is
effective, then
$$
P\in\mathbb{CS}\Big(X,B_{X}\Big)\iff \mathrm{mult}_{P}\big(B_{X}\big)\geqslant 1.%
$$
\end{itemize}
\end{example}

If the~boundary $B_{X}$ is effective, then the~locus
$\mathrm{LCS}(X,B_{X})\subset X$ can be naturally equipped with
a~subscheme structure  (see \cite{Sho93}). Indeed, if $B_{X}$ is
effective, jut put
$$
\mathcal{I}\Big(X, B_{X}\Big)=\pi_{*}\Bigg(\sum_{i=1}^{m}\lceil d_{i}\rceil E_{i}-\sum_{i=1}^{r}\lfloor a_{i}\rfloor \bar{B}_{i}\Bigg),%
$$
and let $\mathcal{L}(X, B_{X})$ be a~subscheme that corresponds to
the~ideal sheaf $\mathcal{I}(X, B_{X})$.

\begin{definition}
\label{definition:log-canonical-subscheme} If the~boundary $B_{X}$
is effective, then we say that
\begin{itemize}
\item the~subscheme $\mathcal{L}(X, B_{X})$ is the~subscheme of
log canonical singularities of
 $(X,B_{X})$,%
\item the~ideal sheaf $\mathcal{I}(X, B_{X})$ is the~multiplier ideal sheaf of the~log pair $(X,B_{X})$.%
\end{itemize}
\end{definition}

If the~boundary $B_{X}$ is effective, then it follows from
the~construction of $\mathcal{L}(X, B_{X})$ that
$$
\mathrm{Supp}\Big(\mathcal{L}\big(X,
B_{X}\big)\Big)=\mathrm{LCS}\Big(X, B_{X}\Big)\subset X.
$$

The following result is the~Nadel--Shokurov vanishing theorem (see
\cite{Sho93}).

\begin{theorem}
\label{theorem:Shokurov-vanishing} Let $H$ be a~nef and big
$\mathbb{Q}$-divisor on $X$ such that
$$
K_{X}+B_{X}+H\equiv D
$$
for some Cartier divisor $D$ on the~variety $X$. Suppose that
$B_{X}$ is effective. Then for $i\geqslant  1$
$$
H^{i}\Big(X,\ \mathcal{I}\big(X, B_{X}\big)\otimes D\Big)=0.
$$
\end{theorem}

\begin{proof}
See \cite[Theorem~9.4.8]{La04}.
\end{proof}

\begin{corollary}
\label{corollary:connectedness} Suppose that $B_{X}$ is effective.
Then the~locus
$$
\mathrm{LCS}\Big(X, B_{X}\Big)\subset X
$$
is connected if the~divisor $-(K_{X}+B_{X})$ is nef and big.
\end{corollary}

The assertion of Corollary~\ref{corollary:connectedness} is a~
special case of the~following result.

\begin{theorem}
\label{theorem:connectedness} Let $\psi\colon X\to Z$ be a~
morphism. Then the~set
$$
\mathrm{LCS}\Big(\bar{X},\ B^{\bar{X}}\Big)
$$
is connected in a~neighborhood of every fiber of the~morphism
$\psi\circ\pi\colon \bar{X}\to Z$ in the~case when
\begin{itemize}
\item the~boundary $B_{X}$ is effective,%
\item the~morphism $\psi$ is surjective and has connected fibers,%
\item the~divisor $-(K_{X}+B_{X})$ is nef and big with respect to $\psi$.%
\end{itemize}
\end{theorem}

\begin{proof}
See \cite[Lemma~5.7]{Sho93}.
\end{proof}

Applying Theorem~\ref{theorem:connectedness}, one can prove
the~following result.

\begin{theorem}
\label{theorem:adjunction} Suppose that the~boundary $B_{X}$ is
effective, the~divisor $K_{X}$ is $\mathbb{Q}$-Cartier, $a_{1}=1$,
and $B_{1}$ is a~Cartier divisor with log terminal singularities.
Then the~following are~equivalent:
\begin{itemize}
\item the~log pair $(X, B_{X})$ is log canonical in a~neighborhood of the~divisor $B_{1}$,%
\item the~singularities of the~log pair
$$
\Bigg(B_{1},\ \sum_{i=2}^{r}a_{i}B_{i}\Big\vert_{B_{1}}\Bigg)
$$
are log canonical.%
\end{itemize}
\end{theorem}

\begin{proof}
See \cite[Theorem~7.5]{Ko97}.
\end{proof}

The simplest application of Theorem~\ref{theorem:adjunction} is
already a~non-obvious result.

\begin{lemma}
\label{lemma:adjunction} Suppose that $\mathrm{dim}(X)=2$ and
$a_{1}\leqslant 1$. Then
$$
\Bigg(\sum_{i=2}^{r}a_{i}B_{i}\Bigg)\cdot B_{1}>1
$$
whenever $(X,B_{X})$ is not log canonical in a~point $P\in B_{1}$
such that $P\not\in\mathrm{Sing}(X)\cup\mathrm{Sing}(B_{1})$.
\end{lemma}

\begin{proof}
Suppose that  $(X,B_{X})$ is not log canonical in~a point $P\in
B_{1}$. By Theorem~\ref{theorem:adjunction}, we have
$$
\Bigg(\sum_{i=2}^{r}a_{i}B_{i}\Bigg)\cdot B_{1}\geqslant
\mathrm{mult}_{P}\Bigg(\sum_{i=2}^{r}a_{i}B_{i}\Big\vert_{B_{1}}\Bigg)>1
$$
if $P\not\in\mathrm{Sing}(X)\cup\mathrm{Sing}(B_{1})$, because
$(X, B_{1}+\sum_{i=2}^{r}a_{i}B_{i})$ is not log canonical at
the~point $P$.
\end{proof}

The assertion of Theorem~\ref{theorem:adjunction} can be slightly
refined.

\begin{definition}
\label{definition:plt} The log pair $(X, B_{X})$ has purely log
terminal singularities if
\begin{itemize}
\item the~inequality $a_{i}\leqslant 1$ holds for every $i\in\{1,\ldots,r\}$,%
\item the~inequality $d_{j}>-1$ for every $j\in\{1,\ldots,m\}$.%
\end{itemize}
\end{definition}

\begin{theorem}
\label{theorem:adjunction-plt} Suppose that $B_{X}$ is effective,
$a_{1}=1$, and $B_{1}$ is a~Cartier divisor. Then
\begin{itemize}
\item the~variety $B_{1}$ has log terminal singularities if $(X, B_{X})$ is purely log terminal,%
\item the~following assertions are equivalent:
\begin{itemize}
\item the~log pair $(X, B_{X})$ is purely log terminal in a~neighborhood of the~divisor $B_{1}$,%
\item the~variety $B_{1}$ has rational singularities, and the~log
pair
$$
\Bigg(B_{1},\ \sum_{i=2}^{r}a_{i}B_{i}\Big\vert_{B_{1}}\Bigg)
$$
is log terminal.%
\end{itemize}
\end{itemize}
\end{theorem}

\begin{proof}
See \cite[Theorem~7.5]{Ko97}.
\end{proof}

Suppose now that the~boundary $B_{X}$ is effective and mobile.
Thus, we have
$$
B_{X}=\sum_{i=1}^{l}c_{i}\mathcal{M}_{i},
$$
where $c_{i}\in\mathbb{Q}$ and $c_{i}\geqslant 0$, and
$\mathcal{M}_{i}$ is a~linear system on $X$ that has no fixed
components.

\begin{definition}
\label{definition:bir-equivalent-boundaries} We say that a~log
pair $(Y,B_{Y})$ is birationally equivalent to  $(X,B_{X})$ and if
\begin{itemize}
\item the~boundary $B_{Y}$ is effective and mobile,%
\item there is a~birational map $\xi\colon X\dasharrow Y$ such
that
$$
B_{Y}=\sum_{i=1}^{l}c_{i}\xi\big(\mathcal{M}_{i}\big),
$$
where $\xi(\mathcal{M}_{i})$ is a~proper transform of the~linear
system $\mathcal{M}_{i}$ on the~variety $Y$.
\end{itemize}
\end{definition}

Thus, the~log pairs $(\bar{X}, B_{\bar{X}})$ and $(X,B_{X})$ are
birationally equivalent.

\begin{definition}
\label{definition:Q-effective} Let $D$ be a~Weil
$\mathbb{Q}$-divisor on $X$. Then $D$ is $\mathbb{Q}$-effective if
$$
\exists\ n\in\mathbb{N}\ \text{such that}\ \big|nmD\big|\ne\varnothing,%
$$
where $m\in\mathbb{N}$ such that $mD$ is an integral Weil divisor.
\end{definition}

\begin{definition}
\label{definition:Kodaira-dimension} The Kodaira dimension of
the~log pair $(X, B_{X})$ is the~number
$$
\kappa\Big(X, B_{X}\Big)=\left\{\aligned
&\mathrm{sup}_{n\in\mathbb{N}}\Bigg(\mathrm{dim}\Big(\phi_{|nm(K_{\bar{X}}+B_{\bar{X}})|}\big(X\big)\Big)\Bigg)\ \text{if}\ K_{\bar{X}}+B_{\bar{X}}\ \text{is $\mathbb{Q}$-effective},\\
&-\infty\ \text{if}\ K_{\bar{X}}+B_{\bar{X}}\ \text{is not $\mathbb{Q}$-effective},\\
\endaligned
\right.
$$
where $m\in\mathbb{N}$ such that $m(K_{\bar{X}}+B_{\bar{X}})$ is
a~ Cartier divisor.
\end{definition}

The number $\kappa(X,B_{X})$ is independent on the~choice of $\pi$
(see \cite[Lemma~1.3.6]{Ch05umn}).

\begin{lemma}
\label{lemma:Kodaira-dimension-is-OK} Let $(Y,B_{Y})$ be a~log
pair that is birationally equivalent to $(X,B_{X})$. Then
$$
\kappa\Big(X, B_{X}\Big)=\kappa\Big(Y,B_{Y}\Big).
$$
\end{lemma}

\begin{proof}
See \cite[Lemma~1.3.6]{Ch05umn}.
\end{proof}

If the~log pair $(X,B_{X})$ is canonical, then it follows from
\cite[Lemma~1.3.6]{Ch05umn} that
$$
\kappa\Big(X, B_{X}\Big)=\left\{\aligned
&\mathrm{sup}_{n\in\mathbb{N}}\Bigg(\mathrm{dim}\Big(\phi_{|nm(K_{X}+B_{X})|}\big(X\big)\Big)\Bigg)\  \text{if}\ K_{X}+B_{X}\ \text{is $\mathbb{Q}$-effective},\\
&-\infty\  \text{if}\ K_{X}+B_{X}\ \text{is $\mathbb{Q}$-effective},\\
\endaligned
\right.
$$
where $m\in\mathbb{N}$ such that $m(K_{X}+B_{X})$ is a~Cartier
divisor.

\begin{corollary}
\label{corollary:Kodaira-dimension} If $(X,B_{X})$ is canonical
and $K_{X}+B_{X}$ is $\mathbb{Q}$-effective, then
$\kappa(X,B_{X})\geqslant 0$.
\end{corollary}

It follows from Definition~\ref{definition:Kodaira-dimension} that
$$
\kappa\Bigg(X,\ \sum_{i=1}^{l}c_{i}^{\prime}\mathcal{M}_{i}\Bigg)\geqslant\kappa\Big(X, B_{X}\Big)=\kappa\Bigg(X,\ \sum_{i=1}^{l}c_{i}\mathcal{M}_{i}\Bigg)%
$$
in the~case when $c_{i}^{\prime}\geqslant c_{i}$ for every
$i\in\{1,\ldots,l\}$

\begin{definition}
\label{definition:canonical-model} The log pair $(X, B_{X})$ is a~
canonical model if
\begin{itemize}
\item the~divisor $K_{X}+B_{X}$ is ample,%
\item the~log pair $(X, B_{X})$ has canonical singularities.
\end{itemize}
\end{definition}

It follows from Definition~\ref{definition:Kodaira-dimension} that
$\kappa(X,B_{X})=\mathrm{dim}(X)$ if $(X,B_{X})$ is canonical
model.

\begin{definition}
\label{definition:canonical-model-birational} A log pair $(Y,
B_{Y})$ is a~canonical model of the~log pair $(X, B_{X})$ if
\begin{itemize}
\item the~log pair $(Y, B_{Y})$ is a~canonical model,%
\item the~log pairs $(Y, B_{Y})$ and $(X, B_{X})$ are birationally
equivalent.
\end{itemize}
\end{definition}

The log pair $(X, B_{X})$ does not have a~canonical model if
$\kappa(X,B_{X})<\mathrm{dim}(X)$.

\begin{theorem}
\label{theorem:canonical-model-is-unique} A canonical model is
unique whenever it exists.
\end{theorem}

\begin{proof}
See \cite[Theorem~1.3.20]{Ch05umn}.
\end{proof}

It follows from \cite{BCHM06} that the~following conditions are
equivalent:
\begin{itemize}
\item the~log pair $(X, B_{X})$ has canonical model,%
\item the~equality $\kappa(X,B_{X})=\mathrm{dim}(X)$ holds.
\end{itemize}

\section{Inequality I}
\label{section:inequality-2}

Let $X$ be a~surface, let $O$ be a~smooth point of $X$, let
$\Delta_{1}$ and $\Delta_{2}$ be curves on $X$~such~that
$$
O\in\Delta_{1}\cap\Delta_{2},
$$
both $\Delta_{1}$ and $\Delta_{2}$ are irreducible and reduced,
both $\Delta_{1}$ and $\Delta_{2}$ are smooth at $O$, and
$\Delta_{1}$ intersects the~curve $\Delta_{2}$ transversally at
$O$, let $D$ be an effective $\mathbb{Q}$-divisor on
the~surface~$X$~such~that
$$
\Delta_{1}\not\subseteq\mathrm{Supp}\big(D\big)\not\supseteq\Delta_{2},
$$
and let $a_{1}$ and $a_{2}$ be non-negative rational~numbers.
Suppose that the~log pair
$$
\Big(X,\ D+a_{1}\Delta_{1}+a_{2}\Delta_{2}\Big)
$$
is not log canonical at $O$. Let $A,B,M,N,\alpha,\beta$ be
non-negative rational numbers such that
\begin{itemize}
\item the~inequality $\alpha a_{1}+\beta a_{2}\leqslant 1$ holds,%
\item the~inequalities $A(B-1)\geqslant 1\geqslant\mathrm{max}(M,N)$ hold,%
\item the~inequalities $\alpha(A+M-1)\geqslant A^{2}(B+D-1)\beta$ and $\alpha(1-M)+A\beta\geqslant A$ holds,%
\item either the~inequality $2M+AN\leqslant 2$ holds or
$$
\alpha\big(B+1-MB-N\big)+\beta\big(A+1-AN-M\big)\geqslant AB-1.
$$%
\end{itemize}

\begin{lemma}
\label{lemma:2-0} The inequalities $A+M\geqslant 1$ and $B>1$
holds. The inequality
$$
\alpha\big(B+1-MB-N\big)+\beta\big(A+1-AN-M\big)\geqslant AB-1
$$
holds. The inequality $\beta(1-N)+B\alpha\geqslant B$ holds. The
inequalities
$$
\frac{\alpha(2-M)}{A+1}+\frac{\beta(2-N)}{B+1}\geqslant 1%
$$
and $\alpha(2-M)B+\beta(1-N)(A+1)\geqslant B(A+1)$ hold.
\end{lemma}

\begin{proof}
The inequality $B>1$ follows from the~inequality $A(B-1)\geqslant
1$. Then
$$
\frac{\alpha}{A+1}+\frac{\beta}{B+1}\geqslant\frac{\alpha}{A+1}+\frac{\beta}{2B}\geqslant\frac{1}{2}
$$
because $2B\geqslant B+1$. Similarly, we see that $A+M\geqslant
1$, because
$$
\frac{\alpha(A+M-1)}{A^{2}(B+D-1)}\geqslant\beta\geqslant 0
$$
and $B+D-1\geqslant 0$. The inequality
$\beta(1-N)+B\alpha\geqslant B$ follows from the~inequalities
$$
\alpha+\frac{\beta(1-N)}{B}\geqslant\frac{2-M}{A+1}\alpha+\frac{\beta(1-N)}{B}\geqslant 1,%
$$
because $A+1\geqslant 2-M$.

Let us show that the~inequality
$$
\alpha\big(2-M\big)B+\beta\big(1-N\big)\big(A+1\big)\geqslant B\big(A+1\big)%
$$
holds. Let $L_{1}$ be the~line in $\mathbb{R}^{2}$ given by the~
equation
$$
x\big(2-M\big)B+y\big(1-N\big)(A+1)-B(A+1)=0
$$
and let $L_{2}$ be the~line that is given by the~equation
$$
x\big(1-N\big)+Ay-A=0,
$$
where $(x,y)$ are coordinates on $\mathbb{R}^{2}$. Then $L_{1}$
intersects the~line $y=0$ in a~point
$$
\Bigg(\frac{A+1}{2-M},0\Bigg)
$$
and $L_{2}$ intersects the~line $y=0$ in a~point $(A/(1-M),0)$.
But
$$
\frac{A+1}{2-M}<\frac{A}{1-M},
$$
which implies that $\alpha(2-M)B+\beta(1-N)(A+1)\geqslant B(A+1)$
if and only if
$$
A^{2}\beta_{0}\big(B+N-1\big)\geqslant \alpha_{0}\big(A+M-1\big),
$$
where $(\alpha_{0},\beta_{0})$ is the~intersection point of the~
lines $L_{1}$ and $L_{2}$. But
$$
\big(\alpha_{0},\beta_{0}\big)=\Bigg(\frac{A(A+1)(B+N-1)}{\Delta},\
\frac{B(A-1+M)}{\Delta}\Bigg),
$$
where $\Delta=2AB-ABM-A+AM-1+M+NA-NAM+N-NM$. But
$$
A^{2}\Big(B\big(A-1+M\big)\Big)\big(B+N-1\big)\geqslant\Big(A\big(A+1\big)\big(B+N-1\big)\Big)\big(A+M-1\big),%
$$
because $A(B-1)\geqslant 1$, which implies that
$A^{2}\beta_{0}(B+N-1)\geqslant \alpha_{0}(A+M-1)$.

Finally, let us show that that the~inequality
$$
\alpha\big(B+1-MB-N\big)+\beta\big(A+1-AN-M\big)\geqslant AB-1
$$
holds. Let $L^{\prime}_{1}$ be the~line in $\mathbb{R}^{2}$ given
by the~equation
$$
x\big(B+1-MB-N\big)+y\beta\big(A+1-AN-M\big)-AB+1=0
$$
and let $L_{2}$ be the~line that is given by the~equation
$$
x\big(1-N\big)+Ay-A=0,
$$
where $(x,y)$ are coordinates on $\mathbb{R}^{2}$. Then
$L^{\prime}_{1}$ intersects the~line $y=0$ in a~point
$$
\Bigg(\frac{AB-1}{B+1-MB-N},0\Bigg)
$$
and $L_{2}$ intersects the~line $y=0$ in a~point $(A/(1-M),0)$.
But
$$
\frac{AB-1}{B+1-MB-N}<\frac{A}{1-M},
$$
which implies that $\alpha(B+1-MB-N)+\beta(A+1-AN-M)\geqslant
AB-1$ if and only if
$$
A^{2}\beta_{1}\big(B+N-1\big)\geqslant \alpha_{1}\big(A+M-1\big),
$$
where $(\alpha_{1},\beta_{1})$ is the~intersection point of the~
lines $L^{\prime}_{1}$ and $L_{2}$. Note that
$$
\big(\alpha_{1},\beta_{1}\big)=\Bigg(\frac{A(AB-A-2+NA+M)}{\Delta^{\prime}},\ \frac{A+1-NA-M}{\Delta^{\prime}}\Bigg),%
$$
where $\Delta^{\prime}=AB-1-ABM+AM+2M-NAM-M^2$.

To complete the~proof, we must show that the~inequality
$$
A^{2}\Big(A+1-NA-M\Big)(B+N-1)\geqslant \Big(A(AB-A-2+NA+M)\Big)(A+M-1)%
$$
holds. This inequality is equivalent to the~inequality
$$
\big(2-M\big)\big(A+M-1\big)\geqslant A\big(AN+2M-2)\big(B+N-1\big),%
$$
which is true, because $M\leqslant 1$ and $AN+2M-2\leqslant 0$.
\end{proof}

The purpose of this section is to prove Theorem~\ref{theorem:I}.
Suppose that the~inequalities
$$
\mathrm{mult}_{O}\Big(D\cdot\Delta_{1}\Big)\leqslant M+Aa_{1}-a_{2}\ \text{and}\ \mathrm{mult}_{O}\Big(D\cdot\Delta_{2}\Big)\leqslant N+Ba_{2}-a_{1}%
$$
hold. Let us show that this assumptions lead to a~contradiction.

\begin{lemma}
\label{lemma:2-1} The inequalities $a_{1}>(1-M)/A$ and
$a_{2}>(1-N)/B$ hold.
\end{lemma}

\begin{proof}
If $a_{1}>1$, then $a_{1}>(1-M)/A$ by Lemma~\ref{lemma:2-0}.
Suppose that $a_{1}\leqslant 1$. Then the~log pair
$$
\Big(X,\ D+\Delta_{1}+a_{2}\Delta_{2}\Big)
$$
is not log canonical~at~the~point $O$.~Then it follows from
Lemma~\ref{lemma:adjunction} that
$$
M+Aa_{1}-a_{2}\geqslant\mathrm{mult}_{O}\Big(D\cdot\Delta_{1}\Big)>1-a_{2},%
$$
which implies that $a_{1}>(1-M)/A$. Similarly, we see that
$a_{2}>(1-N)/B$.
\end{proof}

\begin{lemma}
\label{lemma:2-2} The inequalities $a_{1}<1$ and $a_{2}<1$ hold.
\end{lemma}

\begin{proof}
The inequalities $a_{1}>(1-M)/A$ and $a_{2}>(1-N)/B$ hold by
Lemma~\ref{lemma:2-1}. But
$$
\alpha a_{1}+\beta a_{2}\leqslant 1,
$$
which implies that $a_{1}\alpha<1-\beta(1-N)/B$ and
$a_{2}\beta<1-\alpha(1-M)/A$. Then
$$
a_{1}<1>a_{2},
$$
because $\beta(1-N)+B\alpha\geqslant B$ by Lemma~\ref{lemma:2-0}
and $\alpha(1-M)+A\beta\geqslant A$ by assumption.
\end{proof}

Put $m_{0}=\mathrm{mult}_{O}(D)$. Then $m_{0}$ is a~positive
rational number.

\begin{lemma}
\label{lemma:2-21} The inequalities $m_{0}\leqslant
M+Aa_{1}-a_{2}$ and $m_{0}\leqslant N+Ba_{2}-a_{1}$ hold.
\end{lemma}

\begin{proof}
We have
$$
m_{0}\leqslant \mathrm{mult}_{O}\Big(D\cdot\Delta_{1}\Big)\leqslant M+Aa_{1}-a_{2},%
$$
which implies that $m_{0}\leqslant M+Aa_{1}-a_{2}$. Similarly, we
see that $m_{0}\leqslant N+Ba_{2}-a_{1}$.
\end{proof}

\begin{lemma}
\label{lemma:2-3} The inequality $m_{0}+a_{1}+a_{2}\leqslant 2$
holds.
\end{lemma}

\begin{proof}
We know that $m_{0}+a_{1}+a_{2}\leqslant M+(A+1)a_{1}$ and
$m_{0}+a_{1}+a_{2}\leqslant N+(B+1)a_{2}$. Then
$$
\big(m_{0}+a_{1}+a_{2}\big)\Bigg(\frac{\alpha}{A+1}+\frac{\beta}{B+1}\Bigg)\leqslant
\alpha a_{1}+\beta a_{2}+\frac{\alpha M}{A+1}+\frac{\beta N}{B+1}\leqslant 1+\frac{\alpha M}{A+1}+\frac{\beta N}{B+1},%
$$
which implies that $m_{0}+a_{1}+a_{2}\leqslant 2$, because
$$
\frac{\alpha(2-M)}{A+1}+\frac{\beta(2-N)}{B+1}\geqslant 1
$$
by Lemma~\ref{lemma:2-0}.
\end{proof}

Let $\pi_1\colon X_{1}\to X$ be the~blow up of the~point $O$, and
let $F_{1}$ be the~$\pi_{1}$-exceptional curve. Then
$$
K_{X_{1}}+D^{1}+a_{1}\Delta^{1}_{1}+a_{2}\Delta^{1}_{2}+\big(m_{0}+a_{1}+a_{2}-1\big)F_{1}\equiv\pi_{1}^{*}\Big(K_{X}+D+a_{1}\Delta_{1}+a_{2}\Delta_{2}\Big),
$$
where $D^{1}$, $\Delta^{1}_{1}$, $\Delta^{1}_{2}$ are proper
transforms of $D$, $\Delta_{1}$, $\Delta_{2}$ on the~surface
$X_{1}$, respectively.~The~pair
$$
\Big(X_{1},\ D^{1}+a_{1}\Delta^{1}_{1}+a_{2}\Delta^{1}_{2}+\big(m_{0}+a_{1}+a_{2}-1\big)F_{1}\Big)%
$$
is not log canonical at some point $O_{1}\in F_{1}$. Note that
$m_{0}+a_{1}+a_{2}-1\geqslant 0$.

\begin{lemma}
\label{lemma:2-4} Either $O_{1}=F_{1}\cap\Delta^{1}_{1}$ or
$O_{1}=F_{1}\cap\Delta^{1}_{2}$.
\end{lemma}

\begin{proof}
Suppose that $O_{1}\not \in\Delta^{1}_{1}\cup\Delta^{1}_{2}$. Then
the~log pair
$$
\Big(X_{1},\ D^{1}+\big(m_{0}+a_{1}+a_{2}-1\big)F_{1}\Big)%
$$
is not log canonical at the~point $O_{1}$. Then
$$
m_{0}=D^{1}\cdot F_{1}>1
$$
by Lemma~\ref{lemma:adjunction}, because
$m_{0}+a_{1}+a_{2}\leqslant 2$ by Lemma~\ref{lemma:2-3}. Then
$$
m_{0}\Bigg(\frac{\beta+B\alpha}{AB-1}+\frac{\alpha+A\beta}{AB-1}\Bigg)\leqslant\big(M+Aa_{1}-a_{2}\big)\frac{\beta+B\alpha}{AB-1}+\big(N+Ba_{2}-a_{1}\big)\frac{\alpha+A\beta}{AB-1},%
$$
because $m_{0}\leqslant M+Aa_{1}-a_{2}$ and $m_{0}\leqslant
N+Ba_{2}-a_{1}$. On the~other hand, we have
$$
\big(M+Aa_{1}-a_{2}\big)\frac{\beta+B\alpha}{AB-1}+\big(N+Ba_{2}-a_{1}\big)\frac{\alpha+A\beta}{AB-1}\leqslant 1+\frac{M\beta+MB\alpha+N\alpha+AN\beta}{AB-1},%
$$
because $\alpha a_{1}+\beta a_{2}\leqslant 1$ and $AB-1>0$. But we
already proved that $m_{0}>1$. Thus, we see that
$$
\beta+B\alpha+\alpha+A\beta\leqslant AB-1+M\beta+MB\alpha+N\alpha+AN\beta,%
$$
which is impossible by Lemma~\ref{lemma:2-0}.
\end{proof}

\begin{lemma}
\label{lemma:2-5} The inequality $O_{1}\ne
F_{1}\cap\Delta^{1}_{1}$ holds.
\end{lemma}

\begin{proof}
Suppose that $O_{1}\ne F_{1}\cap\Delta^{1}_{1}$. Then the~log pair
$$
\Big(X_{1},\ D^{1}+a_{1}\Delta^{1}_{1}+\big(m_{0}+a_{1}+a_{2}-1\big)F_{1}\Big)%
$$
is not log canonical at the~point $O_{1}$. Then
$$
M+Aa_{1}-a_{2}-m_{0}=D^{1}\cdot\Delta^{1}_{1}>1-\big(m_{0}+a_{1}+a_{2}-1\big)
$$
by Lemma~\ref{lemma:adjunction}, because $a_{1}<1$ by
Lemma~\ref{lemma:2-2}. We have $a_{1}>(2-M)/(A+1)$. Then
$$
\frac{2-M\alpha}{A+1}+\frac{\beta(1-N)}{B}<\alpha a_{1}+\beta a_{2}\leqslant 1,%
$$
because $a_{2}>(1-N)/B$ by Lemma~\ref{lemma:2-2}. Thus, we see
that
$$
\frac{2-M\alpha}{A+1}+\frac{\beta(1-N)}{B}<1,
$$
which is impossible by Lemma~\ref{lemma:2-0}.
\end{proof}

Therefore, we see that $O_{1}=F_{1}\cap\Delta^{1}_{2}$.  Then the~
log pair
$$
\Big(X_{1},\ D^{1}+a_{1}\Delta^{1}_{1}+a_{2}\Delta^{1}_{2}+\big(m_{0}+a_{1}+a_{2}-1\big)F_{1}\Big)%
$$
is not log canonical at the~point $O_{1}$. We know that
$1\geqslant m_{0}+a_{1}+a_{2}-1\geqslant 0$.

We have a~blow up $\pi_{1}\colon X_{1}\to X$. For any
$n\in\mathbb{N}$, consider a~sequence of blow ups
$$
\xymatrix{
&X_{n}\ar@{->}[rr]^{\pi_{n}}&&X_{n-1}\ar@{->}[rr]^{\pi_{n-1}}&&\cdots\ar@{->}[rr]^{\pi_{3}}
&&X_{2}\ar@{->}[rr]^{\pi_{2}}&&X_{1}\ar@{->}[rr]^{\pi_{1}}&& X}
$$
such that $\pi_{i+1}\colon X_{i+1}\to X_{i}$ is a~blow up of the~
point $F_{i}\cap\Delta^{i}_{2}$ for every $i\in\{1,\ldots,n-1\}$,
where
\begin{itemize}
\item we denote by $F_{i}$ the~exceptional curve of the~morphism $\pi_{i}$,%
\item we denote by $\Delta^{i}_{2}$ the~proper transform of the~curve $\Delta_{2}$ on the~surface $X_{i}$.%
\end{itemize}

For every $k\in\{1,\ldots,n\}$ and for every $i\in\{1,\ldots,k\}$,
let $D^{k}$, $\Delta^{k}_{1}$ and $F^{k}_{i}$ be
the~proper~transforms on the~surface $X_{k}$ of the~divisors $D$,
$\Delta_{1}$ and $F_{i}$, respectively.~Then
$$
K_{X_{n}}+D^{n}+a_{1}\Delta^{n}_{1}+a_{2}\Delta^{n}_{2}+\sum_{i=1}^{n}\Bigg(a_{1}+ja_{2}-j+\sum_{j=0}^{n-1}m_{j}\Bigg)F_{i}\equiv\pi^{*}\Big(K_{X}+D+a_{1}\Delta_{1}+a_{2}\Delta_{2}\Big),
$$
where $\pi=\pi_{n}\circ\cdots\circ\pi_{2}\circ\pi_{1}$ and
$m_{i}=\mathrm{mult}_{O_{i}}(D^{i})$ for every
$i\in\{1,\ldots,n\}$. Then the~log pair
\begin{equation}
\label{equation:log-pair} \Bigg(X_{n},\
D^{n}+a_{1}\Delta^{n}_{1}+a_{2}\Delta^{n}_{2}+\sum_{i=1}^{n}\Bigg(a_{1}+ia_{2}-i+\sum_{j=0}^{i-1}m_{j}\Bigg)F^{n}_{i}\Bigg)
\end{equation}
is not log canonical at some point of the~set $F^{n}_{1}\cup
F^{n}_{2}\cup\cdots\cup F^{n}_{n}$.

For every $k\in\{1,\ldots,n\}$, let us put
$O_{k}=F_{k}\cap\Delta^{k}_{2}$

\begin{lemma}
\label{lemma:2-inductive} For every $i\in\{1,\ldots,n\}$, the~
inequalities
$$
1\geqslant a_{1}+ia_{2}-i+\sum_{j=0}^{i-1}m_{j}\geqslant 0
$$
hold, the~log pair~\ref{equation:log-pair} is log canonical at
every point of the~set $(F^{n}_{1}\cup F^{n}_{2}\cup\cdots\cup
F^{n}_{n})\setminus O_{n}$.
\end{lemma}

It follows from Lemma~\ref{lemma:2-inductive} that there is
$n\in\mathbb{N}$ such that
$$
a_{1}+na_{2}-n+\sum_{j=0}^{n-1}m_{j}>1,
$$
which contradicts Lemma~\ref{lemma:2-inductive}. Therefore, to
complete the~proof of Theorem~\ref{theorem:I}, it is enough to
prove~Lemma~\ref{lemma:2-inductive}. Let us prove
Lemma~\ref{lemma:2-inductive} by induction on $n\in\mathbb{N}$.
The case $n=1$ is done.

By induction, we may assume that~$n\geqslant 2$. For every
$k\in\{1,\ldots,n-1\}$, we may assume that
$$
1\geqslant a_{1}+ka_{2}-k+\sum_{j=0}^{k-1}m_{j}\geqslant 0,
$$
the~singularities of the~log pair
$$
\Bigg(X_{k},\
D^{k}+a_{1}\Delta^{k}_{1}+a_{2}\Delta^{k}_{2}+\sum_{i=1}^{k}\Bigg(a_{1}+ka_{2}-k+\sum_{j=0}^{i-1}m_{j}\Bigg)F^{k}_{i}\Bigg)
$$
are log canonical at every point of the~set $(F^{k}_{1}\cup
F^{k}_{2}\cup\cdots\cup F^{k}_{k})\setminus O_{k}$ and not log
canonical at $O_{k}$.

\begin{lemma}
\label{lemma:2-6} The following inequality holds:
$$
a_{2}>\frac{n-N}{B+n-1}.
$$
\end{lemma}

\begin{proof}
The singularities of the~log pair
$$
\Bigg(X_{n-1},\
D^{n-1}+a_{2}\Delta^{k}_{2}+\Bigg(a_{1}+\big(n-1\big)a_{2}-\big(n-1\big)+\sum_{j=0}^{n-2}m_{j}\Bigg)F^{n}_{n-1}\Bigg)
$$
are not log canonical at the~point $O_{n-1}$. Then
$$
N-Ba_{2}-a_{1}-\sum_{j=0}^{n-2}m_{j}=D^{n-1}\cdot\Delta^{n-1}_{2}>1-\Bigg(a_{1}+\big(n-1\big)a_{2}-\big(n-1\big)+\sum_{j=0}^{n-2}m_{j}\Bigg)
$$
by Lemma~\ref{lemma:adjunction}, because $a_{2}<1$ by
Lemma~\ref{lemma:2-2}. We have
$$
a_{2}>\frac{n-N}{B+n-1},
$$
which completes the~proof.
\end{proof}

\begin{lemma}
\label{lemma:2-7} The inequalities $1\geqslant
a_{1}+na_{2}-n+\sum_{j=0}^{n-1}m_{j}\geqslant 0$ hold.
\end{lemma}

\begin{proof}
The inequality $a_{1}+na_{2}-n+\sum_{j=0}^{n-1}m_{j}\geqslant 0$
follows from the~fact that the~log pair
$$
\Bigg(X_{n-1},\
D^{n-1}+a_{2}\Delta^{k}_{2}+\Bigg(a_{1}+\big(n-1\big)a_{2}-\big(n-1\big)+\sum_{j=0}^{n-2}m_{j}\Bigg)F^{n}_{n-1}\Bigg)
$$
is not log canonical at the~point $O_{n-1}$. Suppose that
$a_{1}+na_{2}-n+\sum_{j=0}^{n-1}m_{j}>1$.

We have $m_{0}+a_{2}\leqslant M+Aa_{1}$ by Lemma~\ref{lemma:2-21}.
Then
$$
a_{1}+nM+nAa_{1}-n\geqslant a_{1}+na_{2}-n+nm_{0}\geqslant a_{1}+na_{2}-n+\sum_{j=0}^{n-1}m_{j}>1,%
$$
which immediately implies that the~inequality
$$
a_{1}>\frac{n+1-Mn}{nA+1}
$$
holds. On the~other hand, we know that the~inequality
$$
a_{2}>\frac{n-N}{B+n-1}
$$
holds by Lemma~\ref{lemma:2-6}. Therefore, we see that
$$
\Bigg(\frac{\alpha-M}{A}+\beta\Bigg)+\alpha\frac{A-1+M}{A(An+1)}+\beta\frac{1-B-N}{B+n-1}=\alpha\frac{n+1-Mn}{nA+1}+\beta\frac{n-N}{B+n-1}<\alpha a_{1}+\beta a_{2}\leqslant 1,%
$$
where $\alpha(1-M)/A+\beta\geqslant 1$. Therefore, we see that
$$
\alpha\frac{A+M-1}{A(An+1)}<\beta\frac{B+N-1}{B+n-1},
$$
where $n\geqslant 2$. But $A+M>1$ and $B>1$ by
Lemmas~\ref{lemma:2-1} and \ref{lemma:2-2}. Thus, we see that
$$
\frac{A(An+1)}{\alpha(A+M-1)}>\frac{B+n-1}{\beta(B+N-1)},
$$
but $A^{2}(B+N-1)\beta\leqslant\alpha(A+M-1)$ by assumption. Then
$$
\frac{A}{\alpha(A+M-1)}-\frac{B-1}{\beta(B+N-1)}\geqslant\Bigg(\frac{A^{2}}{\alpha(A+M-1)}-\frac{1}{\beta(B+M-1)}\Bigg)n+\frac{A}{\alpha(A+M-1)}-\frac{B-1}{\beta(B+N-1)}>0,
$$
which implies that $\beta A(B+N-1)>\alpha(B-1)(A+M-1)$. Then
$$
\frac{\alpha(A+M-1)}{A}\geqslant\beta A\big(B+N-1\big)>\alpha\big(B-1\big)\big(A+M-1\big),%
$$
because $A^{2}(B+N-1)\beta\leqslant\alpha(A+M-1)$ by assumption.
Then $\alpha\ne 0$ and
$$
A\big(B-1\big)<1,
$$
which is impossible, because $A(B-1)\geqslant 1$ by assumption.
\end{proof}

\begin{lemma}
\label{lemma:2-8} The log pair~\ref{equation:log-pair} is log
canonical in every point of the~set
$$
F_{n}\setminus\Bigg(\Big(F_{n}\cap
F^{n}_{n-1}\Big)\bigcup\Big(F_{n}\cap\Delta^{n}_{2}\Big)\Bigg).
$$
\end{lemma}

\begin{proof}
Suppose that there is a~point $Q\in F_{n}$ such that
$$
F_{n}\cap F^{n}_{n-1}\ne Q\ne F_{n}\cap\Delta^{n}_{2},
$$
but the~log pair~\ref{equation:log-pair} is not log canonical at
$Q$. Then the~log pair
$$
\Bigg(X_{n},\
D^{n}+\Bigg(a_{1}+na_{2}-n+\sum_{j=0}^{n-1}m_{j}\Bigg)F_{n}\Bigg)%
$$
is not log canonical at the~point $Q$ as well.  Then
$$
m_{0}\geqslant m_{n-1}=D^{n}\cdot F_{n}>1
$$
by Lemma~\ref{lemma:adjunction}, because
$a_{1}+na_{2}-n+\sum_{j=0}^{n-1}m_{j}\leqslant 1$ by
Lemma~\ref{lemma:2-7}. Then
$$
m_{0}\Bigg(\frac{\beta+B\alpha}{AB-1}+\frac{\alpha+A\beta}{AB-1}\Bigg)\leqslant\big(M+Aa_{1}-a_{2}\big)\frac{\beta+B\alpha}{AB-1}+\big(N+Ba_{2}-a_{1}\big)\frac{\alpha+A\beta}{AB-1},%
$$
because $m_{0}\leqslant M+Aa_{1}-a_{2}$ and $m_{0}\leqslant
N+Ba_{2}-a_{1}$ by Lemma~\ref{lemma:2-21}. We have
$$
\big(M+Aa_{1}-a_{2}\big)\frac{\beta+B\alpha}{AB-1}+\big(N+Ba_{2}-a_{1}\big)\frac{\alpha+A\beta}{AB-1}\leqslant 1+\frac{M\beta+MB\alpha+N\alpha+AN\beta}{AB-1},%
$$
because $\alpha a_{1}+\beta a_{2}\leqslant 1$ and $AB-1>0$. But
$m_{0}>1$. Thus, we see that
$$
\beta+B\alpha+\alpha+A\beta\leqslant AB-1+M\beta+MB\alpha+N\alpha+AN\beta,%
$$
which contradicts our initial assumptions.
\end{proof}

\begin{lemma}
\label{lemma:2-9} The log pair~\ref{equation:log-pair} is log
canonical at the~point $F_{n}\cap F^{n}_{n-1}$.
\end{lemma}

\begin{proof}
Suppose that the~log pair~\ref{equation:log-pair} is not log
canonical at  $F_{n}\cap F^{n}_{n-1}$. Then the~log pair
$$
\Bigg(X_{n},\
D^{n}+\Bigg(a_{1}+\big(n-1\big)a_{2}-\big(n-1\big)+\sum_{j=0}^{n-2}m_{j}\Bigg)F^{n}_{n-1}+\Bigg(a_{1}+na_{2}-n+\sum_{j=0}^{n-1}m_{j}\Bigg)F_{n}\Bigg)%
$$
is not log canonical at the~point $F_{n}\cap F^{n}_{n-1}$ as well.
Then
$$
m_{n-2}-m_{n-1}=D^{n}\cdot
F_{n-2}>1-\Bigg(a_{1}+na_{2}-n+\sum_{j=0}^{n-1}m_{j}\Bigg)
$$
by Lemma~\ref{lemma:adjunction}, because
$a_{1}+(n-1)a_{2}-(n-1)+\sum_{j=0}^{n-2}m_{j}$. Note that
$$
M+Aa_{1}-a_{2}-m_{0}\geqslant\mathrm{mult}_{O}\Big(D\cdot\Delta_{1}\Big)-m_{0}\geqslant D\cdot\Delta_{1}-m_{0}=D^{1}\cdot\Delta^{1}_{1}\geqslant 0,%
$$
which implies that $m_{0}+a_{2}\leqslant Aa_{1}+M$. Then
$$
nM+nAa_{1}-na_{2}\geqslant nm_{0}\geqslant\big(n+1\big)m_{0}-m_{n-1}\geqslant m_{n-2}-m_{n-1}+\sum_{j=0}^{n-1}m_{j}>n+1-a_{1}-na_{2},%
$$
which gives $a_{1}>(n+1-nM)/(An+1)$. The proof of
Lemma~\ref{lemma:2-7} implies a~contradiction.
\end{proof}

The assertion of Lemma~\ref{lemma:2-inductive} is proved. The
assertion of Theorem~\ref{theorem:I} is proved.

\section{Inequality II}
\label{section:inequality-1}

Let $X$ be a~surface, let $O$ be a~smooth point of $X$, let
$\mathcal{M}$ be a~linear system on  $X$ that does not have fixed
components, let $\Delta_{1}$ be an irreducible and reduced curve
on $X$~such~that
$$
O\in\Delta_{1}\setminus\mathrm{Sing}\big(\Delta_{1}\big),
$$
let $\epsilon$ and $a_{1}$ be rational numbers such that
$\epsilon>0$ and $a_{1}\leqslant 0$. Suppose that the~log pair
$$
\Big(X,\ \epsilon\mathcal{M}+a_{1}\Delta_{1}\Big)
$$
is not terminal at the~point $O$. To prove
Theorem~\ref{theorem:II}, we must prove that
\begin{itemize}
\item the~inequality~\ref{equation:Cheltsov} holds, which means
$$
\mathrm{mult}_{O}\Big(M_{1}\cdot M_{2}\Big)\geqslant
\left\{\aligned
&\frac{1-2a_{1}}{\epsilon^{2}}\ \text{if}\ a_{1}\geqslant -1\slash 2,\\
&\frac{-4a_{1}}{\epsilon^{2}}\ \text {if}\ a_{1}\leqslant -1\slash 2,\\
\endaligned
\right.
$$
where $M_{1}$ and $M_{2}$ are general curves in $\mathcal{M}$,

\item if the~inequality~\ref{equation:Cheltsov} is, actually, an
equality, then
\begin{itemize}
\item either $-a_{1}\in\mathbb{N}$ and $\mathrm{mult}_{O}(\mathcal{M})=2/\epsilon$,%
\item or $a_{1}=0$ and $\mathrm{mult}_{O}(\mathcal{M})=1/\epsilon$.%
\end{itemize}
\end{itemize}

There is a~birational morphism $\pi\colon\bar{X}\to X$ such that
\begin{itemize}
\item the~morphism $\pi$ is a~sequence of blow ups of smooth points,%

\item the~morphism contracts $m$ irreducible curves $E_{1},E_{2},\ldots,E_{m}$ to the~point $O$,%

\item the~morphism $\pi$ induces an isomorphism
$$
\bar{X}\setminus \bigcup_{i=1}^{m}E_{i}\cong X\setminus O,
$$

\item for some rational numbers $d_{1},d_{2},\ldots,d_{m}$, we
have
$$
K_{\bar{X}}+\epsilon\bar{\mathcal{M}}+a_{1}\bar{\Delta}_{1}\equiv\pi^{*}\Big(K_{X}+\epsilon\mathcal{M}+a_{1}\Delta_{1}\Big)+\sum_{i=1}^{m}d_{i}E_{i},
$$
where $\bar{\mathcal{M}}$ and $\bar{\Delta}_{1}$ are proper transforms of $\mathcal{M}$ and $\Delta_{1}$ on the~surface $\bar{X}$, respectively,%

\item the~inequality $d_{m}\leqslant 0$ holds and
\begin{itemize}
\item either the~equality $m=1$ holds,%
\item or $m\geqslant 2$ and $d_{i}>0$ for every
$i\in\{1,\ldots,m-1\}$.
\end{itemize}
\end{itemize}

\begin{lemma}
\label{lemma:a1-negative} Suppose that $a_{1}\geqslant 0$. Then
$m=1$.
\end{lemma}

\begin{proof}
The required assertion is well-known and easy to prove. So we omit
the~proof.
\end{proof}

Let $\chi\colon \hat{X}\to X$ be a~blow up of the~point $O$, let
$E$ be the~exceptional curve of $\chi$. Then
$$
K_{\hat{X}}+\epsilon\hat{\mathcal{M}}+a_{1}\hat{\Delta}_{1}+\Big(\epsilon\mathrm{mult}_{O}\big(\mathcal{M}\big)+a_{1}-1\Big)E\equiv\pi^{*}\Big(K_{X}+\epsilon\mathcal{M}+a_{1}\Delta_{1}\Big),
$$
where $\hat{\mathcal{M}}$ and $\hat{\Delta}_{1}$ are proper
transforms of $\mathcal{M}$ and $\Delta_{1}$ on the~surface
$\hat{X}$,~respectively.

\begin{lemma}
\label{lemma:m-1} Suppose that $m=1$. Then
\begin{itemize}
\item the~inequality~\ref{equation:Cheltsov} holds,%
\item if the~inequality~\ref{equation:Cheltsov} is, actually, an
equality, then
\begin{itemize}
\item either $-a_{1}\in\mathbb{N}$ and $\mathrm{mult}_{O}(\mathcal{M})=2/\epsilon$,%
\item or $a_{1}=0$ and $\mathrm{mult}_{O}(\mathcal{M})=1/\epsilon$.%
\end{itemize}
\end{itemize}
\end{lemma}
\begin{proof}
We have $\pi=\chi$ and
$d_{1}=1-\epsilon\mathrm{mult}_{O}(\mathcal{M})-a_{1}\leqslant 0$.
Hence, we have
$$
\mathrm{mult}_{O}\Big(M_{1}\cdot M_{1}\Big)\geqslant
\mathrm{mult}^{2}_{P}\big(\mathcal{M}\big)\geqslant
\frac{\big(1-a_{1}\big)^{2}}{\epsilon^{2}}\geqslant
\mathrm{max}\left(\frac{1-2a_{1}}{\epsilon^{2}},\
\frac{-4a_{1}}{\epsilon^{2}}\right),
$$
which implies the~inequality~\ref{equation:Cheltsov}. Moreover, if
the~inequality~\ref{equation:Cheltsov} is equality, then
\begin{itemize}
\item either $a_{1}=-1$ and $\mathrm{mult}_{O}(\mathcal{M})=2/\epsilon$,%
\item or $a_{1}=0$ and $\mathrm{mult}_{O}(\mathcal{M})=1/\epsilon$,%
\end{itemize}
which completes the~proof.
\end{proof}

Let us prove Theorem~\ref{theorem:II} by induction on $m$. We may
assume that $m\geqslant 2$. Then $a_{1}<0$ and
$$
\Bigg(\hat{X},\
a_{1}\hat{\Delta}_{1}+\epsilon\hat{\mathcal{M}}+\Big(\epsilon\mathrm{mult}_{O}\big(\mathcal{M}\big)+a_{1}-1\Big)E\Bigg)
$$
is not terminal at some point $Q\in E$. Note that
$d_{1}=1-\epsilon\mathrm{mult}_{O}(\mathcal{M})-a_{1}$.

Let $\hat{M}_{1}$ and $\hat{M}_{2}$ be proper transforms of
$M_{1}$ and $M_{2}$ on the~surface $\hat{X}$, respectively. Then
$$
\mathrm{mult}_{O}\Big(M_{1}\cdot M_{2}\Big)\geqslant\mathrm{mult}^{2}_{P}\big(\mathcal{M}\big)+\mathrm{mult}_{Q}\Big(\hat{M}_{1}\cdot \hat{M}_{2}\Big),%
$$
where $\mathrm{mult}_{O}(\mathcal{M})\geqslant 1/\epsilon$. But
$d_{1}=1-\epsilon\mathrm{mult}_{O}(\mathcal{M})-a_{1}\leqslant 0$
and the~log pair
$$
\Bigg(\hat{X},\
\epsilon\hat{\mathcal{M}}+\Big(\epsilon\mathrm{mult}_{O}\big(\mathcal{M}\big)+a_{1}-1\Big)E\Bigg)
$$
is not terminal at the~point $Q$ as well. By induction, we have
\begin{equation}
\label{equation:Cheltsov-auxiliary}
\mathrm{mult}_{Q}\Big(\hat{M}_{1}\cdot \hat{M}_{2}\Big)\geqslant
\left\{\aligned
&\frac{3-2\epsilon\mathrm{mult}_{O}\big(\mathcal{M}\big)-2a_{1}}{\epsilon^{2}}\ \text{if}\ \epsilon\mathrm{mult}_{O}\big(\mathcal{M}\big)+a_{1}\geqslant 1\slash 2,\\
&\frac{4-4\epsilon\mathrm{mult}_{O}\big(\mathcal{M}\big)-4a_{1}}{\epsilon^{2}}\ \text {if}\ \epsilon\mathrm{mult}_{O}\big(\mathcal{M}\big)+a_{1}\leqslant 1\slash 2,\\
\endaligned
\right.
\end{equation}
and  if the~inequality~\ref{equation:Cheltsov-auxiliary} is an
equality, then
\begin{itemize}
\item either $-a_{1}-\epsilon\mathrm{mult}_{O}(\mathcal{M})+1\in\mathbb{N}$ and $\mathrm{mult}_{Q}(\hat{M}_{1})=\mathrm{mult}_{Q}(\hat{M}_{2})=2/\epsilon$,%
\item or $\epsilon\mathrm{mult}_{O}(\mathcal{M})+a_{1}=1$ and $\mathrm{mult}_{Q}(\hat{M}_{1})=\mathrm{mult}_{Q}(\hat{M}_{2})=1/\epsilon$.%
\end{itemize}

\begin{lemma}
\label{lemma:Cheltsov-inequality-1} Suppose that
$\epsilon\mathrm{mult}_{O}(\mathcal{M})+a_{1}\leqslant 1/2$. Then
\begin{itemize}
\item the~inequality~\ref{equation:Cheltsov} holds,%
\item if the~inequality~\ref{equation:Cheltsov} is, actually, an
equality, then
\begin{itemize}
\item either $-a_{1}\in\mathbb{N}$ and $\mathrm{mult}_{O}(\mathcal{M})=2/\epsilon$,%
\item or $a_{1}=0$ and $\mathrm{mult}_{O}(\mathcal{M})=1/\epsilon$.%
\end{itemize}
\end{itemize}
\end{lemma}

\begin{proof}
Then $a_{1}\leqslant
1/2-\epsilon\mathrm{mult}_{O}(\mathcal{M})\leqslant -1/2$ and
$$
\mathrm{mult}_{O}\Big(M_{1}\cdot
M_{2}\Big)\geqslant\mathrm{mult}^{2}_{P}\big(\mathcal{M}\big)+\frac{4-4\epsilon\mathrm{mult}_{O}\big(\mathcal{M}\big)-4a_{1}}{\epsilon^{2}}=
\frac{-4a_{1}}{\epsilon^{2}}+\Bigg(\mathrm{mult}_{O}\big(\mathcal{M}\big)-\frac{2}{\epsilon}\Bigg)^{2}
\geqslant \frac{-4a_{1}}{\epsilon^{2}},
$$
which implies the~inequality~\ref{equation:Cheltsov}. Moreover, if
the~inequality~\ref{equation:Cheltsov} is an~equality, then
$$
\mathrm{mult}_{O}\big(\mathcal{M}\big)=\frac{2}{\epsilon},
$$
and the~inequality~\ref{equation:Cheltsov-auxiliary} is
an~equality as well. By induction, we see that
$$
-a_{1}-1=-a_{1}-\epsilon\mathrm{mult}_{O}(\mathcal{M})+1\in\mathbb{N},
$$
which implies that $-a_{1}\in\mathbb{N}$.
\end{proof}

Thus, to complete the~proof of Theorem~\ref{theorem:II}, we may
assume that $\epsilon\mathrm{mult}_{O}(\mathcal{M})+a_{1}\geqslant
1/2$.

\begin{lemma}
\label{lemma:Cheltsov-inequality-2} Suppose that $a_{1}\geqslant
-1/2$. Then
\begin{itemize}
\item the~inequality~\ref{equation:Cheltsov} holds,%
\item the~inequality~\ref{equation:Cheltsov} is not an equality.
\end{itemize}
\end{lemma}

\begin{proof}
It follows from
$\epsilon\mathrm{mult}_{O}(\mathcal{M})+a_{1}\geqslant 1/2$ and
$a_{1}\geqslant -1/2$ that
$$
\mathrm{mult}_{O}\Big(M_{1}\cdot M_{1}\Big)\geqslant\mathrm{mult}^{2}_{P}\big(\mathcal{M}\big)+\frac{3-2\epsilon\mathrm{mult}_{O}\big(\mathcal{M}\big)-2a_{1}}{\epsilon^{2}}\geqslant\frac{2-2a_{1}}{\epsilon^{2}},%
$$
which completes the~proof.
\end{proof}

Thus, to complete the~proof of Teorem~\ref{theorem:II}, we may
assume that $a_{1}\leqslant -1/2$. Then
$$
\mathrm{mult}_{O}\Big(M_{1}\cdot
M_{1}\Big)\geqslant\mathrm{mult}^{2}_{P}\big(\mathcal{M}\big)+\frac{3-2\epsilon\mathrm{mult}_{O}\big(\mathcal{M}\big)-2a_{1}}{\epsilon^{2}}
\geqslant\frac{9/4-a_{1}+a^{2}_{1}}{\epsilon^{2}}\geqslant\frac{-4a_{1}}{\epsilon^{2}},%
$$
because $\mathrm{mult}_{O}(\mathcal{M})\geqslant
(1/2-a_{1})/\epsilon$, which implies
the~inequality~\ref{equation:Cheltsov}.

\begin{lemma}
\label{lemma:Cheltsov-inequality-3}
The~inequality~\ref{equation:Cheltsov} is not an equality.
\end{lemma}

\begin{proof}
Suppose that the~inequality~\ref{equation:Cheltsov} is an
equality. Then $a_{1}=-3/2$ and
$$
\mathrm{mult}_{O}\big(\mathcal{M}\big)=\frac{\big(1\slash 2-a_{1}\big)}{\epsilon},%
$$
which implies that $\mathrm{mult}_{O}(\mathcal{M})=2/\epsilon$.
The~inequality~\ref{equation:Cheltsov-auxiliary} must be
an~equality as well. Then
$$
\frac{1}{2}=\epsilon\mathrm{mult}_{O}(\mathcal{M})+a_{1}=1
$$
by induction, which is a~contradiction.
\end{proof}

Thus, the~assertion of Theorem~\ref{theorem:II} is proved.

\section{Orbifolds}
\label{section:del-Pezzo-orbifolds}

Let $X$ be a~Fano variety with at most quotient singularities.
The~number
$$
\mathrm{lct}\big(X\big)=\mathrm{sup}\left\{\lambda\in\mathbb{Q}\ \left|\ %
\aligned
&\text{the~log pair}\ \Big(X, \lambda D\Big)\ \text{is log canonical}\\
&\text{for every effective $\mathbb{Q}$-divisor $D\equiv-K_{X}$}\\
\endaligned\right.\right\}%
$$
is called the~global log canonical threshold of
the~Fano~variety~$X$ (cf. Definition~\ref{definition:threshold}).

\begin{theorem}
\label{theorem:KE} The variety $X$ admits an orbifold
K\"ahler--Einstein metric if
$$
\mathrm{lct}\big(X\big)>\frac{\mathrm{dim}(X)}{\mathrm{dim}(X)+1}.
$$
\end{theorem}

\begin{proof}
See \cite{Ti87}, \cite{DeKo01}, \cite[Appendix~A]{ChSh08c}.
\end{proof}

The~following result is proved in \cite{Ko09}.

\begin{theorem}
\label{theorem:Kosta} Let $X$ be a~hypersurface in
$\mathbb{P}(1,1,2,3)$ of degree $6$. Then
$$
\mathrm{lct}\big(X\big)\geqslant\frac{5}{6}
$$
if the~set~$\mathrm{Sing}(X)$~consists of Du Val singular points
of type $\mathbb{A}_{3}$.
\end{theorem}

\begin{proof}
Suppose that the~set~$\mathrm{Sing}(X)$~consists of Du Val
singular points of type $\mathbb{A}_{3}$. Put $\omega=5/6$.

Suppose that $\mathrm{lct}(X)<\omega$. Then there is an effective
$\mathbb{Q}$-divisor $D$ on the~surface $X$~such~that
$$
D\equiv -K_{X},
$$
and the~log pair $(X, \omega D)$ is not log canonical at some
point $P\in X$.

Let $C$ be the~curve in  $|-K_{X}|$ such that $P\in C$. Then $C$
is~irreducible, and the~log pair
$$
\big(X,\omega C\big)
$$
is log canonical. By Remark~\ref{remark:convexity}, we may assume
that $C\not\subset\mathrm{Supp}(D)$.

The surface $X$ must be singular at the~point $P$, because
otherwise we have
$$
1=K_{X}^{2}=C\cdot D\geqslant\mathrm{mult}_{P}\big(D\big)>1/\omega=\frac{6}{5}>1.%
$$

Let $\pi\colon\tilde{X}\to X$ be a~birational morphism such that
\begin{itemize}
\item the~morphism $\pi$ contracts three irreducible curves $E_{1}$, $E_{2}$, $E_{3}$ to the~point $P\in X$,%

\item the~morphism $\pi$ induces an isomorphism
$$
\bar{X}\setminus \Big(E_{1}\cup E_{2}\cup E_{3}\Big)\cong X\setminus P,%
$$

\item the~surface $\bar{X}$ is smooth along the~curves $E_{1}$, $E_{2}$, $E_{3}$.%
\end{itemize}

We have $E_{1}^{2}=E_{2}^{2}=E_{3}^{2}=-2$. We may assume that
$E_{1}\cdot E_{3}=0$ and $E_{1}\cdot E_{2}=E_{2}\cdot E_{3}=1$.

Let $\bar{C}$ be the~proper transform of the~curve $C$ on
the~surface $\bar{X}$. Then the~equivalence
$$
\bar{C}\equiv\pi^*\big(C\big)-E_1-E_2-E_3
$$
holds. Let $\bar{D}$ be the~proper transform of $D$ on the~surface
$\bar{X}$. Then
$$
\bar{D}\equiv\pi^*\big(D\big)-a_1E_1-a_2E_2-a_3E_3,
$$
where $a_{1}$, $a_{2}$ and $a_{3}$ are positive rational numbers.
Then
$$
\left\{\aligned
&1-a_1-a_3=\bar{D}\cdot\bar{C}\geqslant 0,\\
&2a_1 - a_2=\bar{D}\cdot E_{1}\geqslant 0,\\
&2a_2 - a_1 - a_3=\bar{D}\cdot E_{2}\geqslant 0,\\
&2a_3 - a_2=\bar{D}\cdot E_{3}\geqslant 0,\\
\endaligned
\right.
$$
which gives $1\geqslant a_{1}+a_{3}$, $2a_1\geqslant a_2$, $3a_2
\geqslant 2a_3$, $2a_3\geqslant a_2$, $3a_2\geqslant 2a_1$,
$a_1\leqslant 3/4$, $a_2\leqslant 1$, $a_3\leqslant 3/4$.

It follows from the~equivalence
$$
K_{\bar{X}}+\bar{D}+a_1E_1+a_2E_2+a_3E_3\equiv\pi^*\Big(K_X+D\Big)%
$$
that there is a~point $Q\in E_1\cup E_2\cup E_3$ such that the~log
pair
$$
\Big(\bar{X},\ \bar{D}+a_1E_1+a_2E_2+a_3E_3\Big)%
$$
is not log canonical at the~point $Q$, because $(X, D)$ is not log
canonical at $P\in X$.

Suppose that $Q\in E_1$ and $Q\not \in E_2$. Then $(\bar{X},
\bar{D}+E_1)$ is not log canonical at $Q$. Then
$$
2a_{1}-a_{2}=\bar{D}\cdot E_{1}\geqslant\mathrm{mult}_{Q}\Big(\bar{D}\cdot E_{1}\Big)>1%
$$
by Lemma~\ref{lemma:adjunction}. Therefore, we see that
$$
1\geqslant \frac{4}{3} a_1 \geqslant 2a_1 - \frac{2}{3} a_1\geqslant 2a_1 - a_2>1,%
$$
which is a~contradiction.

Suppose that $Q\in E_2$ and $Q\not\in E_1\cup E_3$. Then
$(\bar{X}, \bar{D}+E_2)$ is not log canonical at $Q$. Then
$$
2a_{2}-a_{1}-a_{3}=\bar{D}\cdot E_{2}\geqslant\mathrm{mult}_{Q}\Big(\bar{D}\cdot E_{2}\Big)>1%
$$
by Lemma~\ref{lemma:adjunction}. Therefore, we see that
$$
1\geqslant a_{2}=2a_2-\frac{a_2}{2}-\frac{a_2}{2}\geqslant 2a_2-a_1-a_3>1,%
$$
which is a~contradiction.

We may assume that $Q=E_1\cap E_2$. Then
$(\bar{X},\bar{D}+a_1E_1+a_2E_2)$ is not log canonical at $Q$.~But
$$
a_{1}+\frac{a_{2}}{2}\leqslant a_{1}+a_{3}\leqslant 1,
$$
because $2a_3\geqslant a_2$. Hence, we can apply
Theorem~\ref{theorem:I} to the~log pair
$$
\Big(\bar{X},\ \bar{D} + a_1E_1+a_2E_2\Big)
$$
with $M=N=0$, $A=2$, $B=3/2$, $\alpha=1$ and $\beta=1/2$. Thus, we
have
$$
2a_{2}-a_{1}-a_{3}=\mathrm{mult}_{O}\Big(\bar{D}\cdot E_{2}\Big)>\frac{3}{2}a_{2}-a_{1},%
$$
because $\bar{D}\cdot E_{1}=2a_{1}-a_{2}$. Then $a_{2}>2a_{3}$,
which is a~contradiction, because $2a_3\geqslant a_2$.
\end{proof}

The~following result is proved in \cite{Ko09}.

\begin{theorem}
\label{theorem:Kosta-2} Let $X$ be a~singular hypersurface in
$\mathbb{P}(1,1,2,3)$ of degree $6$. Then
$$
\mathrm{lct}\big(X\big)=\frac{4}{5}
$$
if the~set~$\mathrm{Sing}(X)$~consists of Du Val singular points
of type $\mathbb{A}_{4}$.
\end{theorem}

\begin{proof}
Suppose that the~set~$\mathrm{Sing}(X)$~consists of Du Val
singular points of type $\mathbb{A}_{4}$.

Let $P$ be a point of the~set $\mathrm{Sing}(X)$, and let
$\pi\colon\tilde{X}\to X$ be a~birational morphism such that
\begin{itemize}
\item the~morphism $\pi$ contracts four irreducible curves $E_{1}$, $E_{2}$, $E_{3}$, $E_{4}$ to the~point $P\in X$,%

\item the~morphism $\pi$ induces an isomorphism
$$
\bar{X}\setminus \Big(E_{1}\cup E_{2}\cup E_{3}\cup E_{4}\Big)\cong X\setminus P,%
$$

\item the~surface $\bar{X}$ is smooth along the~curves $E_{1}$, $E_{2}$, $E_{3}$, $E_{4}$.%
\end{itemize}

We have $E_{1}^{2}=E_{2}^{2}=E_{3}^{2}=E_{4}^{2}=-2$. We may
assume that
$$
E_{1}\cdot E_{3}=E_{1}\cdot E_{4}=E_{2}\cdot E_{4}=0,
$$
and $E_{1}\cdot E_{2}=E_{2}\cdot E_{3}=E_{3}\cdot E_{4}=1$.

There exists a unique smooth irreducible curve
$\bar{Z}\subset\bar{X}$ such that
$$
\pi\big(\bar{Z}\big)\sim -2K_{X}
$$
and $E_2\cap E_3\in Z$. Put $\omega=4/5$ and $Z=\pi(\bar{Z})$.
Then
$$
Z\sim\pi^{*}\Big(-2K_{X}\Big)-E_1-2E_2-2E_3-E_4,
$$
which implies that $(X, \omega Z)$ is log canonical and not log
terminal. Then $\mathrm{lct}(X)\leqslant\omega$.

Suppose that $\mathrm{lct}(X)<\omega$. Then there is an effective
$\mathbb{Q}$-divisor $D$ on the~surface $X$~such~that
$$
D\equiv -K_{X},
$$
and the~log pair $(X, \omega D)$ is not log canonical at some
point $O\in X$.

By Remark~\ref{remark:convexity}, we may assume that
$Z\not\subset\mathrm{Supp}(D)$, because $Z$ is irreducible.

Arguing as in the~proof of Theorem~\ref{theorem:Kosta}, we see
that we may assume that $O=P$.

Let $\bar{D}$ be the~proper transform of $D$ on the~surface
$\bar{X}$. Then
$$
\bar{D}\equiv\pi^*\big(D\big)-a_1E_1-a_2E_2-a_3E_3-a_4E_4,
$$
where $a_{1}$, $a_{2}$, $a_{3}$ and $a_{4}$ are positive rational
numbers. Then it follows from the~equivalence
$$
K_{\bar{X}}+\bar{D}+a_1E_1+a_2E_2+a_3E_3+a_4E_{4}\equiv\pi^*\Big(K_X+D\Big),%
$$
that there is a~point $Q\in E_1\cup E_2\cup E_3\cup E_{4}$ such
that the~log pair
$$
\Big(\bar{X},\ \bar{D}+a_1E_1+a_2E_2+a_3E_3+a_4E_{4}\Big),%
$$
is not log canonical at the~point $Q$, because $(X, D)$ is not log
canonical at $P\in X$.

Let $C$ be the~curve in  $|-K_{X}|$ such that $P\in C$. Then $C$
is~irreducible, and the~log pair
$$
\big(X,\omega C\big)
$$
is log canonical. By Remark~\ref{remark:convexity}, we may assume
that $C\not\subset\mathrm{Supp}(D)$.

Let $\bar{C}$ be the~proper transform of the~curve $C$ on
the~surface $\bar{X}$. Then the~equivalence
$$
\bar{C}\equiv\pi^*\big(C\big)-E_{1}-E_{2}-E_{3}-E_{4}
$$
holds. Note that $\bar{D}\cdot\bar{C}\geqslant 0$. Thus, we have
$$
1-a_1-a_4=\bar{D}\cdot\bar{C}\geqslant 0,
$$
which implies that $a_{1}+a_{4}\leqslant 1$. Similarly, we see
that
$$
\left\{\aligned
&2a_1-a_2=\bar{D}\cdot E_{1}\geqslant 0,\\
&2a_2-a_1-a_3=\bar{D}\cdot E_{2}\geqslant 0,\\
&2a_3-a_2-a4=\bar{D}\cdot E_{3}\geqslant 0,\\
&2a_4-a_3=\bar{D}\cdot E_{4}\geqslant 0,\\
\endaligned
\right.
$$
which implies that $a_1\leqslant 4/5$, $a_2\leqslant 6/5$,
$a_3\leqslant 6/5$ and $a_4\leqslant 4/5$. Thus, we see that
$$
\mathbb{LCS}\Big(\bar{X},\ \bar{D}+a_1E_1+a_2E_2+a_3E_3+a_4E_{4}\Big)=\big\{Q\big\},%
$$
by Theorem~\ref{theorem:connectedness}, because $\omega<5/6$.
Similarly, we see that
$$
\frac{4}{5}a_{1}+ \frac{2}{5}a_{2}=\omega a_{1}+ \frac{\omega a_{2}}{2}\leqslant 1.%
$$

Suppose that $Q\in E_1$ and $Q\not \in E_2$. Then $(\bar{X},
\bar{D}+E_1)$ is not log canonical at $Q$. Then
$$
2a_{1}-a_{2}=\bar{D}\cdot E_{1}\geqslant\mathrm{mult}_{Q}\Big(\bar{D}\cdot E_{1}\Big)>1%
$$
by Lemma~\ref{lemma:adjunction}. Therefore, we see that
$$
1\geqslant\frac{5}{4}a_1\geqslant 2a_1-\frac{3}{4}a_1\geqslant2a_1-a_2>1,%
$$
which is a~contradiction.

Suppose that $Q\in E_2$ and $Q\not\in E_1\cup E_3$. Then
$$
2a_{2}-a_{1}-a_{3}=\bar{D}\cdot E_{2}\geqslant\mathrm{mult}_{Q}\Big(\bar{D}\cdot E_{2}\Big)>1%
$$
by Lemma~\ref{lemma:adjunction}. Therefore, we see that
$$
1\geqslant\frac{5}{6}a_2\geqslant 2a_2-\frac{a_2}{2}-\frac{2}{3}a_2\geqslant2a_2-a_1-a_3>1,%
$$
which is a~contradiction.

Suppose that $Q=E_1\cap E_2$. Then we can apply
Theorem~\ref{theorem:I} to the~log pair
$$
\Big(\bar{X},\ \omega\bar{D} + \omega a_1E_1+\omega a_2E_2\Big)
$$
with $M=N=0$, $A=2$, $B=3/2$, $\alpha=1$ and $\beta=1/2$. Thus, we
have
$$
2a_{2}-a_{1}-a_{3}=\mathrm{mult}_{O}\Big(\bar{D}\cdot E_{2}\Big)>\frac{3}{2}a_{2}-a_{1},%
$$
because $\bar{D}\cdot E_{1}=2a_{1}-a_{2}$. Then $a_{2}>2a_{3}$,
which easily leads to a~contradiction.

To complete the~proof, we may assume that $Q=E_2\cap E_3$. Then
the~log pair
$$
\Big(\bar{X},\ \omega\bar{D} + \omega a_2E_2+\omega a_3E_3\Big)
$$
is not log canonical at the~point $Q$. Then
$$
2a_2-\frac{1}{2}a_2-a_3 \geqslant 2a_2-a_1-a_3=\bar{D}\cdot E_2\geqslant\mathrm{mult}_{Q}\Big(\bar{D}\cdot E_2\Big)>\frac{5}{4}-a_3,%
$$
by Lemma~\ref{lemma:adjunction}. Similarly, we see that
$$
2a_3-a_2-a_4=\bar{D}\cdot E_3\geqslant\mathrm{mult}_Q\Big(\bar{D}\cdot E_3 \Big)>\frac{5}{4}-a_2,%
$$
which implies that $a_2>5/6$ and $a_3>5/6$.

Let $\xi\colon\tilde{X}\to\bar{X}$ be a~blow up of $Q$, let $E$ be
the~exceptional curve of $\xi$, and let $\tilde{D}$ be the~proper
transform on the~surface $\tilde{X}$ of the~divisor $\bar{D}$. Put
$m=\mathrm{mult}_{Q}(\bar{D})$. Then
$$
\tilde{D}\equiv\xi^{*}\big(\bar{D}\big)-mE.%
$$

Let $\tilde{E}_{1}$, $\tilde{E}_{2}$, $\tilde{E}_{3}$,
$\tilde{E}_{4}$ be the~proper transforms on $\tilde{X}$  of
$E_{1}$, $E_{2}$, $E_{3}$, $E_{4}$, respectively. Then
$$
K_{\tilde{X}}+\omega\tilde{D}+\omega a_2\tilde{E}_2+\omega a_3\tilde{E}_3+ \Big(\omega a_2+\omega a_3+\omega m-1\Big)E \equiv\xi^{*}\Big(K_{\bar{X}}+\omega\bar{D}+\omega a_2E_2+\omega a_3E_3\Big),%
$$
which implies that there is a point $R \in E$ such that the~pair
$$
\Bigg(\tilde{X},\omega\tilde{D}+\omega a_2\tilde{E}_2+\omega a_3\tilde{E}_3+ \Big(\omega a_2+\omega a_3+\omega\mathrm{mult}_{Q}\big(\bar{D}\big)-1\Big)E\Bigg)%
$$
is not log canonical at the~point $R$.

Let $\tilde{Z}$ be the~proper transform on $\tilde{X}$ of the~
curve $\bar{Z}$. Then
$$
0\leqslant\tilde{Z}\cdot\tilde{D}=2-a_2-a_3-\mathrm{mult}_{Q}\big(\bar{D}\big)=2-a_2-a_3-m,%
$$
which implies that $m+a_{2}+a_{3}\leqslant 2$. In particular, we
see that the~inequality
$$
\omega a_2+\omega a_3+\omega m-1\leqslant 2\omega-1\leqslant\frac{3}{5}%
$$
holds. Thus, we see that
$$
\mathbb{LCS}\Bigg(\tilde{X},\omega\tilde{D}+\omega a_2\tilde{E}_2+\omega a_3\tilde{E}_3+ \Big(\omega a_2+\omega a_3+\omega\mathrm{mult}_{Q}\big(\bar{D}\big)-1\Big)E\Bigg)=\big\{R\big\}%
$$
by Theorem~\ref{theorem:connectedness}. Similarly, we see that
$$
\left\{\aligned
&2a_3-a_2-a_4-m=\tilde{E}_3\cdot\tilde{D}\geqslant 0,\\
&2a_2-a_1-a_3- m=\tilde{E}_2\cdot\tilde{D}\geqslant 0,\\
\endaligned
\right.
$$
which implies that $E\cdot\tilde{D}=m\leqslant 1/2$.

Suppose that $Q\not\in E_2\cup E_3$. Then the~log pair
$$
\Bigg(\tilde{X},\omega\tilde{D}+\Big(\omega a_2+\omega a_3+\omega\mathrm{mult}_{Q}\big(\bar{D}\big)-1\Big)E\Bigg)%
$$
is not log canonical at the~point $R$. Then
$$
\frac{5}{4}>\frac{1}{2}\geqslant m=\tilde{D}\cdot E\geqslant\mathrm{mult}_{R}\Big(\tilde{D}\cdot E\Big)>\frac{5}{4}%
$$
by Lemma~\ref{lemma:adjunction}.  The obtained contradiction shows
that either $R=\tilde{E}_2\cap E$ or $R=\tilde{E}_3\cap E$.

Without loss of generality, we may assume that $R=\tilde{E}_2\cap
E$. Then the~log pair
$$
\Bigg(\tilde{X},\omega\tilde{D}+\omega a_2\tilde{E}_2+\Big(\omega a_2+\omega a_3+\omega\mathrm{mult}_{Q}\big(\bar{D}\big)-1\Big)E\Bigg)%
$$
is not log canonical at the~point $R$. Hence, it follows from
Lemma~\ref{lemma:adjunction} that
$$
\frac{5}{4}-a_{2}>\frac{6}{5}-a_{2}=2-\frac{4}{5}-a_2\geqslant 2-a_2-a_3\geqslant m=\tilde{D}\cdot E\geqslant\mathrm{mult}_{R}\Big(\tilde{D}\cdot E\Big)>\frac{5}{4}-a_2,%
$$
because $m+a_{2}+a_{3}\leqslant 2$. The obtained contradiction
concludes the~proof.
\end{proof}

The following result is proved in \cite{ChShPa08}.

\begin{theorem}
\label{theorem:I-3-W-11-21-29-37-D-95} Let $X$ be a~quasismooth
hypersurface in $\mathbb{P}(11,21,29,37)$ of degree $95$. Then
$$
\mathrm{lct}\big(X\big)=\frac{11}{4}.
$$
\end{theorem}

\begin{proof}
We may assume that $X$ is defined by the~quasihomogeneous equation
$$
t^2y+tz^2+xy^4+x^6z=0\subset\mathrm{Proj}\Big(\mathbb{C}\big[x,y,z,t\big]\Big),
$$
where $\mathrm{wt}(x)=11$, $\mathrm{wt}(y)=21$,
$\mathrm{wt}(z)=29$, $\mathrm{wt}(t)=37$.

Let $O_{x}$ be the~point on $X$ that is given by $y=z=t=0$. Then
$O_{x}$ is a~singular point of type
$$
\frac{1}{11}\big(5,2\big)
$$
of the~surface $X$. Let us define singular points $O_{y}$, $O_{z}$
and $O_{t}$ of the~surface $X$ in a~way similar to the~way we
defined the~point $O_{x}$. Then $O_{y}$, $O_{z}$ and $O_{t}$ are
singular points of types
$$
\frac{1}{21}\big(1,2\big),\ \frac{1}{29}\big(11,21\big),\ \frac{1}{37}\big(11,29\big)%
$$
of the~surface $X$, respectively.

Let $C_{x}$ be the~curve that is cut out on $X$ by the~equation
$x=0$. Then
$$
C_{x}=L_{xt}+R_{x},
$$
where $L_{xt}$ and $R_{x}$ are irreducible reduced curves on the~
surface $X$ such that $L_{xt}$ is given by the~equations $x=t=0$,
and the~curve $R_{x}$ is given by the~equations $x=yt+z^2=0$. We
have
$$
\mathrm{lct}\Bigg(X, \frac{3}{11}C_x\Bigg)=\frac{11}{4},
$$
which implies that $\mathrm{lct}(X)\leqslant 11/4$.

Let $C_{y}$ be the~curve that is cut out on $X$ by the~equation
$y=0$. Then
$$
C_{y}=L_{yz}+R_{y},
$$
where $L_{yz}$ and $R_{y}$ are irreducible reduced curves on the~
surface $X$  such that $L_{yz}$ is given by the~equations $y=z=0$,
and the~curve $R_{y}$ is given by the~equations $y=zt+x^6=0$.

Let $C_{z}$ be the~curve that is cut out on $X$ by the~equation
$z=0$. Then
$$
C_{z}=L_{yz}+R_{z},
$$
where $R_{z}$ is an irreducible reduced curve that is given by
the~ equations $z=xy^3+t^2=0$.

Let $C_{t}$ be the~curve that is cut out on $X$ by the~equation
$t=0$. Then
$$
C_{t}=L_{xt}+R_{t},
$$
where $R_{t}$ is an irreducible reduced curve that is given by
the~ equations $t=y^4+x^5z=0$.

The intersection numbers among the~divisors $D$, $L_{xt}$,
$L_{yz}$, $R_x$, $R_y$, $R_z$, $R_t$ are as follows:
$$
D\cdot L_{xt}=\frac{1}{7\cdot 29}, \ \ D\cdot R_x=\frac{2}{7\cdot 37}, \ \ D\cdot R_y=\frac{18}{29\cdot 37},%
$$

$$
D\cdot L_{yz}=\frac{3}{11\cdot 37}, \ \ D\cdot R_z=\frac{2}{7\cdot 11}, \ \ D\cdot R_t=\frac{12}{11\cdot 29},%
$$

$$
L_{xt}\cdot R_x=\frac{2}{21}, \ \ L_{yz}\cdot R_y=\frac{6}{37}, \ \ L_{yz}\cdot R_z=\frac{2}{11}, \ \ L_{xt}\cdot R_t=\frac{4}{29},%
$$

$$
L_{xt}^2=-\frac{47}{21\cdot 29}, \ \ R_x^2=-\frac{52}{21\cdot 37}, \ \ R_y^2=-\frac{48}{29\cdot 37},%
$$

$$
L_{yz}^2=-\frac{45}{11\cdot 37}, \ \ R_z^2=\frac{16}{11\cdot 21}, \ \ R_t^2=\frac{104}{11\cdot 29}.%
$$

We have $L_{xt}\cap R_x=\{O_y\}$, $L_{yz}\cap R_y=\{O_t\}$,
$L_{yz}\cap R_z=\{O_x\}$, $L_{xt}\cap R_t=\{O_z\}$. We have
$$
\mathrm{min}\Bigg(\mathrm{lct}\Bigg(X, \frac{3}{21}C_y\Bigg), \mathrm{lct}\Bigg(X, \frac{3}{29}C_z\Bigg), \mathrm{lct}\Bigg(X, \frac{3}{37}C_t\Bigg) \Bigg)\geqslant\frac{11}{4}.%
$$

Put $\omega=11/4$. Suppose that $\mathrm{lct}(X)<\omega$. Then,
there is an effective $\mathbb{Q}$-divisor
$$
D\equiv  -K_X
$$
such that the~log pair $(X, \omega D)$ is not log canonical at
some point $P\in X$.

By Remark~\ref{remark:convexity} we may assume that the~support of
the~divisor $D$ does not contain at least one component of each
divisor $C_x$, $C_y$, $C_z$, $C_t$. The inequalities
$$
21\Big(D\cdot L_{xt}\Big)=\frac{3}{29}<\frac{4}{11}>\frac{6}{37}=21\Big(D\cdot R_x\Big)%
$$
imply that $P\ne O_y$. Similarly, we see that $P\ne O_x$. The
inequalities
$$
29\Big(D\cdot L_{xt}\Big)=\frac{1}{7}<\frac{4}{11}>\frac{3}{11}=\frac{29}{4}\Big(D\cdot R_t\Big)%
$$
imply that $P\ne O_z$, because the~curve $R_t$ is singular at
$O_z$, and its orbifold multiplicity at the~point $O_{z}$ equals
to $4$.

Put $D=mL_{xt}+\Omega$, where $m$ is a~non-negative rational
number, and $\Omega$ is an effective $\mathbb{Q}$-divisor whose
support does not contain $L_{xt}$. Then $m\leqslant 4/11$, because
the~log pair $(X, \omega D)$~is~log~canonical at the~point $O_y$.
Then
$$
\Big(D-mL_{xt}\Big)\cdot L_{xt}=\frac{3+47m}{21\cdot 29}\leqslant \frac{4}{11},%
$$
which implies that $P\in L_{xt}$  by Lemma~\ref{lemma:adjunction}.
Similarly, we see that either $P=O_t$ or
$$
P\not\in C_x\cup C_y\cup C_z\cup C_t.
$$

Suppose that $P\ne O_t$. Let $\mathcal{L}$ be the~pencil that is
cut out on $X$ by the~equation
$$
\lambda yt+\mu z^2=0,
$$
where $[\lambda:\mu]\in\mathbb{P}^1$. The base locus of the~pencil
$\mathcal{L}$ consists of the~curve $L_{yz}$ and the~point $O_y$.

Let $E$ be the~unique curve in $\mathcal{L}$ that passes through
$P$. Then $E$ is cut out on $X$ by
$$
z^2=\alpha yt
$$
for some non-zero $\alpha\in\mathbb{C}$, because $P\not\in C_x\cup
C_y\cup C_z\cup C_t$.

Suppose that $\alpha\ne -1$. Then $E$ is isomorphic to the~curve
defined by the~equations
$$
yt-z^2=t^2y+xy^4+x^6z=0,
$$
which implies that $E$ is smooth at the~point $P$. Moreover, we
have
$$
E=L_{yz}+C,
$$
where $C$ is an irreducible and reduced curve. Then $P\in C$. We
have
$$
D\cdot C=D\cdot E-D\cdot L_{yz}=\frac{169}{7\cdot 11\cdot 37},
$$
and $C^2>0$. Arguing as above, we obtain a~contradiction with
Lemma~\ref{lemma:adjunction}.

Suppose that $\alpha=-1$. Then $E=L_{yz}+R_x+M$, where $M$ is an
irreducible and reduced curve, which is smooth at the~point $P$.
We have
$$
D\cdot M=D\cdot E- D\cdot L_{yz}-D\cdot R_x=\frac{147}{7\cdot 11\cdot 37},%
$$
and $M^2>0$. Arguing as above, we obtain a~contradiction with
Lemma~\ref{lemma:adjunction}.

Therefore, we see that $P=O_t$. Put
$$
D=\Delta+aL_{yz}+bR_{x},
$$
where $a$ and $b$ are non-negative rational numbers, and $\Delta$
is an effective $\mathbb{Q}$-divisor whose support contains
neither $L_{yz}$ nor $R_x$. Then $a>0$ since otherwise we would
obtain an absurd inequality
$$
\frac{3}{11}=37D\cdot
L_{yz}\geqslant\mathrm{mult}_{O_t}(D)>\frac{4}{11}.
$$

Note that $R_{y}\not\subset\mathrm{Supp}(\Delta)$, because $a>0$.

If $b>0$, then $L_{xt}$ is not contained in the~support of $D$,
and hence
$$
\frac{3}{21\cdot 29}=D\cdot L_{xt}\geqslant b\Big(R_x \cdot L_{xt}\Big)=\frac{2b}{21},%
$$
which implies that $b\leqslant 3/58$. Similarly, we have
$$
\frac{18}{29\cdot 37}=D\cdot R_y\geqslant \frac{6a}{37}+\frac{b}{37}+\frac{\mathrm{mult}_{O_t}(D)-a-b}{37}>\frac{5a}{37}+\frac{4}{11\cdot 37},%
$$
and hence $a<82/1595$. One can easily check that the~equality
$$
\Delta\cdot L_{yz}=\frac{3}{11\cdot 37}+a\frac{45}{11\cdot 37}-\frac{b}{37}%
$$
holds. Similarly, one can easily check that the~equality
$$
\Delta\cdot R_{x}=\frac{2}{7\cdot 37}+b\frac{52}{21\cdot 37}-\frac{a}{37}%
$$
holds. Thus, applying Theorem~\ref{theorem:I} to the~log pair
$$
\Big(X,\ \omega\Delta+\omega aL_{yz}+\omega bR_{x}\Big)
$$
with $M=3/11$, $N=2/7$, $A=45/11$, $B=52/21$, $\alpha=675/197$ and
$\beta=77/197$, we see that
$$
\frac{24681}{45704}=\alpha \omega a+\beta\omega b>1,
$$
which is a~contradiction.
\end{proof}

\begin{theorem}
\label{theorem:I-4-W-13-14-23-33-D-79} Let $X$ be a~quasismooth
hypersurface in $\mathbb{P}(13,14,23,33)$ of degree $79$. Then
$$
\mathrm{lct}\big(X\big)=\frac{65}{32}.
$$
\end{theorem}

\begin{proof}
The surface $X$ can be defined by the~quasihomogeneous equation
$$
z^{2}t+y^{4}z+xt^{2}+x^{5}y=0\subset\mathrm{Proj}\Big(\mathbb{C}\big[x,y,z,t\big]\Big),
$$
where $\mathrm{wt}(x)=13$, $\mathrm{wt}(y)=14$,
$\mathrm{wt}(z)=23$, $\mathrm{wt}(t)=33$.

Let us define points $O_{x}$, $O_{y}$, $O_{z}$, $O_{t}$ as in the~
proof of Theorem~\ref{theorem:I-3-W-13-17-27-41-D-95}. Then
$O_{x}$, $O_{y}$, $O_{z}$, $O_{t}$ are singular points of
the~surface $X$ of types
$$
\frac{1}{13}\big(1,2\big),\ \frac{1}{14}\big(13,5\big),\ \frac{1}{23}\big(13,14\big),\ \frac{1}{33}\big(14,23\big),%
$$
respectively.

Let $C_{x}$ be the~curve that is cut out on $X$ by the~equation
$x=0$. Then
$$
C_{x}=L_{xz}+R_{x},
$$
where $L_{xz}$ and $R_{x}$ are irreducible reduced curves on the~
surface $X$ such that $L_{xt}$ is given by the~equations $x=z=0$,
and the~curve $R_{x}$ is given by the~equations $x=y^4+zt=0$.

Let $C_{y}$ be the~curve that is cut out on $X$ by the~equation
$y=0$. Then
$$
C_{y}=L_{yt}+R_{y},
$$
where $L_{yt}$ and $R_{y}$ are irreducible reduced curves on the~
surface $X$  such that $L_{yt}$ is given by the~equations $y=t=0$,
and the~curve $R_{y}$ is given by the~equations $y=z^2+xt=0$.

Let $C_{z}$ be the~curve that is cut out on $X$ by the~equation
$z=0$. Then
$$
C_{z}=L_{xz}+R_{z},
$$
where $R_{z}$ is an irreducible reduced curve that is given by
the~ equations $z=x^4y+t^2=0$.

Let $C_{t}$ be the~curve that is cut out on $X$ by the~equation
$t=0$. Then
$$
C_{t}=L_{yt}+R_{t},
$$
where $R_{t}$ is an irreducible reduced curve that is given by
the~ equations $t=x^5+y^3z=0$.

The intersection numbers among the~divisors $D$, $L_{xz}$,
$L_{yt}$, $R_x$, $R_y$, $R_z$, $R_t$ are as follows:
$$
L_{xz}^2=-\frac{43}{14\cdot 33},\ R_x^2=-\frac{40}{23\cdot 33},\ L_{xz}\cdot R_{x}=\frac{4}{33},\ D\cdot L_{xz}=\frac{4}{14\cdot 33},\ D\cdot R_{x}=\frac{16}{23\cdot 33},%
$$
$$
L_{yt}^2=-\frac{32}{13\cdot 23},\ R_{y}^2=-\frac{38}{13\cdot 33},\ L_{yt}\cdot R_{y}=\frac{2}{13},\ D\cdot L_{yt}=\frac{4}{13\cdot 23},\ D\cdot R_{y}=\frac{8}{13\cdot 33},%
$$
$$
R_{z}^2=\frac{20}{13\cdot 14},\ L_{xz}\cdot R_{z}=\frac{2}{14},\ D\cdot R_{z}=\frac{8}{13\cdot 14},\ R_{t}^2=\frac{95}{14\cdot 13},\ L_{yt}\cdot R_{t}=\frac{5}{23},%
$$
$$
D\cdot R_{t}=\frac{20}{14\cdot 23},\ R_{y}\cdot R_{x}=L_{xz}\cdot R_{y}=\frac{1}{33},\ R_{x}\cdot L_{yt}=\frac{1}{23}.%
$$

We have $L_{xz}\cap R_x=\{O_t\}$, $L_{xz}\cap R_z=\{O_y\}$,
$L_{yt}\cap R_y=\{O_x\}$, $L_{yt}\cap R_t=\{O_z\}$. Then
$$
\mathrm{lct}\left(X, \frac{4}{13}C_x\right)=\frac{65}{32}<\mathrm{lct}\left(X, \frac{4}{14}C_y\right)=\frac{21}{8}<\mathrm{lct}\left(X, \frac{4}{25}C_t\right)=\frac{33}{10}<\mathrm{lct}\left(X, \frac{4}{23}C_z\right)=\frac{69}{20},%
$$
which implies, in particular, that $\mathrm{lct}(X)\leqslant
65/32$.

Put $\omega=65/32$. Suppose that $\mathrm{lct}(X)<\omega$. Then,
there is an effective $\mathbb{Q}$-divisor
$$
D\equiv  -K_X
$$
such that the~log pair $(X, \omega D)$ is not log canonical at
some point $P\in X$.

By Remark~\ref{remark:convexity} we may assume that
$\mathrm{Supp}(D)$ does not contain at least one component of
every curve $C_x$, $C_y$, $C_z$ and $C_t$. Arguing as in the~proof
of Lemma~\ref{theorem:I-3-W-13-17-27-41-D-95}, we see that
$P=O_{t}$~(see~\cite{ChShPa08}).

The curve $L_{xz}$ is contained in $\mathrm{Supp}(D)$, because
otherwise we have
$$
\mathrm{mult}_{O_t}(D)\leqslant 33\Big(D\cdot L_{xz}\Big)=\frac{2}{7}<\frac{32}{65},%
$$
which is a~contradiction. Therefore, we see that
$R_x\not\subset\mathrm{Supp}(D)$. Put
$$
D=\Delta+aL_{xz}+bR_y,
$$
where $a$ and $b$ are non-negative rational numbers, and $\Delta$
is an effective $\mathbb{Q}$-divisor whose support does contain
the~curves $L_{xz}$ and $R_y$. Then
$$
\frac{16}{23\cdot 33}=D\cdot R_x\geqslant a\Big(L_{xz}\cdot R_x\Big)+\frac{\mathrm{mult}_{O_t}(D)-a}{33}>\frac{3a}{33}+\frac{32}{33\cdot 65},%
$$
which implies that $a<304/4485$. If $b\ne 0$, then
$L_{yt}\not\subset\mathrm{Supp}(D)$ and
$$
\frac{4}{13\cdot 23}=D\cdot L_{yt}\geqslant b\Big(R_y\cdot L_{yt}\Big)=\frac{2b}{13},%
$$
which implies that $b\leqslant 2/23$. We know that the~equality
$$
33\Big(\Delta\cdot L_{xz}\Big)=\frac{4}{14}+a\frac{43}{14}-b%
$$
holds. Similarly, we know that the~equality
$$
33\Big(\Delta\cdot R_{y}\Big)=\frac{8}{13}+b\frac{38}{23}-a%
$$
holds. Thus, putting $M=4/14$, $N=8/13$, $A=43/14$, $B=38/23$
$\alpha=700771/301108$ and
$$
\beta=\frac{69069}{150554},
$$
and applying Theorem~\ref{theorem:I} to the~log pair $(X,
\omega\Delta+\omega aL_{xz}+\omega bR_{y})$,~we~see that
$$
\frac{66727051}{166211616}=\alpha \omega a+\beta\omega b>1,
$$
which is a~contradiction.
\end{proof}

\begin{theorem}
\label{theorem:I-4-W-11-17-24-31-D-79}Let $X$ be a~quasismooth
hypersurface in $\mathbb{P}(11,17,24,31)$ of degree $79$. Then
$$
\mathrm{lct}\big(X\big)=\frac{33}{16}.
$$
\end{theorem}

\begin{proof}
The surface $X$ can be defined by the~quasihomogeneous equation
$$
t^2y+tz^2+xy^4+x^5z=0\subset\mathrm{Proj}\Big(\mathbb{C}\big[x,y,z,t\big]\Big),
$$
where $\mathrm{wt}(x)=11$, $\mathrm{wt}(y)=17$,
$\mathrm{wt}(z)=24$, $\mathrm{wt}(t)=31$.

Let us define points $O_{x}$, $O_{y}$, $O_{z}$, $O_{t}$ as in the~
proof of Theorem~\ref{theorem:I-3-W-13-17-27-41-D-95}. Then
$O_{x}$, $O_{y}$, $O_{z}$, $O_{t}$ are singular points of the~
surface $X$ of types
$$
\frac{1}{11}\big(2,3\big),\ \frac{1}{17}\big(1,2\big),\ \frac{1}{24}\big(11,17\big),\ \frac{1}{31}\big(11,24\big),%
$$
respectively.

Let $C_{x}$ be the~curve that is cut out on $X$ by the~equation
$x=0$. Then
$$
C_{x}=L_{xt}+R_{x},
$$
where $L_{xt}$ and $R_{x}$ are irreducible reduced curves on the~
surface $X$ such that $L_{xt}$ is given by the~equations $x=t=0$,
and the~curve $R_{x}$ is given by the~equations $x=yt+z^2=0$.

Let $C_{y}$ be the~curve that is cut out on $X$ by the~equation
$y=0$. Then
$$
C_{y}=L_{yz}+R_{y},
$$
where $L_{yz}$ and $R_{y}$ are irreducible reduced curves on the~
surface $X$  such that $L_{yz}$ is given by the~equations $y=z=0$,
and the~curve $R_{y}$ is given by the~equations $y=zt+x^5=0$.

Let $C_{z}$ be the~curve that is cut out on $X$ by the~equation
$z=0$. Then
$$
C_{z}=L_{yz}+R_{z},
$$
where $R_{z}$ is an irreducible reduced curve that is given by
the~ equations $z=xy^3+t^2=0$.

Let $C_{t}$ be the~curve that is cut out on $X$ by the~equation
$t=0$. Then
$$
C_{t}=L_{xt}+R_{t},
$$
where $R_{t}$ is an irreducible reduced curve that is given by
the~ equations $t=y^4+x^4z=0$.

The intersection numbers among the~divisors $D$, $L_{xt}$,
$L_{yz}$, $R_x$, $R_y$, $R_z$, $R_t$ are as follows:

$$D\cdot L_{xt}=\frac{1}{6\cdot 17}, \ \ D\cdot R_x=\frac{8}{17\cdot 31}, \ \ D\cdot R_y=\frac{5}{6\cdot 31}, $$
$$D\cdot L_{yz}=\frac{4}{11\cdot 31}, \ \ D\cdot R_z=\frac{8}{11\cdot 17}, \ \ D\cdot R_t=\frac{2}{3\cdot 11}, $$
$$L_{xt}\cdot R_x=\frac{2}{17}, \ \ L_{yz}\cdot R_y=\frac{5}{31}, \ \ L_{yz}\cdot R_z=\frac{2}{11}, \ \ L_{xt}\cdot R_t=\frac{1}{6},$$

$$L_{xt}^2=-\frac{37}{17\cdot 24}, \ \ R_x^2=-\frac{40}{17\cdot 31}, \ \ R_y^2=-\frac{35}{24\cdot 31},$$

$$L_{yz}^2=-\frac{38}{11\cdot 31}, \ \ R_z^2=\frac{14}{11\cdot 17}, \ \ R_t^2=\frac{10}{3\cdot 11}.$$

We have $L_{xt}\cap R_x=\{O_y\}$, $L_{yz}\cap R_y=\{O_t\}$,
$L_{yz}\cap R_z=\{O_x\}$, $L_{xt}\cap R_t=\{O_z\}$. Then
$$
\mathrm{lct}\left(X, \frac{4}{11}C_x\right)=\frac{33}{16}<\mathrm{max}\Bigg(\mathrm{lct}\left(X, \frac{4}{17}C_y\right),\ \mathrm{lct}\left(X, \frac{4}{24}C_z\right),\ \mathrm{lct}\left(X, \frac{4}{31}C_t\right)\Bigg),%
$$
which implies, in particular, that $\mathrm{lct}(X)\leqslant
33/16$.

Put $\omega=33/16$. Suppose that $\mathrm{lct}(X)<\omega$. Then,
there is an effective $\mathbb{Q}$-divisor
$$
D\equiv  -K_X
$$
such that the~log pair $(X, \omega D)$ is not log canonical at
some point $P\in X$.

By Remark~\ref{remark:convexity} we may assume that
$\mathrm{Supp}(D)$ does not contain at least one component of
every curve $C_x$, $C_y$, $C_z$ and $C_t$. Arguing as in the~proof
of Lemma~\ref{theorem:I-3-W-13-17-27-41-D-95}, we see that
$P=O_{t}$~(see~\cite{ChShPa08}).

The curve $L_{yz}$ is contained in $\mathrm{Supp}(D)$, because
otherwise we have
$$
31\Big(D\cdot L_{yz}\Big)=\frac{4}{11}<\frac{16}{33},
$$
which is a~contradiction. Therefore, we see that
$R_{y}\not\subset\mathrm{Supp}(D)$. Put
$$
D=\Delta+aL_{yz}+bR_{x},
$$
where $a$ and $b$ are non-negative rational numbers, and $\Delta$
is an effective divisor whose support does not contain the~curves
$L_{yz}$ and $R_x$. Then it follows from the~inequality
$$
\frac{1}{6\cdot 17}=D\cdot L_{xt}\geqslant b\Big(R_x \cdot L_{xt}\Big)=\frac{2b}{17}%
$$
that $b\leqslant 1/12$. On the~other hand, we have
$$
\frac{5}{6\cdot 31}=D\cdot R_y\geqslant\frac{5a}{31}+\frac{b}{31}+\frac{\mathrm{mult}_{O_t}(D)-a-b}{31}>\frac{4a}{31}+\frac{16}{31\cdot 33},%
$$
which implies that $a<23/264$. We know that the~equality
$$
31\Big(\Delta\cdot L_{yz}\Big)=\frac{4}{11}+a\frac{38}{11}-b%
$$
holds. Similarly, we know that the~equality
$$
31\Big(\Delta\cdot R_{x}\Big)=\frac{8}{17}+b\frac{40}{17}-a%
$$
holds. Thus, applying Theorem~\ref{theorem:I} to the~log pair
$$
\Big(X,\ \omega\Delta+\omega aL_{yz}+\omega bR_{x}\Big)
$$
with $M=4/11$, $N=8/17$, $A=38/11$, $B=40/17$, $\alpha=1444/453$
and $\beta=187/453$,~we~see that
$$
\frac{6221}{9664}=\alpha \omega a+\beta\omega b>1,
$$
which is a~contradiction.
\end{proof}

\begin{theorem}
\label{theorem:I-3-W-13-17-27-41-D-95} Let $X$ be a~quasismooth
hypersurface in $\mathbb{P}(13,17,27,41)$ of degree $95$. Then
$$
\mathrm{lct}\big(X\big)=\frac{65}{24}.
$$
\end{theorem}

\begin{proof}
 The surface $X$ can be defined by the~
quasihomogeneous equation
$$
z^{2}t+y^{4}z+xt^{2}+x^{6}y=0\Big(\mathbb{C}\big[x,y,z,t\big]\Big),
$$
where $\mathrm{wt}(x)=13$, $\mathrm{wt}(y)=17$,
$\mathrm{wt}(z)=27$, $\mathrm{wt}(t)=41$.

Let us define points $O_{x}$, $O_{y}$, $O_{z}$, $O_{t}$ as in the~
proof of Theorem~\ref{theorem:I-3-W-13-17-27-41-D-95}. Then
$O_{x}$, $O_{y}$, $O_{z}$, $O_{t}$ are singular points of the~
surface $X$ of types
$$
\frac{1}{13}\big(1,2\big),\ \frac{1}{17}\big(13,7\big),\ \frac{1}{27}\big(13,17\big),\ \frac{1}{41}\big(17,27\big),%
$$
respectively.

Let $C_{x}$ be the~curve that is cut out on $X$ by the~equation
$x=0$. Then
$$
C_{x}=L_{xz}+R_{x},
$$
where $L_{xz}$ and $R_{x}$ are irreducible reduced curves on the~
surface $X$ such that $L_{xt}$ is given by the~equations $x=z=0$,
and the~curve $R_{x}$ is given by the~equations $x=y^4+zt=0$.

Let $C_{y}$ be the~curve that is cut out on $X$ by the~equation
$y=0$. Then
$$
C_{y}=L_{yt}+R_{y},
$$
where $L_{yt}$ and $R_{y}$ are irreducible reduced curves on the~
surface $X$  such that $L_{yt}$ is given by the~equations $y=t=0$,
and the~curve $R_{y}$ is given by the~equations $y=z^2+xt=0$.

Let $C_{z}$ be the~curve that is cut out on $X$ by the~equation
$z=0$. Then
$$
C_{z}=L_{xz}+R_{z},
$$
where $R_{z}$ is an irreducible reduced curve that is given by
the~ equations $z=t^2+x^5y=0$.

Let $C_{t}$ be the~curve that is cut out on $X$ by the~equation
$t=0$. Then
$$
C_{t}=L_{yt}+R_{t},
$$
where $R_{t}$ is an irreducible reduced curve that is given by
the~ equations $t=x^6+y^3z=0$.

The intersection numbers among the~divisors $D$, $L_{xz}$,
$L_{yt}$, $R_x$, $R_y$, $R_z$, $R_t$ are as follows:
$$
D\cdot L_{xz}=\frac{3}{17\cdot 41},\ D\cdot L_{yt}=\frac{1}{9\cdot 13},\ D\cdot R_{x}=\frac{4}{9\cdot 41},%
$$
$$
D\cdot R_{y}=\frac{6}{13\cdot 41},\ D\cdot R_{z}=\frac{6}{13\cdot 17},\ D\cdot R_{t}=\frac{2}{3\cdot 17},%
$$
$$
L_{xz}^2=-\frac{55}{17\cdot 41},\  L_{yt}^2=-\frac{37}{13\cdot 27},\ R_x^2=-\frac{56}{27\cdot 41},\  R_y^2=-\frac{48}{13\cdot 41},\ R_{z}^2=\frac{28}{13\cdot 17}%
$$
$$
R_{t}^2=\frac{16}{3\cdot 17},\ L_{xz}\cdot R_{x}=\frac{4}{41}, \ L_{yt}\cdot R_{y}=\frac{2}{13}, \ L_{xz}\cdot R_{z}=\frac{2}{17},\ L_{yt}\cdot R_{t}=\frac{2}{9}.%
$$

We have $L_{xz}\cap R_x=\{O_t\}$, $L_{xz}\cap R_z=\{O_y\}$,
$L_{yt}\cap R_y=\{O_x\}$, $L_{yt}\cap R_t=\{O_z\}$. Then
$$
\frac{65}{24}=\mathrm{lct}\left(X, \frac{3}{13}C_x\right)<\frac{51}{12}=\mathrm{lct}\left(X, \frac{3}{17}C_y\right)<\frac{41}{8}=\mathrm{lct}\left(X, \frac{3}{41}C_t\right)<\frac{21}{4}=\mathrm{lct}\left(X, \frac{3}{27}C_z\right).%
$$
which implies that $\mathrm{lct}(X)\leqslant 65/24$.

Put $\omega=65/24$. Suppose that $\mathrm{lct}(X)<\omega$. Then,
there is an effective $\mathbb{Q}$-divisor
$$
D\equiv  -K_X
$$
such that the~log pair $(X, \omega D)$ is not log canonical at
some point $P\in X$.

By Remark~\ref{remark:convexity} we may assume that
$\mathrm{Supp}(D)$ does not contain at least one component of
every curve $C_x$, $C_y$, $C_z$ and $C_t$. Arguing as in the~proof
of Theorem~\ref{theorem:I-3-W-13-17-27-41-D-95}, we see that
$P=O_{t}$~(see~\cite{ChShPa08}).

The curve $L_{xz}$ is contained in $\mathrm{Supp}(D)$, because
otherwise we have
$$
\frac{24}{65}<\mathrm{mult}_{O_t}(D)\leqslant 41\Big(D\cdot L_{xz}\Big)=\frac{3}{17}<\frac{24}{65},%
$$
which is a~contradiction. Then $R_x\not\subset\mathrm{Supp}(D)$.
Put
$$
D=aL_{xz}+bR_y+\Delta,
$$
where $a$ and $b$ are non-negative rational numbers, and $\Delta$
is an effective $\mathbb{Q}$-divisor whose support does not
contain the~curves $L_{xz}$ nor $R_y$. Then
$$
\frac{4}{9\cdot 41}=D\cdot R_x\geqslant a\Big(L_{xz}\cdot R_x\Big)+\frac{\mathrm{mult}_{O_t}(D)-a}{41}>\frac{3a}{41}+\frac{24}{41\cdot 65},%
$$
which implies that $a\leqslant 44/585$. If $b\ne 0$, then
$$
\frac{1}{9\cdot 13}=D\cdot L_{yt}\geqslant b\Big(R_y\cdot L_{yt}\Big)=\frac{2b}{13},%
$$
which implies that $b\leqslant 1/18$. We know that the~equality
$$
41\Big(\Delta\cdot L_{xz}\Big)=\frac{3}{17}+a\frac{55}{17}-b%
$$
holds. Similarly, we know that the~equality
$$
41\Big(\Delta\cdot R_{y}\Big)=\frac{6}{13}+b\frac{48}{41}-a%
$$
holds. Thus, applying Theorem~\ref{theorem:I} to the~log pair
$$
\Big(X,\ \omega\Delta+\omega aL_{xz}+\omega bR_{y}\Big)
$$
with $M=6/13$, $N=3/17$, $A=48/41$, $B=55/17$ $\alpha=29952/19505$
and $\beta=5729/19505$,~we~get
$$
\frac{306379}{1053270}=\alpha \omega b+\beta\omega a>1,
$$
which is a~contradiction.
\end{proof}

\begin{theorem}
\label{theorem:I-2-W-14-17-29-41-D-99} Let $X$ be a~quasismooth
hypersurface in $\mathbb{P}(14,17,29,41)$ of degree $99$. Then
$$
\mathrm{lct}\big(X\big)=\frac{51}{10}.
$$
\end{theorem}

\begin{proof}
The surface $X$ can be defined by the~quasihomogeneous equation
$$
t^2y+tz^2+xy^5+x^5z=0\subset\mathrm{Proj}\Big(\mathbb{C}\big[x,y,z,t\big]\Big),
$$
where $\mathrm{wt}(x)=14$, $\mathrm{wt}(y)=17$,
$\mathrm{wt}(z)=29$, $\mathrm{wt}(t)=41$.

Let us define points $O_{x}$, $O_{y}$, $O_{z}$, $O_{t}$ as in the~
proof of Theorem~\ref{theorem:I-3-W-13-17-27-41-D-95}. Then
$O_{x}$, $O_{y}$, $O_{z}$, $O_{t}$ are singular points of the~
surface $X$ of types
$$
\frac{1}{14}\big(3,13\big),\ \frac{1}{17}\big(12,7\big),\ \frac{1}{29}\big(11,17\big),\ \frac{1}{41}\big(14,29\big),%
$$
respectively.

Let $C_{x}$ be the~curve that is cut out on $X$ by the~equation
$x=0$. Then
$$
C_{x}=L_{xt}+R_{x},
$$
where $L_{xt}$ and $R_{x}$ are irreducible reduced curves on the~
surface $X$ such that $L_{xt}$ is given by the~equations $x=t=0$,
and the~curve $R_{x}$ is given by the~equations $x=yt+z^2=0$.

Let $C_{y}$ be the~curve that is cut out on $X$ by the~equation
$y=0$. Then
$$
C_{y}=L_{yz}+R_{y},
$$
where $L_{yz}$ and $R_{y}$ are irreducible reduced curves on the~
surface $X$  such that $L_{yz}$ is given by the~equations $y=z=0$,
and the~curve $R_{y}$ is given by the~equations $y=zt+x^5=0$.

Let $C_{z}$ be the~curve that is cut out on $X$ by the~equation
$z=0$. Then
$$
C_{z}=L_{yz}+R_{z},
$$
where $R_{z}$ is an irreducible reduced curve that is given by
the~ equations $z=xy^4+t^2=0$.

Let $C_{t}$ be the~curve that is cut out on $X$ by the~equation
$t=0$. Then
$$
C_{t}=L_{xt}+R_{t},
$$
where $R_{t}$ is an irreducible reduced curve that is given by
the~ equations $t=y^5+x^4z=0$.

The intersection numbers among the~divisors $D$, $L_{xt}$,
$L_{yz}$, $R_x$, $R_y$, $R_z$, $R_t$ are as follows:
$$
D\cdot L_{xt}=\frac{2}{17\cdot 29}, \ \ D\cdot R_x=\frac{4}{17\cdot 41}, \ \ D\cdot R_y=\frac{10}{29\cdot 41},%
$$
$$
D\cdot L_{yz}=\frac{1}{7\cdot 41}, \ \ D\cdot R_z=\frac{2}{7\cdot 17}, \ \ D\cdot R_t=\frac{5}{7\cdot 29},%
$$
$$
L_{xt}\cdot R_x=\frac{2}{17}, \ \ L_{yz}\cdot R_y=\frac{5}{41}, \ \ L_{yz}\cdot R_z=\frac{1}{7}, \ \ L_{xt}\cdot R_t=\frac{5}{29},%
$$
$$
L_{xt}^2=-\frac{44}{17\cdot 29}, \ \ R_x^2=-\frac{54}{17\cdot 41}, \ \ R_y^2=-\frac{60}{29\cdot 41},%
$$
$$
L_{yz}^2=-\frac{53}{14\cdot 41}, \ \ R_z^2=\frac{12}{7\cdot 17}, \ \ R_t^2=\frac{135}{14\cdot 29}.%
$$

We have $L_{xt}\cap R_x=\{O_y\}$, $L_{yz}\cap R_y=\{O_t\}$,
$L_{yz}\cap R_z=\{O_x\}$, $L_{xt}\cap R_t=\{O_z\}$. Then
$$
\mathrm{lct}\left(X, \frac{2}{17}C_y\right)=\frac{51}{10}<\mathrm{max}\Bigg(\mathrm{lct}\left(X, \frac{2}{14}C_x\right),\ \mathrm{lct}\left(X, \frac{2}{29}C_z\right),\ \mathrm{lct}\left(X, \frac{2}{41}C_t\right)\Bigg),%
$$
which implies that $\mathrm{lct}(X)\leqslant 51/10$.

Put $\omega=51/10$. Suppose that $\mathrm{lct}(X)<\omega$. Then,
there is an effective $\mathbb{Q}$-divisor
$$
D\equiv  -K_X
$$
such that the~log pair $(X, \omega D)$ is not log canonical at
some point $P\in X$.

By Remark~\ref{remark:convexity} we may assume that
$\mathrm{Supp}(D)$ does not contain at least one irreducible
component of every curve $C_x$, $C_y$, $C_z$ and $C_t$.

Arguing as in the~proof of
Theorem~\ref{theorem:I-3-W-13-17-27-41-D-95}, we see that
$P=O_{t}$. Put
$$
D=aL_{yz}+bR_{x}+\Delta,
$$
where $a$ and $b$ are non-negative rational numbers, and $\Delta$
is an effective $\mathbb{Q}$-divisor whose support contains
neither $L_{yz}$ nor $R_x$. Then $a>0$, because otherwise we have
$$
\frac{1}{7}=41\Big(D\cdot L_{yz}\Big)\geqslant \mathrm{mult}_{O_t}(D)>\frac{10}{51},%
$$
which is a~contradiction. Then
$R_{y}\not\subset\mathrm{Supp}(\Delta)$. If $b>0$, then
$$
\frac{2}{17\cdot 29}=D\cdot L_{xt}\geqslant b\Big(R_x \cdot L_{xt}\Big)=\frac{2b}{17},%
$$
which implies that $b\leqslant 1/19$. Similarly, we see that
$$
\frac{10}{29\cdot 41}=D\cdot R_y\geqslant \frac{5a}{41}+\frac{b}{41}+\frac{\mathrm{mult}_{O_t}(D)-a-b}{41}>\frac{4a}{41}+\frac{4}{21\cdot 41},%
$$
which implies that $a<47/1218$. We know that the~equality
$$
41\Big(\Delta\cdot L_{yz}\Big)=\frac{2}{14}+a\frac{53}{14}-b%
$$
holds. Similarly, we know that the~equality
$$
41\Big(\Delta\cdot R_{x}\Big)=\frac{4}{17}+b\frac{54}{17}-a%
$$
holds. Thus, applying Theorem~\ref{theorem:I} to the~log pair
$$
\Big(X,\ \omega\Delta+\omega aL_{yz}+\omega bR_{x}\Big)
$$
with $M=1/7$, $N=4/17$, $A=53/14$, $B=54/17$, $\alpha=2809/874$
and $\beta=119/437$,~we~see that
$$
\frac{2414323}{3548440}=\alpha \omega a+\beta\omega b>1,
$$
which is a~contradiction.
\end{proof}

It should be pointed out that singular del Pezzo surfaces that
satisfy the~hypotheses of Theorems~\ref{theorem:Kosta},
\ref{theorem:Kosta-2}, \ref{theorem:I-3-W-11-21-29-37-D-95},
\ref{theorem:I-4-W-13-14-23-33-D-79},
\ref{theorem:I-4-W-11-17-24-31-D-79},
\ref{theorem:I-3-W-13-17-27-41-D-95},
\ref{theorem:I-2-W-14-17-29-41-D-99} admit a K\"ahler--Einstein
metric by Theorem~\ref{theorem:KE}.

\section{Icosahedron}
\label{section:icosahedron}

Fix embeddings $\mathbb{A}_{5}\cong
G_{1}\subset\mathrm{Aut}(\mathbb{P}^{1})$ and $\mathbb{A}_{5}\cong
G_{2}\subset\mathrm{Aut}(\mathbb{P}^{2})$, fix the~induced
embedding
$$
\mathbb{A}_{5}\times \mathbb{A}_{5}\cong G_{1}\times
G_{2}\subset\mathrm{Aut}\Big(\mathbb{P}^{1}\times\mathbb{P}^{2}\Big)\cong\mathbb{PGL}\big(2,\mathbb{C}\big)\times\mathbb{PGL}\big(2,\mathbb{C}\big),
$$
put $G=G_{1}\times G_{2}$ and
$X=\mathbb{P}^{1}\times\mathbb{P}^{2}$. Let $\pi_{1}\colon X\to
\mathbb{P}^{1}$ and $\pi_{2}\colon X\to \mathbb{P}^{2}$ be natural
projections.~Then
$$
\mathcal{P}\big(X,G\big)\supseteq\big\{\pi_{1}, \pi_{2}\big\},
$$
where $\mathcal{P}(X,G)$ is the~$G$-pliability of the~variety $X$
(see Definition~\ref{definition:pliablity}).

\begin{theorem}
\label{theorem:A5-A5} The equality $\mathcal{P}(X,G)=\{\pi_{1},
\pi_{2}\}$ holds.
\end{theorem}

It follows from Theorem~\ref{theorem:A5-A5} that there is no
$G$-equivariant birational map $X\dasharrow V$ such that
$$
\mathrm{dim}_{\mathbb{Q}}\Big(\mathrm{Cl}^{G}\big(V\big)\otimes\mathbb{Q}\Big)=1,
$$
and $V$ is a~Fano threefold with at most terminal singularities.

\begin{corollary}
\label{corollary:A5-A5-conjugacy} Let $\gamma\colon
X\dasharrow\mathbb{P}^{3}$ be any birational map. Then
the~subgroup
$$
\gamma\circ G
\circ\gamma^{-1}\subset\mathrm{Bir}\big(\mathbb{P}^{3}\big)
$$
is not conjugate to any subgroup in
$\mathbb{PGL}(3,\mathbb{C})\cong\mathrm{Aut}(\mathbb{P}^{3})\subset\mathrm{Bir}(\mathbb{P}^{3})$.
\end{corollary}

Similarly, It follows from Theorem~\ref{theorem:A5-A5} that
\begin{itemize}
\item for any $G$-equivariant rational map $\xi\colon
X\dasharrow\mathbb{P}^{1}$ whose general fiber is a~rational
surface, there are maps $\rho\in \mathrm{Bir}^{G}(X)$ and
$\sigma\in\mathrm{Aut}(\mathbb{P}^{1})$ such that there is a~
commutative diagram
$$
\xymatrix{
&X\ar@{->}[d]_{\pi_{1}}\ar@{-->}[rr]^{\rho}&& X\ar@{->}[d]^{\xi}&\\
&\mathbb{P}^{1}\ar@{->}[rr]^{\sigma}&&\mathbb{P}^{1},&}
$$
\item for any $G$-equivariant rational map $\xi\colon
X\dasharrow\mathbb{P}^{2}$ whose general fiber is a~rational
curve, there are maps $\rho\in \mathrm{Bir}^{G}(X)$ and
$\sigma\in\mathrm{Bir}(\mathbb{P}^{2})$ such that there is a~
commutative diagram
$$
\xymatrix{
&X\ar@{->}[d]_{\pi_{2}}\ar@{-->}[rr]^{\rho}&&X\ar@{->}[d]^{\xi}&\\
&\mathbb{P}^{2}\ar@{-->}[rr]^{\sigma}&&\mathbb{P}^{2}.&}
$$%
\end{itemize}

\begin{theorem}
\label{theorem:A5-A5-Bir} The isomorphism
$\mathrm{Bir}^{G}(X)\cong\mathbb{A}_{5}\times \mathbb{S}_{5}$
holds.
\end{theorem}

\begin{proof}
Let $\xi$ be an element in $\mathrm{Bir}^{G}(X)$. Using
the~notations of the~proof of Theorem~\ref{theorem:A5-A5},
put~
$$
\bar{X}=X,\ \bar{\pi}=\pi_{1},\ \nu=\xi.
$$

Arguing as in the~proof of Theorem~\ref{theorem:A5-A5}, we see
that there is a~commutative~diagram
$$
\xymatrix{
&X\ar@{->}[d]_{\pi_{1}}\ar@{-->}[rr]^{\xi}&&X\ar@{->}[d]^{\pi_{1}}&\\
&\mathbb{P}^{1}\ar@{->}[rr]_{\sigma}&&\mathbb{P}^{1},&}
$$
where $\sigma\in\mathrm{Aut}^{G_{1}}(\mathbb{P}^{1})$. Then it
follows from Remark~\ref{remark:Kollar} and
Theorem~\ref{theorem:Cremona-Icosahedron} that
\begin{itemize}
\item either $\xi$ induces an isomorphism of the~generic fiber of $\pi_{1}$,%
\item or $\xi\circ\tau$ induces an isomorphism of the~generic fiber of $\pi_{1}$ for some $\tau\in\mathrm{Bir}^{G_{2}}(\mathbb{P}^{2})$,%
\end{itemize}
which implies that either $\xi$ or $\xi\circ\tau$ is biregular by
\cite[Theorem~1.5]{Ch07c} and \cite[Example~5.4]{Ch07b}.~Then
$$
\mathrm{Bir}^{G}\big(X\big)\cong\mathrm{Aut}^{G_{1}}\big(\mathbb{P}^{1}\big)\times\mathrm{Bir}^{G_{2}}\big(\mathbb{P}^{2}\big),%
$$
but $\mathrm{Aut}^{G_{1}}(\mathbb{P}^{1})\cong\mathbb{A}_{5}$ and
$\mathrm{Bir}^{G_{2}}(\mathbb{P}^{2})\cong\mathbb{S}_{5}$ by
Theorem~\ref{theorem:A5-A5}.
\end{proof}

Let us prove Theorem~\ref{theorem:A5-A5}. Suppose that
$\mathcal{P}(X,G)\ne\{\pi_{1},\pi_{2}\}$. Let us derive a~
contradiction.

There is a~$G$-Mori fibration $\bar{\pi}\colon \bar{X}\to \bar{S}$
that is not square birationally equivalent neither to
the~fibration $\pi_{1}$ nor to $\pi_{2}$, and there is a~
$G$-equivariant birational map $\nu\colon \bar{X}\dasharrow
X$.~Put
$$
\mathcal{M}_{\bar{X}}=\Big|\bar{\pi}^{*}\big(D\big)\Big|
$$
for a~very ample divisor $D$ on the~variety $\bar{S}$ in the~case
when $\mathrm{dim}(\bar{S})\ne 0$, and put
$$
\mathcal{M}_{\bar{X}}=\Big|-mK_{\bar{X}}\Big|
$$
for a~sufficiently big and divisible $m\in\mathbb{N}$ in the~case
when $\mathrm{dim}(\bar{S})=0$. Put
$\mathcal{M}_{X}=\nu(\mathcal{M}_{\bar{X}})$.

\begin{lemma}
\label{lemma:A5-A5-not-in-fibers} For every $i\in\{1,2\}$, there
is $\lambda_{i}\in\mathbb{Q}$ such that $\lambda_{i}>0$ and
$$
K_{X}+\lambda_{i}\mathcal{M}_{X}\equiv \pi_{i}^{*}\big(H_{i}\big),%
$$
where $H_{i}$ is a~$\mathbb{Q}$-divisor on $\mathbb{P}^{i}$.
\end{lemma}

\begin{proof}
The existence of positive rational numbers $\lambda_{1}$ and
$\lambda_{2}$ is obvious in the~case when the~equality
$\mathrm{dim}(\bar{S})=0$ holds. Thus, we may assume that either
$\mathrm{dim}(\bar{S})=1$ or $\mathrm{dim}(\bar{S})=2$.

The  fibration $\bar{\pi}$ is not square birational to $\pi_{1}$.
Thus, the~existence of $\lambda_{1}$ is obvious.

The fibration $\bar{\pi}$ is not square birational to $\pi_{2}$,
which implies the~existence of $\lambda_{2}$ in the~case when
$\mathrm{dim}(\bar{S})=2$. Hence, we may assume that
$\mathrm{dim}(\bar{S})=1$. Then $\bar{S}\cong\mathbb{P}^{1}$.

Suppose that $\lambda_{2}$ does not exist. Then there  is a~
commutative diagram
$$
\xymatrix{
&\bar{X}\ar@{->}[d]_{\bar{\pi}}\ar@{-->}[rr]^{\nu}&&X\ar@{->}[d]^{\pi_{2}}&\\
&\mathbb{P}^{1}&&\mathbb{P}^{2}\ar@{-->}[ll]^{\zeta},&}
$$
where $\zeta$ is a~rational dominant map. The normalization of a~
general fiber of $\zeta$ is a~rational curve, because general
fiber of $\bar{\pi}$ is a~rational surface. On the~other hand,
the~map $\zeta$ is $G_{2}$-equivariant, which is impossible,
because $\mathbb{P}^{2}$ is $G_{2}$-birationally rigid by
Theorem~\ref{theorem:Cremona-Icosahedron}.
\end{proof}

Let $D$ be a~general fiber of $\pi_{1}$. Then it follows from
Theorems~\ref{theorem:G-rigidity} and
\ref{theorem:Cremona-Icosahedron} that
\begin{itemize}
\item either the~log pair
$(D,\lambda_{1}\mathcal{M}_{X}\vert_{D})$
has canonical singularities,%

\item or there is an involution $\tau\in\mathrm{Bir}^{G}(X)$ such
that the~log pair
$$
\Bigg(D,
\lambda_{1}^{\prime}\tau\big(\mathcal{M}_{X}\big)\Big\vert_{D}\Bigg)
$$
has canonical singularities, where
$\lambda_{1}^{\prime}\in\mathbb{Q}$ such that
$$
K_{X}+\lambda_{1}^{\prime}\tau\big(\mathcal{M}_{X}\big)\equiv \pi_{1}^{*}\big(H_{1}^{\prime}\big),%
$$
where $H_{1}^{\prime}$ is a~$\mathbb{Q}$-divisor on
$\mathbb{P}^{1}$.
\end{itemize}

\begin{corollary}
\label{corollary:composition} We may assume that $(D,\lambda_{1}\mathcal{M}_{X}\vert_{D})$ has canonical singularities.%
\end{corollary}

There is a~commutative diagram
$$
\xymatrix{
&&W\ar@{->}[ld]_{\alpha}\ar@{->}[rd]^{\beta}&&\\%
&X&&\bar{X}\ar@{-->}[ll]_{\nu},&}
$$ %
where $\alpha$ and $\beta$ are $G$-equivariant birational
morphisms, and $W$ is a~smooth variety.

Let $\mathcal{M}_{W}$ be a~proper transform of
$\mathcal{M}_{\bar{X}}$ on the~variety $W$. Take
$\epsilon\in\mathbb{Q}$. Then
$$
\alpha^{*}\Big(K_{X}+\epsilon\mathcal{M}_{X}\Big)+\sum_{i=1}^{k}a^{\epsilon}_{i}F_{i}\equiv
K_{W}+\epsilon\mathcal{M}_{W}\equiv
\beta^{*}\Big(K_{\bar{X}}+\epsilon\mathcal{M}_{\bar{X}}\Big)+\sum_{i=1}^{r}b^{\epsilon}_{i}E_{i},
$$
where $a^{\epsilon}_{i}$ is a~rational number, $b^{\epsilon}_{i}$
is a~positive rational number, $F_{i}$ is an exceptional prime
divisor of  the~morphism  $\alpha$, and $E_{i}$ is an exceptional
prime divisor of the~morphism $\beta$.

\begin{lemma}
\label{lemma:conic-bundle-1} The inequality
$\lambda_{1}>\lambda_{2}$ holds.
\end{lemma}

\begin{proof}
Suppose that $\lambda_{2}\geqslant\lambda_{1}$. Then $H_{2}$ is
$\mathbb{Q}$-effective and nef.

Using Lemmas~\ref{lemma:rigidity-fibration} and
\ref{lemma:rigidity-Fano}, we may assume that there is
$l\in\{1,\ldots,k\}$~such~that the~surfaces $F_{1},\ldots,F_{l}$
are all $\alpha$-exceptional divisors such that
$F_{1},\ldots,F_{l}$ are not contracted~by the~morphism $\beta$,
and $a^{\lambda_{2}}_{1},\ldots,a^{\lambda_{2}}_{l}$ are negative
rational numbers.

Put $Z_{i}=\alpha(F_{i})$ for every $i\in\{1,\ldots,r\}$.~Then
\begin{itemize}
\item either $\pi_{2}(Z_{i})$ is an irreducible curve,%
\item or $\pi_{2}(Z_{i})$ is a~closed point.%
\end{itemize}

Suppose that $\mathrm{dim}(\pi_{2}(Z_{i}))=1$. Let $\Gamma_{i}$ be
 a~fiber of $\pi_{2}$ over a~
general~point~of~$\pi_{2}(Z_{i})$.~Then
$$
2=\lambda_{2}\mathcal{M}_{X}\cdot\Gamma_{i}\geqslant\lambda_{2}\mathrm{mult}_{Z_{i}}\big(\mathcal{M}_{X}\big)\big|Z_{i}\cap \Gamma_{i}\big|\geqslant 12\lambda_{2}\mathrm{mult}_{Z_{i}}\big(\mathcal{M}_{X}\big),%
$$
because the~smallest $G_{1}$-orbit in $\mathbb{P}^{1}$ consists of
$12$ points (see \cite{Spr77}). Thus, we see that $i>l$.

Thus, we see that $\pi_{2}(Z_{i})$ is a~closed  point for every
$i\in\{1,\ldots,l\}$.

It follows from \cite[Proposition~1]{Pu04d}  that there is a~
commutative diagram
$$
\xymatrix{
&&U\ar@{->}[ld]_{\gamma}\ar@{->}[rr]^{\delta}&&W\ar@{->}[ld]_{\alpha}\ar@{->}[rd]^{\beta}&&\\%
&\mathbb{P}^{1}\times S\ar@{->}[rr]^{\upsilon}\ar@{->}[d]_{\iota}&&X\ar@{->}[d]^{\pi_{2}}&&\bar{X}\ar@{-->}[ll]_{\nu},&\\%
&S\ar@{->}[rr]_{\omega}&&\mathbb{P}^{2}&&&}%
$$ %
where $S$ is a~smooth surface, $\omega$ is a~$G_{2}$-invariant
birational morphism, $\upsilon$ is an induced birational morphism,
$U$ is a~smooth threefold, $\gamma$ and $\delta$ are $G$-invariant
birational morphisms, and
$$
\mathrm{dim}\Big(\iota\circ\gamma\big(\bar{F}_{i}\big)\Big)=1,
$$
where $\bar{F}_{i}$ is a~proper transform of $F_{i}$ on
the~variety $U$, and $i\in\{1,\ldots,k\}$. Put
$V=\mathbb{P}^{1}\times S$.

Let $\mathcal{M}_{U}$ and $\mathcal{M}_{V}$ be proper transforms
of $\mathcal{M}_{X}$ on the~varieties $U$ and $V$, respectively.
Then
$$
K_{U}+\lambda_{2}\mathcal{M}_{U}\equiv \gamma^{*}\Big(K_{V}+\lambda_{2}\mathcal{M}_{V}\Big)+\sum_{i=1}^{k}c_{i}^{\lambda_{2}}\bar{F}_{i}+\sum_{i=1}^{s}c_{i}^{\lambda_{2}}B_{i},\\%
$$
where $c_{i}^{\lambda_{2}}$ and  $d_{i}^{\lambda_{2}}$ are
rational numbers, and $B_{i}$ is an $\delta$-exceptional divisor.

Note that $\bar{F}_{i}$ is not necessary contracted by $\gamma$.
If $\bar{F}_{i}$ is not contracted by $\gamma$, we have
$c_{i}^{\lambda_{2}}=0$.

There is a~$G_{2}$-invariant divisor $R$ on the~surface $S$ such
that
$$
K_{V}+\lambda_{2}\mathcal{M}_{V}\equiv \iota^{*}\big(R\big).%
$$

We have $D\cong\mathbb{P}^{2}$. Let us identify $S$ with a~proper
transform of $D$ on the~variety $V$. Then
$$
R\equiv  K_{S}+\lambda_{2}\mathcal{M}_{V}\Big\vert_{S},
$$
and $\mathcal{M}_{V}\vert_{S}$ is a~proper transform of the~linear
system $\mathcal{M}_{X}\vert_{D}$. But
$$
\kappa\Bigg(D,\lambda_{2}\mathcal{M}_{X}\Big\vert_{D}\Bigg)\geqslant 0,%
$$
because the~log pair $(D, \lambda_{1}\mathcal{M}_{X}\vert_{D})$ is
canonical and $\lambda_{2}\geqslant\lambda_{1}$. Then $R$ is
$\mathbb{Q}$-effective.

By Lemmas~\ref{lemma:rigidity-fibration} and
\ref{lemma:rigidity-Fano}, we have $c_{t}^{\lambda_{2}}<0$ for
some $t\in\{1,\ldots,l\}$. But the~equality
$$
\mathrm{dim}\Big(\iota\circ\gamma\big(\bar{F}_{t}\big)\Big)=1
$$
holds. Arguing as in the~case $\mathrm{dim}(\pi_{2}(Z_{i}))=1$, we
obtain a~contradiction.
\end{proof}

Therefore, we see that $H_{1}\equiv rF$ for some positive
$r\in\mathbb{Q}$, because $\lambda_{2}<\lambda_{1}$. Then
$$
K_{W}+\lambda_{1}\mathcal{M}_{W}\equiv
\alpha^{*}\big(rD\big)+\sum_{i=1}^{k}a^{\lambda_{1}}_{i}F_{i},
$$
and, for simplicity, we put $a_{i}=a^{\lambda_{1}}_{i}$ for every
$i\in\{1,\ldots,k\}$. Put
$$
\mathcal{J}=\Bigg\{P\in\mathbb{P}^{1}\ \Big\vert\ \exists\ i\in\big\{1,\ldots,k\big\}\ \text{such that}\ P\in\pi_{1}\circ\alpha\big(F_{i}\big),\ \pi_{1}\circ\alpha\big(F_{i}\big)\ne\mathbb{P}^{1},\ a_{i}<0\Bigg\},%
$$
and let $D_{\lambda}\cong\mathbb{P}^{2}$ be a~fiber of
the~morphism $\pi_{1}$ over a~point $\lambda\in\mathbb{P}^{1}$.
Then
$$
\alpha^{*}\big(D_{\lambda}\big)\equiv
\bar{D}_{\lambda}+\sum_{i=1}^{k}b^{\lambda}_{i}F_{i},
$$
where $\bar{D}_{\lambda}$ is a~proper transform of $D_{\lambda}$
on the~variety $W$, and $b^{\lambda}_{i}\in\mathbb{N}\cup\{0\}$.
Note that
$$
b^{\lambda}_{i}\ne 0\iff\exists\ \lambda\in\mathcal{J}\ \Big\vert\ \alpha\big(F_{i}\big)\subset D_{\lambda},%
$$
and $\mathcal{J}\ne\varnothing$ by
Corollary~\ref{corollary:composition} and
Lemmas~\ref{lemma:rigidity-fibration} and
\ref{lemma:rigidity-Fano}. For every $\lambda\in\mathcal{J}$, put
$$
\delta_{\lambda}=\mathrm{max}\Bigg\{-\frac{a_{i}}{b^{\lambda}_{i}}\ \Big\vert\ \alpha\big(F_{i}\big)\subset D_{\lambda},\ a_{i}<0\Bigg\}>0,%
$$
and put $\delta_{\lambda}=0$ for every
$\lambda\in\mathbb{P}^{1}\setminus\mathcal{J}$. Then it follows
from Definition~\ref{definition:lct-local} and
Corollary~\ref{corollary:composition} that
$$
\delta_{\lambda}=-\mathrm{c}^{1}_{D_{\lambda}}\big(X,\lambda_{1}\mathcal{M}_{X},D_{\lambda}\big).%
$$

\begin{lemma}
\label{lemma:super-maximal-singularity} %
The inequality $\sum_{\lambda\in
\mathcal{J}}\delta_{\lambda}\geqslant  r$ holds.
\end{lemma}

\begin{proof}
Suppose that $\sum_{\lambda\in \mathcal{J}}\delta_{\lambda}<r$. We
have
$$
K_{W}+\lambda_{1}\mathcal{M}_{W}\equiv
\alpha^{*}\Bigg(\Big(r-\sum_{\lambda\in
\mathcal{J}}\delta_{\lambda}\Big)D\Bigg)+\sum_{\lambda\in\mathcal{J}}\delta_{\lambda}\bar{D}_{\lambda}+\sum_{i=1}^{k}\big(a_{i}+\delta_{\lambda}b^{\lambda}_{i}\big)F_{i},
$$
where $a_{i}+\delta_{\lambda}b^{\lambda}_{i}\geqslant 0$ for every
$i\in\{1,\ldots,k\}$. Then
$$
\kappa\big(\bar{X},\lambda_{1}M_{\bar{X}}\big)=\kappa\big(X,\lambda_{1}M_{X}\big)=\kappa\big(W,\lambda_{1}M_{W}\big)=1,%
$$
which is impossible, because either
$\kappa(\bar{X},\lambda_{1}M_{\bar{X}})\leqslant 0$, or
$\kappa(\bar{X},\lambda_{1}M_{\bar{X}})=3$.
\end{proof}

\begin{corollary}
\label{corollary:super-maximal-singularity-CS-locus} %
For any set of rational numbers
$\{c_{\lambda}\}_{\lambda\in\mathcal{J}}$ such that
$\sum_{\lambda\in\mathcal{J}}c_{\lambda}\leqslant r$, one has
$$
\mathbb{CS}\Bigg(X,\lambda_{1}\mathcal{M}_{X}-\sum_{\lambda\in\mathcal{J}}c_{\lambda}D_{\lambda}\Bigg)\ne\varnothing.%
$$
\end{corollary}

Let $Z\cong\mathbb{P}^{1}$ be a~fiber of $\pi_{2}$, and let
$C\cong\mathbb{P}^{1}$ be a~line in $D\cong\mathbb{P}^{2}$. Then
$$
\mathbb{NE}\big(X\big)=\mathbb{R}_{\geqslant 0}Z\oplus\mathbb{R}_{\geqslant 0}C,%
$$
and $K_{X}^{2}\equiv 9Z+12C$. Let $M_{1}$ and $M_{2}$ be general
divisors in $\mathcal{M}_{X}$. Put
$T_{0}=\lambda_{1}^{2}M_{1}\cdot M_{2}$. Then
$$
T_{0}=Z_{X}+\sum_{\lambda\in \mathbb{P}^{1}}C_{\lambda}\equiv 9Z+\big(12+6r\big)C,%
$$
where $C_{\lambda}$ is an effective cycle whose components are
contained in a~fiber of $\pi_{1}$ over a~point~$\lambda\in
\mathbb{P}^{1}$, and $Z_{X}$ is an effective cycle such that every
component of $Z_{X}$ does not lie in a~fiber of $\pi_{1}$. But
$$
12+6r\leqslant 12+6\sum_{\lambda\in \mathcal{J}}\delta_{\lambda}
$$
by Lemma~\ref{lemma:super-maximal-singularity}. Let
$\beta_{\lambda}\in\mathbb{Q}_{\geqslant 0}$ such that
$C_{\lambda}\equiv\beta_{\lambda}C$. Then
$$
\sum_{\lambda\in \mathcal{J}}\beta_{\lambda}\leqslant\sum_{\lambda\in\mathbb{P}^{1}}\beta_{\lambda}\leqslant 12+6r\leqslant 12+6\sum_{\lambda\in \mathcal{J}}\delta_{\lambda},%
$$
because $Z$ and $C$ generates the~cone $\mathbb{NE}(X)$. Let
$O^{1}_{\lambda}$ be the~$G_{1}$-orbit of the~point
$\lambda\in\mathbb{P}^{1}$. Then
$$
\sum_{\substack{O^{1}_{\lambda}\\ \lambda\in \mathcal{J}}}\beta_{\lambda}\big|O^{1}_{\lambda}\big|\leqslant 12+6r\leqslant 12+6\sum_{\substack{O^{1}_{\lambda}\\ \lambda\in \mathcal{J}}}\delta_{\lambda}\big|O^{1}_{\lambda}\big|,%
$$
which implies that there is a~point $\omega\in \mathbb{P}^{1}$
such that
$$
\beta_{\omega}\big|O^{1}_{\omega}\big|\leqslant 12+6\delta_{\omega}\big|O^{1}_{\omega}\big|,%
$$
where $|O^{1}_{\omega}|\geqslant 12$ (see \cite{Spr77}).

\begin{corollary}
\label{corollary:super-maximal-inequality} There is
$t\in\{1,\ldots,k\}$ such that $\alpha(F_{t})\subset D_{\omega}$
and $\beta_{\omega}\leqslant 1-6a_{t}/b^{\omega}_{t}$, where
$$
\frac{a_{t}}{b^{\omega}_{t}}=-\delta_{\omega}=\mathrm{min}\Bigg\{\frac{a_{i}}{b^{\omega}_{i}}\ \Big\vert\ \alpha\big(F_{i}\big)\subset D_{\omega},\ a_{i}<0\Bigg\}=\mathrm{c}^{1}_{D_{\omega}}\big(X,\lambda_{1}\mathcal{M}_{X},D_{\omega}\big),%
$$
the~log pair
$(X,\lambda_{1}\mathcal{M}_{X}-\delta_{\omega}D_{\omega})$ is
canonical along $D_{\omega}$. But
$$
\alpha\big(F_{t}\big)\in\mathbb{CS}\Big(X,\lambda_{1}\mathcal{M}_{X}+\delta_{\omega}D_{\omega}\Big).%
$$
\end{corollary}

Arguing as in the~proof of Lemma~\ref{lemma:A5-22}, we see that
\begin{itemize}
\item there are no $G_{2}$-invariant curves on $D_{\omega}\cong\mathbb{P}^{2}$ of degree $1$, $3$ and $5$,%
\item there~is unique $G_{2}$-invariant conic $\Gamma_{2}\subset D_{\omega}\cong\mathbb{P}^{2}$, which is smooth,%
\item the~action of the~group $G_2$ on the~curve
$\Gamma_{2}\cong\mathbb{P}^{1}$ induces an embedding
$$
G_{2}\subset\mathrm{Aut}\big(\Gamma_{2}\big)\cong\mathbb{PGL}\Big(2,\mathbb{C}\Big),
$$%
\item every $G_{2}$-orbit in $D_{\omega}\cong\mathbb{P}^{2}$ consists of at least $6$ points,%
\item every $G_{2}$-orbit in $\Gamma_{2}$ consists of at least $12$ points (cf. \cite{Spr77}).%
\end{itemize}

\begin{lemma}
\label{lemma:maximal-singularity-curve} The equality
$\mathrm{dim}(\alpha(F_{t}))=0$ holds.
\end{lemma}

\begin{proof}
Suppose that $\alpha(F_{t})$ is an irreducible curve. Put
$\Lambda=\alpha(F_{t})$. Then
$$
\mathrm{mult}_{\Lambda}\big(\mathcal{M}_{X}\big)>1\big\slash\lambda_{1},
$$
and we may assume that $\Lambda$ is $G_{2}$-invariant by swapping
$F_{t}$ with its $G_2$-orbit.

Let $L$ be a~general line in $D_{\omega}\cong\mathbb{P}^{2}$. Then
$$
3\big\slash\lambda_{1}=L\cdot\mathcal{M}_{X}>\frac{L\cdot\Lambda}{\lambda_{1}},
$$
which implies that $\Lambda=\Gamma_{2}$.

Put $C_{\omega}=m\Lambda+\Delta$, where $\Delta$ is an effective
one-cycle such that $\Lambda\not\subseteq\mathrm{Supp}(\Delta)$.
Then
$$
m\geqslant \left\{\aligned
&1+2\delta_{\omega}\ \text{if}\ \delta_{\omega}\leqslant 1\slash 2,\\
&4\delta_{\omega}\ \text {if}\ \delta_{\omega}\geqslant 1\slash 2,\\
\endaligned
\right.
$$
by Theorem~\ref{theorem:II}. But $\beta_{\omega}\geqslant 2m$ and
$\beta_{\omega}\leqslant 1+6\delta_{\omega}$ by
Corollary~\ref{corollary:super-maximal-inequality}. Then
$$
1\slash 2+3\delta_{\omega}\geqslant \left\{\aligned
&1+2\delta_{\omega}\ \text{if}\ \delta_{\omega}\leqslant 1\slash 2,\\
&4\delta_{\omega}\ \text {if}\ \delta_{\omega}\geqslant 1\slash 2,\\
\endaligned
\right.
$$
which implies that $\delta_{\omega}=1/2$ and $m=2$, which is
impossible by Theorem~\ref{theorem:II}.
\end{proof}

Put $P=\alpha(F_{t})$. Let $O^{2}_{P}\subset
D_{\omega}\cong\mathbb{P}^{2}$ be the~$G_{2}$-orbit of $P\in
D_{\omega}$.

\begin{lemma}
\label{lemma:vertical-horizontal} The inequalities
$\mathrm{mult}_{P}(Z_{X})\leqslant 3/2$ and
$\mathrm{mult}_{P}(C_{\omega})\leqslant \beta_{\omega}/2$ hold.
\end{lemma}

\begin{proof} Put $r=|O^{2}_{P}|$. Then $r\geqslant 6$. Then
$$
9=\Big(Z_{X}+\sum_{\lambda\in \mathbb{P}^{1}}C_{\lambda}\Big)\cdot D_{\omega}=Z_{X}\cdot D_{\omega}\geqslant\sum_{Q\in O^2_{P}}\mathrm{mult}_{Q}\big(Z_{X}\big)=r\mathrm{mult}_{P}\big(Z_{X}\big)\geqslant 6\mathrm{mult}_{P}\big(Z_{X}\big).%
$$
which implies that $\mathrm{mult}_{P}(Z_{X})\leqslant 3/2$. Let us
show that $\mathrm{mult}_{P}(C_{\omega})\leqslant
\beta_{\omega}/2$.

We may consider $C_{\omega}$ as an effective $G_{2}$-invariant
$\mathbb{Q}$-divisor on $\mathbb{P}^{2}\cong D_{\omega}$ of degree
$\beta_{\omega}$. Then
$$
C_{\omega}=m\Gamma_{2}+\Delta,
$$
where $m\in\mathbb{Q}$ such that $\beta_{\omega}/2\geqslant
m\geqslant 0$, and $\Delta$ is an effective $\mathbb{Q}$-divisor
such that $\Gamma_{2}\not\subseteq\mathrm{Supp}(\Delta)$.

Suppose that $P\in\Gamma_{2}$. Then
$$
2\big(\beta_{\omega}-2m\big)=\Gamma_{2}\cdot\Delta\geqslant
\sum_{Q\in
O_{P}}\mathrm{mult}_{Q}\big(\Delta\big)=r\Big(\mathrm{mult}_{P}\big(C_{\omega}\big)-m\Big)\geqslant
r\Big(\mathrm{mult}_{P}\big(C_{\omega}\big)-m\Big),
$$
and $r\geqslant 12$ (see \cite{Spr77}). Therefore, we immediately
see that
$$
\mathrm{mult}_{P}\big(C_{\omega}\big)\leqslant\frac{2\beta_{\omega}+\big(r-4\big)m}{r}\leqslant\frac{2\beta_{\omega}+\big(r-4\big)\beta_{\omega}\slash 2}{r}=\frac{\beta_{\omega}}{2},%
$$
because $r\geqslant 4$. Thus, to complete the~proof we may assume
that $P\not\in\Gamma_{2}$.

Suppose that $\mathrm{mult}_{P}(C_{\omega})>\beta_{\omega}/2$.
Then $\mathrm{mult}_{P}(\Delta)>\beta_{\omega}/2$, and there is
$\mu\in\mathbb{Q}$ such that
$$
\mathrm{mult}_{P}\big(\mu \Delta\big)\geqslant 2
$$
and $\mu<4/\beta_{\omega}$. In particular, we see that
$O^{2}_{P}\subseteq\mathrm{LCS}(\mathbb{P}^{2},\mu \Delta)$.

Suppose that there is a~$G_{2}$-invariant reduced curve
$\Omega\subset\mathbb{P}^{2}$ such that
$$
\mu \Delta=\upsilon \Omega+\Upsilon,
$$
where $\upsilon\geqslant 1$, and $\Upsilon$ is an effective
$\mathbb{Q}$-divisor, such that
$\Omega\not\subseteq\mathrm{Supp}(\Upsilon)$. Then
$$
4>\mu\beta_{\omega}-2m\mu=\mu \Delta\cdot H=\big(\upsilon \Omega+\Upsilon\big)\cdot H\geqslant\upsilon \Omega\cdot H\geqslant\Omega\cdot H,%
$$
where $H$ is a~general line on $\mathbb{P}^{2}$. Then
$\Omega=\Gamma_{2}$, which is a~contradiction.

We see that the~scheme $\mathcal{L}(\mathbb{P}^{2},\mu\Delta)$ is
zero-dimensional, and its support contains $O^{2}_{P}$.

It follows from Theorem~\ref{theorem:Shokurov-vanishing} that
the~sequence
$$
\mathbb{C}^{3}\cong H^{0}\Big(\mathbb{P}^{2},\ \mathcal{O}_{\mathbb{P}^{2}}\big(1\big)\Big)\longrightarrow H^{0}\Big(\mathcal{O}_{\mathcal{L}(\mathbb{P}^{2},\,\mu\Delta)}\Big)\longrightarrow 0%
$$
is exact, which implies that $r\leqslant 3$. But $r\geqslant 6$,
which is a~contradiction.
\end{proof}

\begin{lemma}
\label{lemma:maximal-singularity-point} The inequality
$\mathrm{mult}_{P}(\mathcal{M}_{X})<(2+\delta_{\omega})/\lambda_{1}$
holds.
\end{lemma}

\begin{proof} Suppose that $\mathrm{mult}_{P}(\mathcal{M}_{X})\geqslant
(2+\delta_{\omega})/\lambda_{1}$. Then
$$
\frac{\beta_{\omega}}{2}\geqslant\mathrm{mult}_{P}\big(C_{\omega}\big)\geqslant
\lambda_{1}^{2}\mathrm{mult}_{P}\Big(M_{1}\cdot
M_{2}\Big)-\mathrm{mult}_{P}\big(Z_{X}\big)\geqslant
\lambda_{1}^{2}\mathrm{mult}^{2}_{P}\big(\mathcal{M}_{X}\big)-3/2\geqslant
\delta_{\omega}^{2}+2\delta_{\omega}+\frac{5}{2},
$$
by Lemma~\ref{lemma:vertical-horizontal}. But
$\beta_{\omega}\leqslant 1+6\delta_{\omega}$ by
Corollary~\ref{corollary:super-maximal-inequality}, which is a~
contradiction.
\end{proof}

Put $X_{0}=X$ and $\Theta_{0}=P$. For some $N\in\mathbb{N}$, there
exists a sequence of blow ups
$$ \xymatrix{
&X_{N}\ar@{->}[rr]^{\psi_{N,N-1}}&&X_{N-1}\ar@{->}[rr]^{\psi_{N-1,N-2}}&&\cdots\ar@{->}[rr]^{\psi_{3,2}}
&&X_{2}\ar@{->}[rr]^{\psi_{2,1}}&&X_{1}\ar@{->}[rr]^{\psi_{1,0}}&& X_{0},}%
$$
such that the~following conditions are satisfied:
\begin{itemize}
\item $\psi_{1,\,0}$ is a~blow up of the~point $\Theta_{0}$,%
\item $\psi_{i,\,i-1}$ is a~blow up of a~point $\Theta_{i-1}\in G_{i-1}$ for every $i\in\{2,\ldots,K\}$,\\ where $G_{i-1}\cong\mathbb{P}^{2}$ is the~exceptional divisor of $\psi_{i-1,\,i-2}$,%
\item $\psi_{K+1,\,K}$ is a~blow up of an irreducible curve $\Theta_{K}\subset G_{K}\cong\mathbb{P}^{2}$,%
\item $\psi_{i,\,i-1}$ is a~blow up of an irreducible curve $\Theta_{i-1}\subset G_{i-1}$ for every $i\in\{K+2,\ldots,N\}$\\ such that $\psi_{i-1,i-2}(\Theta_{i-1})=\Theta_{i-1}$, where $G_{i-1}$ is the~exceptional divisor~of~$\psi_{i-1,\,i-2}$,%
\item the~exceptional divisor $G_{N}$ and the~exceptional divisor $F_{t}$ induce the~same discrete valuation of the~field of rational functions of the~variety $X$.%
\end{itemize}

For $N\geqslant j\geqslant i\geqslant 0$, let $G_{i}^{j}$,
$\mathcal{M}_{X_{j}}$, $D_{\omega}^{j}$ be proper transforms on
$X_{j}$ of $G_{i}$, $\mathcal{M}_{X}$, $D_{\omega}$, respectively,
and let $\psi_{j,\,i}=\psi_{i+1,i}\circ\ldots\circ\psi_{j,j-1}$,
where $\psi_{j,j}=\mathrm{id}_{X_{j}}$. Then
$$
K_{X_{N}}+\lambda_{1}\mathcal{M}_{X_{N}}\equiv  \psi_{N,\,0}^{*}\big(rD_{\omega}\big)+\sum_{i=1}^{N}c_{i}G_{i}^{N},%
$$
where $c_{i}\in \mathbb{Q}$ and $c_{N}=a_{t}$. Similarly, we have
$$
\psi_{N,\,0}^{*}\big(D_{\omega}\big)\equiv  D_{\omega}^{N}+\sum_{i=1}^{N}d_{i}G_{i}^{N},%
$$
where $d_{i}\in\mathbb{N}$ and $d_{N}=b^{\omega}_{t}$. We have
$c_{N}<0$ and $\delta_{\omega}=-c_{N}/d_{N}$. We may assume that
$$
-\delta_{\omega}=\frac{a_{t}}{b^{\omega}_{t}}=\frac
{c_{N}}{d_{N}}<\frac{c_{i}}{d_{i}}
$$
for $i<N$. A priori, the~curve $\Theta_{i}$ may be singular for
$i\geqslant K$. But Lemma~\ref{lemma:maximal-singularity-point}
implies that
\begin{itemize}
\item the~curve $\Theta_{K}$ is a~line in $G_{K}\cong \mathbb{P}^{2}$,%
\item for $i>K$, the~curve $\Theta_{i}$ is a~section of
the~induced morphism
$$
\psi_{i-1,\,i-2}\Big\vert_{G_{i-1}}\colon G_{i-1}\longrightarrow \Theta_{i-2}\cong \mathbb{P}^{1},%
$$
which implies that $\Theta_{i}\cong\mathbb{P}^{1}$ for $i\geqslant
K$.
\end{itemize}

Let $\Gamma$ be an oriented graph whose set of vertices consists
of the~divisors $G_{1},\ldots,G_{N}$ such that
$$
\big(G_{j},G_{i}\big)\in\Gamma \iff j>i\ \text{and}\ \Theta_{j-1}\subset G_{i}^{j-1}\subset X_{j-1},%
$$
where $(G_{j},G_{i})$ is the~edge that goes from the~vertex
$G_{j}$ to the~vertex $G_{i}$. Let $P_{i}$ be the~number of
oriented pathes in $\Gamma$ that goes from $G_{N}$ to $G_{i}$.
Then
$$
c_{N}=\sum_{i=1}^{K}P_{i}\big(2-\nu_{i}\big)+\sum_{i=K+1}^{N}P_{i}\big(1-\nu_{i}\big),
$$
where
$\nu_{i}=\lambda_{1}\mathrm{mult}_{P_{i-1}}(\mathcal{M}_{X_{i-1}})$.
Note that $(G_{i},G_{i-1})\in\Gamma$ for every
$i\in\{1,\ldots,N\}$.

Put $\Sigma_{0}=\sum_{i=1}^{K}P_i$ and
$\Sigma_{1}=\sum_{i=K+1}^{N}P_i$.  Then it follows from
\cite{Pu98} that the~inequality
$$
\mathrm{mult}_{P}\big(Z_{X}\big)+\mathrm{mult}_{P}\big(C_{\omega}\big)\geqslant\frac{\big(2\Sigma_{0}+\Sigma_{1}-c_{N}\big)^{2}}{\big(\Sigma_{0}+\Sigma_{1}\big)\Sigma_{0}}
$$
holds. Put $\Sigma_{0}^{\prime}=\sum_{i=1}^{M}P_{i}$, where
$M\leqslant K$ be a~biggest natural such that $P_{M-1}\in
F_{\omega}^{M-1}$. Then
$$
\mathrm{mult}_{P}\big(Z_{X}\big)\Sigma_{0}+\mathrm{mult}_{P}\big(C_{\omega}\big)\Sigma_{0}^{\prime}\geqslant\frac{\big(2\Sigma_{0}+\Sigma_{1}-c_{N}\big)^{2}}{\Sigma_{0}+\Sigma_{1}}
$$
by \cite{Pu98}. The choice of $M$ implies that
$d_{N}\geqslant\Sigma_{0}^{\prime}\leqslant\Sigma_{0}$. But
$$
\mathrm{mult}_{P}\big(C_{\omega}\big)\leqslant\frac{\beta_{\omega}}{2}\leqslant \frac{1}{2}+3\delta_{\omega}=\frac{1}{2}-3\frac{c_{N}}{d_{N}}%
$$
by Lemma~\ref{lemma:vertical-horizontal} and
Corollary~\ref{corollary:super-maximal-inequality}. Therefore, we
see that
$$
\Bigg(\mathrm{mult}_{P}\big(Z_{X}\big)+\frac{1}{2}\Bigg)\Sigma_{0}-3c_{N}\geqslant\frac{\big(2\Sigma_{0}+\Sigma_{1}-c_{N}\big)^{2}}{\Sigma_{0}+\Sigma_{1}},
$$
where $c_{N}<0$ and $\mathrm{mult}_{P}(Z_{X})\leqslant 3/2$.
Therefore, we see that
$$
2\Sigma_{0}-3c_{N}\geqslant\frac{\big(2\Sigma_{0}+\Sigma_{1}-c_{N}\big)^{2}}{\big(\Sigma_{0}+\Sigma_{1}\big)}.
$$

\begin{lemma}
\label{lemma:point-conic} The inequality
$\mathrm{mult}_{P}(Z_{X})\leqslant 3/4$ holds.
\end{lemma}

\begin{proof}
The log pair $(X,\lambda_{1}\mathcal{M}_{X})$ is not canonical at
the~point $P\in D_{\omega}$. Then the~log pair
$$
\Big(X,\ \lambda_{1}\mathcal{M}_{X}+D_{\omega}\Big)
$$
is not log canonical at $P$. Then $(D_{\omega},
\lambda_{1}\mathcal{M}_{X}\vert_{D_{\omega}})$ is not log
canonical at $P$ by Theorem~\ref{theorem:adjunction}.

There is $\mu\in\mathbb{Q}$ such that $0<\mu<\lambda_{1}$ and
$(D_{\omega}, \mu\mathcal{M}_{X}\vert_{D_{\omega}})$ is not log
canonical at $P$. Then
$$
\mathrm{LCS}\Bigg(D_{\omega},\ \mu\mathcal{M}_{X}\Big\vert_{D_{\omega}}\Bigg)%
$$
is connected by Theorem~\ref{theorem:connectedness}. Arguing as in
the~proof of Lemma~\ref{lemma:vertical-horizontal}, we see that
\begin{itemize}
\item either $\mathcal{L}(D_{\omega}, \lambda_{1}\mathcal{M}_{X}\vert_{D_{\omega}})$ is zero-dimensional,%
\item or $P\in\Gamma_{2}$.
\end{itemize}

If $\mathcal{L}(D_{\omega},
\lambda_{1}\mathcal{M}_{X}\vert_{D_{\omega}})$ is
zero-dimensional, then $|O^{2}_{P}|=1$ by connectedness. But
$|O^{2}_{P}|\geqslant 6$.

We see that $P\in\Gamma_{2}$. Then $|O^{2}_{P}|\geqslant 12$ (see
\cite{Spr77}). Then
$$
9=\Bigg(Z_{X}+\sum_{\lambda\in \mathbb{P}^{1}}C_{\lambda}\Bigg)\cdot D_{\omega}=Z_{X}\cdot D_{\omega}\geqslant\sum_{Q\in O^{2}_{P}}\mathrm{mult}_{Q}\big(Z_{X}\big)=\big|O^{2}_{P}\big|\mathrm{mult}_{P}\big(Z_{X}\big)\geqslant 12\mathrm{mult}_{P}\big(Z_{X}\big),%
$$
which implies that $\mathrm{mult}_{P}(Z_{X})\leqslant 3/4$.
\end{proof}

Therefore, we see that $c_{N}<0$, $\Sigma_{0}\geqslant 1$,
$\Sigma_{1}\geqslant 1$ and
$$
\frac{5}{4}\Sigma_{0}-3c_{N}\geqslant\frac{\big(2\Sigma_{0}+\Sigma_{1}-c_{N}\big)^{2}}{\big(\Sigma_{0}+\Sigma_{1}\big)},
$$
which is a~contradiction. The assertion of
Theorem~\ref{theorem:A5-A5} is proved.

\appendix

\section{Non-rationality}
\label{section:rigidity}

Let $X$ be a~variety, let $G\subset\mathrm{Aut}(X)$ be a~finite
subgroup, let $\pi\colon X\to S$ be a~morphism.

\begin{definition}
\label{definition:Mori-fiber-space} The morphism $\pi$ is a~
$G$-Mori fibration~if the~following conditions are satisfied:
\begin{itemize}
\item the~morphism $\pi$ is $G$-equivariant,%
\item the~variety $X$ has terminal singularities,%
\item every $G$-invariant Weil divisor on $X$ is a~$\mathbb{Q}$-Cartier divisor,%
\item the~morphism $\pi$ is surjective and $\pi_*(\mathcal{O}_{X})=\mathcal{O}_S$,%
\item the~divisor $-K_X$ is relatively ample for $\pi$,%
\item the~inequality $\mathrm{dim}(S)<\mathrm{dim}(X)$ holds and
$$
\mathrm{dim}_{\mathbb{Q}}\Bigg(\mathrm{Pic}^{G}\Big(X\big/S\Big)\otimes\mathbb{Q}\Bigg)=1.
$$
\end{itemize}
\end{definition}

Suppose that $\pi\colon X\to S$ is a~$G$-Mori fibration. It
follows from \cite{Zh06} and \cite{GrHaSta03} that
$$
\text{the~variety}\ S\ \text{is rationally connected}\iff \text{the~variety}\ X\ \text{is rationally connected}.%
$$

\begin{remark}
\label{remark:KMM} The group $G$ naturally acts on the~variety
$S$. However, this action is not necessary faithful. One can show
that every $G$-invariant Weil divisor on $S$ is
a~$\mathbb{Q}$-Cartier divisor (see \cite{KMM}).
\end{remark}

Suppose, in addition, that $X$ is rationally connected. Then
$$
\mathrm{dim}_{\mathbb{Q}}\Big(\mathrm{Pic}^{G}\big(X\big)\otimes\mathbb{Q}\Big)=\mathrm{dim}_{\mathbb{Q}}\Big(\mathrm{Pic}^{G}\big(S\big)\otimes\mathbb{Q}\Big)+1.
$$

\begin{definition}
\label{definition:rigidity} The fibration $\pi$ is
$G$-birationally rigid if, given any $G$-equivariant birational
map $\xi\colon X\dasharrow X^{\prime}$ to a~$G$-Mori fibration
$\pi^{\prime}\colon X^{\prime}\to S^{\prime}$, there is a~
commutative diagram
$$
\xymatrix{
&X\ar@{->}[d]_{\pi}\ar@{-->}[rr]^{\rho}&&X\ar@{-->}[rr]^{\xi}&&X^{\prime}\ar@{->}[d]^{\pi^{\prime}}&\\
&S\ar@{-->}[rrrr]^{\sigma}&&&&S^{\prime},&}
$$
where $\sigma$ is a~birational map, and
$\rho\in\mathrm{Bir}^{G}(X)$ such that the~rational $\xi\circ\rho$
induces an isomorphism of the~generic fibers of the~$G$-Mori
fibrations $\pi$ and $\pi^{\prime}$.
\end{definition}

\begin{definition}
\label{definition:superrigidity} The fibration $\pi$  is
$G$-bi\-ra\-ti\-onally superri\-gid if, given any
$G$-equi\-va\-ri\-ant bi\-ra\-ti\-o\-nal map $\xi\colon
X\dasharrow X^{\prime}$ to a~$G$-Mori fibration
$\pi^{\prime}\colon X^{\prime}\to S^{\prime}$, there is a~
commutative diagram
$$
\xymatrix{
&X\ar@{->}[d]_{\pi}\ar@{-->}[rr]^{\xi}&&X^{\prime}\ar@{->}[d]^{\pi^{\prime}}&\\
&S\ar@{-->}[rr]_{\sigma}&&S^{\prime},&}
$$
where $\sigma$ is a~birational map, and $\xi$ induces an
isomorphism of the~generic fibers of $\pi$ and $\pi^{\prime}$.
\end{definition}

We say that $X$ is $G$-birationally rigid ($G$-birationally
superrigid, respectively) if $\mathrm{dim}(S)=0$ and $\pi\colon
X\to S$ is $G$-birationally rigid ($G$-birationally superrigid,
respectively).

\begin{remark}
\label{remark:Fanos-rigid-superrigid} Suppose that
$\mathrm{dim}(S)=0$. Then the~following conditions are equivalent:
\begin{itemize}
\item the~Fano variety $X$ is $G$-birationally superrigid,%
\item the~Fano variety $X$ is $G$-birationally rigid and $\mathrm{Bir}^{G}(X)=\mathrm{Aut}^{G}(X)$.%
\end{itemize}
\end{remark}

We say that $\pi\colon X\to S$ is birationally rigid (birationally
superrigid, respectively) if $\pi\colon X\to S$ is
$G$-birationally rigid ($G$-birationally superrigid,
respectively), where $G$ is a~trivial group.

\begin{example}
\label{example:Pukhlikov}%
It follows from \cite{Pu98} that $\pi\colon X\to S$ is
birationally rigid if
\begin{itemize}
\item the~equalities $\mathrm{dim}(X)=3$ and $\mathrm{dim}(S)=1$ hold,%
\item the~inequality $K_{X}^{2}\cdot F\leqslant 2$ holds, where $F$ is a~general fiber of $\pi$,%
\item the~threefold $X$ is smooth and
$$
K_{X}^{2}\not\in\mathrm{Int}\Big(\overline{\mathbb{NE}}\big(X\big)\Big),
$$
where $\mathrm{Int}(\overline{\mathbb{NE}}(X))$ is an interior of
the~closure of the~cone of effective one-cycles.
\end{itemize}
\end{example}

We say that $X$ is birationally rigid (birationally superrigid,
respectively) if $\mathrm{dim}(S)=0$ and the~fibration $\pi\colon
X\to S$ is birationally rigid (birationally superrigid,
respectively).

\begin{example}
\label{example:hypersurfaces} General hypersurface in
$\mathbb{P}^{n}$ of degree $n\geqslant 4$  is birationally
superrigid  by \cite{Pu98a}.
\end{example}

It follows from Definition~\ref{definition:rigidity} that if
the~$G$-Mori fibration $\pi\colon X\to S$ is $G$-birationally
rigid and $X\not\cong\mathbb{P}^{n}$, then there is no
$G$-equivariant birational map $X\dasharrow\mathbb{P}^{n}$, where
$n=\mathrm{dim}(X)$.

\begin{definition}
\label{definition:square} We say that the~$G$-Mori fibration
$\pi\colon X\to S$ is square birationally equivalent~to a~$G$-Mori
fibration~$\pi^{\prime}\colon X^{\prime}\to S^{\prime}$ if there
exists a~commutative diagram
$$
\xymatrix{
&X\ar@{->}[d]_{\pi}\ar@{-->}[rr]^{\xi}&&X^{\prime}\ar@{->}[d]^{\pi^{\prime}}&\\
&S\ar@{-->}[rr]_{\sigma}&&S^{\prime},&}
$$
where $\sigma$ is a~birational map, and $\xi$ is a~$G$-equivariant
birational map such that the~map $\xi$ induces an~isomorphism of
the~generic fibers of the~$G$-Mori fibrations $\pi\colon X\to S$
and $\pi^{\prime}\colon X^{\prime}\to S^{\prime}$.
\end{definition}

The following definition is due to \cite{CoMe02}.

\begin{definition}
\label{definition:pliablity} Let $V$ be a~variety, and let
$\Gamma\subset\mathrm{Aut}(V)$ be a~finite subgroup. The set
$$
\mathcal{P}\big(V,\ G\big)=\Big\{\text{a $G$-Mori fibration}\ \tau\colon Y\longrightarrow T\ \Big\vert\ \text{$\exists$\ a~$G$-equivariant birational map}\ Y\dasharrow V \Big\}\Big\slash\bigstar%
$$
is a~$G$-pliability of the~variety $V$, where $\bigstar$ is
a~square birational equivalence.
\end{definition}

We put $\mathcal{P}(V)=\mathcal{P}(V,G)$ if the~group $G$ is
trivial. The~following conditions are equivalent:
\begin{itemize}
\item the~fibration $\pi\colon X\to S$ is $G$-birationally rigid,%
\item the~set $\mathcal{P}(X,G)$ consists of the~fibration $\pi\colon X\to S$,%
\item the~equality $|\mathcal{P}(X,G)|=1$ holds.%
\end{itemize}

\begin{remark}
\label{remark:BCHM} In the~notation and assumption of
Definition~\ref{definition:pliablity}, it follows from
\cite{BCHM06} that
$$
\mathcal{P}\big(V,\ \Gamma\big)\ne\varnothing \iff -K_{V}\ \text{is not pseudoeffective}\iff\ V\ \text{is uniruled}.%
$$
\end{remark}

\begin{example}
\label{example:Corti-Mella} Let $X$ be a~sufficiently general
quartic in $\mathbb{P}^{4}$ that is given by
$$
w^{2}x^2+wyz+xg_{3}(y,z,t,w)+g_{4}(y,z,t,w)=0\subset\mathbb{P}^4\cong\mathrm{Proj}\Big({\mathbb C}[x,y,z,t,w]\Big),%
$$
where $g_{i}$ is a~homogeneous polynomial of degree $i$. Then
$|\mathcal{P}(X)|=2$ by \cite{CoMe02}.
\end{example}

Let $\bar{\pi}\colon \bar{X}\to \bar{S}$ be a~$G$-Mori fibration,
let $\nu\colon \bar{X}\dasharrow X$ be a~$G$-equivariant
birational map.~Put
$$
\mathcal{M}_{\bar{X}}=\Big|\bar{\pi}^{*}\big(D\big)\Big|
$$
for a~very ample divisor $D$ on the~variety $\bar{S}$ in the~case
when $\mathrm{dim}(\bar{S})\ne 0$, and put
$$
\mathcal{M}_{\bar{X}}=\Big|-mK_{\bar{X}}\Big|
$$
for a~sufficiently big and divisible $m\in\mathbb{N}$ in the~case
when $\mathrm{dim}(\bar{S})=0$. Put
$\mathcal{M}_{X}=\nu(\mathcal{M}_{\bar{X}})$.

\begin{lemma}
\label{lemma:rigidity-fibration-easy} Suppose that
$\mathrm{dim}(\bar{S})\ne 0$. Then either there  is a~commutative
diagram
\begin{equation}
\label{equation:diagram-fibres} \xymatrix{
&\bar{X}\ar@{->}[d]_{\bar{\pi}}\ar@{-->}[rr]^{\nu}&&X\ar@{->}[d]^{\pi}&\\
&\bar{S}&&S\ar@{-->}[ll]^{\zeta},&}
\end{equation}
where $\zeta$ is a~rational dominant map, or there is
$\lambda\in\mathbb{Q}$ such that
$$
K_{X}+\lambda\mathcal{M}_{X}\equiv \pi^{*}\big(H\big),
$$
where $H$ is a~$G$-invariant $\mathbb{Q}$-divisor on the~variety
$S$ such that
\begin{itemize}
\item either the~log pair $(X, \lambda \mathcal{M}_{X})$ is not canonical,%
\item or the~divisor $H$ is not $\mathbb{Q}$-effective.%
\end{itemize}
\end{lemma}

\begin{proof}
The diagram~\ref{equation:diagram-fibres} exists $\iff$
$\mathcal{M}_{X}$ lies in the~fibers of $\pi$.

We may assume that $\mathcal{M}_{X}$ does not lie in the~fibers of
$\pi$. Then there is $\lambda\in\mathbb{Q}$ such that
$$
K_{X}+\lambda\mathcal{M}_{X}\equiv \pi^{*}\big(H\big),
$$
where $H$ is a~$G$-invariant $\mathbb{Q}$-divisor on the~variety
$S$. By Lemma~\ref{lemma:Kodaira-dimension-is-OK}, we have
$$
\kappa\Big(K_{X}+\lambda\mathcal{M}_{X}\Big)=-\infty,
$$
but $\kappa(K_{X}+\lambda\mathcal{M}_{X})=0$ if $(X, \lambda
\mathcal{M}_{X})$ is canonical and $H$ is $\mathbb{Q}$-effective
(see Corollary~\ref{corollary:Kodaira-dimension}).
\end{proof}

Let $\epsilon$ be a~positive rational number. There is a~
commutative diagram
$$
\xymatrix{
&&W\ar@{->}[ld]_{\alpha}\ar@{->}[rd]^{\beta}&&\\%
&X&&\bar{X}\ar@{-->}[ll]_{\nu},&}
$$ %
where $\alpha$ and $\beta$ are $G$-equivariant birational
morphisms, and $W$ is a~smooth variety.

Let $\mathcal{M}_{W}$ be a~proper transform of the~linear system
$\mathcal{M}_{\bar{X}}$ on the~variety $W$. Then
$$
\alpha^{*}\Big(K_{X}+\epsilon\mathcal{M}_{X}\Big)+\sum_{i=1}^{k}a^{\epsilon}_{i}F_{i}\equiv
K_{W}+\epsilon\mathcal{M}_{W}\equiv
\beta^{*}\Big(K_{\bar{X}}+\epsilon\mathcal{M}_{\bar{X}}\Big)+\sum_{i=1}^{r}b^{\epsilon}_{i}E_{i},
$$
where $a^{\epsilon}_{i}$ is a~rational number, $b^{\epsilon}_{i}$
is a~positive rational number, $F_{i}$ is a~$G$-orbit of an
exceptional prime divisor of  the~morphism  $\alpha$, and $E_{i}$
is a~$G$-orbit of an exceptional prime divisor of $\beta$.

\begin{lemma}
\label{lemma:rigidity-fibration} Suppose that
$\mathrm{dim}(\bar{S})\ne 0$. Then either there  is a~commutative
diagram
$$
\xymatrix{
&\bar{X}\ar@{->}[d]_{\bar{\pi}}\ar@{-->}[rr]^{\nu}&&X\ar@{->}[d]^{\pi}&\\
&\bar{S}&&S\ar@{-->}[ll]^{\zeta},&}
$$
where $\zeta$ is a~rational dominant map, or there is
$\lambda\in\mathbb{Q}$ such that
$$
K_{X}+\lambda\mathcal{M}_{X}\equiv \pi^{*}\big(H\big),
$$
where $H$ is a~$G$-invariant $\mathbb{Q}$-divisor on the~variety
$S$ such that
\begin{itemize}
\item either there is $i\in\{1,\ldots,k\}$ such that $a_{i}^{\lambda}<0$ and $F_{i}$ is not $\beta$-exceptional,%
\item or the~divisor $H$ is not $\mathbb{Q}$-effective.%
\end{itemize}
\end{lemma}

\begin{proof}
Arguing as in the~proof of
Lemma~\ref{lemma:rigidity-fibration-easy}, we may assume that
there is $\lambda\in\mathbb{Q}$ such that
$$
K_{X}+\lambda\mathcal{M}_{X}\equiv \pi^{*}\big(H\big),
$$
where $H$ is a~$G$-invariant $\mathbb{Q}$-divisor on the~variety
$S$.

Suppose that $H$ is $\mathbb{Q}$-effective, and
$a_{i}^{\lambda}\geqslant 0$ for every $i$ such that $F_{i}$ is
not $\beta$-exceptional. Then
$$
\beta\Bigg((\pi\circ\alpha)^{*}\big(H\big)+\sum_{i=1}^{k}a^{\lambda}_{i}F_{i}\Bigg)\equiv
\beta\Bigg((\pi\circ\alpha)^{*}\big(H\big)+\sum_{a_{i}^{\lambda}\geqslant
0}a^{\lambda}_{i}F_{i}\Bigg)\equiv
K_{\bar{X}}+\lambda\mathcal{M}_{\bar{X}},
$$
which implies that $K_{\bar{X}}+\lambda\mathcal{M}_{\bar{X}}$ is
$\mathbb{Q}$-effective, which is a~contradiction.
\end{proof}

The assertion of Lemma~\ref{lemma:rigidity-fibration} is an
analogue of \cite[Proposition~2]{Pu04d}.

\begin{lemma}
\label{lemma:rigidity-Fano} Suppose that
$\mathrm{dim}(\bar{S})=0$. Then there is $\lambda\in\mathbb{Q}$
such that
$$
K_{X}+\lambda\mathcal{M}_{X}\equiv \pi^{*}\big(H\big),
$$
where $H$ is a~$G$-invariant $\mathbb{Q}$-divisor on the~variety
$S$. Then
\begin{itemize}
\item either $\mathrm{dim}(S)=0$ and  $\nu$ is an isomorphism,%
\item or there is $i\in\{1,\ldots,k\}$ such that $a_{i}^{\lambda}<0$ and $F_{i}$ is not $\beta$-exceptional,%
\item or the~divisor $H$ is not $\mathbb{Q}$-effective.%
\end{itemize}
\end{lemma}

\begin{proof}
Recall that $\mathcal{M}_{\bar{X}}=|-mK_{\bar{X}}|$, where $m$ is
a~sufficiently big and sufficiently divisible natural number $m$.
Then $\mathcal{M}_{X}$ does not lie in the~fibers of $\pi$. Then
there is $\lambda\in\mathbb{Q}$ such that
$$
K_{X}+\lambda\mathcal{M}_{X}\equiv \pi^{*}\big(H\big),
$$
where $H$ is a~$G$-invariant $\mathbb{Q}$-divisor on the~variety
$S$. Then
$$
\mathrm{dim}\big(S\big)\geqslant\kappa\Big(X,\ \lambda
\mathcal{M}_{X}\Big)=\kappa\Big(\bar{X},\
\lambda\mathcal{M}_{\bar{X}}\Big)=\left\{\aligned
&\mathrm{dim}(\bar{X})\ \text{if}\ \lambda>1/m,\\
&0\ \text{if}\ \lambda=1/m,\\
&-\infty\ \text{if}\ \mu<1/m,\\
\endaligned
\right.
$$
by Lemma~\ref{lemma:Kodaira-dimension-is-OK}. Therefore, we see
that $\lambda\leqslant 1/m$.

Suppose that $H$ is $\mathbb{Q}$-effective, and
$a_{i}^{\lambda}\geqslant 0$ for every $i$ such that $F_{i}$ is
not $\beta$-exceptional. Then
$$
\beta\Bigg((\pi\circ\alpha)^{*}\big(H\big)+\sum_{i=1}^{k}a^{\lambda}_{i}F_{i}\Bigg)\equiv
\beta\Bigg((\pi\circ\alpha)^{*}\big(H\big)+\sum_{a_{i}^{\lambda}\geqslant
0}a^{\lambda}_{i}F_{i}\Bigg)\equiv
K_{\bar{X}}+\lambda\mathcal{M}_{\bar{X}},
$$
which implies that $K_{\bar{X}}+\lambda\mathcal{M}_{\bar{X}}$ is
$\mathbb{Q}$-effective. Then $\lambda=1/m$. Thus, we see that
$$
\big(\pi\circ\alpha\big)^{*}\big(H\big)+\sum_{i=1}^{k}a^{\lambda}_{i}F_{i}\equiv  \sum_{i=1}^{r}b^{\lambda}_{i}E_{i}.%
$$

Suppose that $\mathrm{dim}(S)=0$. Then it follows from
\cite[Lemma~2.19]{Ko91} that
$$
\sum_{i=1}^{k}a^{\lambda}_{i}F_{i}=\sum_{i=1}^{r}b^{\lambda}_{i}E_{i},%
$$
which implies that $(X, \lambda\mathcal{M}_{X})$ is terminal. Then
$\nu$ is an isomorphism by
Theorem~\ref{theorem:canonical-model-is-unique}.

To complete the~proof, we may assume that $\mathrm{dim}(S)\ne 0$.
Then
$$
\mathrm{dim}_{\mathbb{Q}}\Big(\mathrm{Pic}^{G}\big(\bar{S}\big)\otimes\mathbb{Q}\Big)+1+k=\mathrm{dim}_{\mathbb{Q}}\Big(\mathrm{Pic}^{G}\big(X\big)\otimes\mathbb{Q}\Big)+k=\mathrm{dim}_{\mathbb{Q}}\Big(\mathrm{Pic}^{G}\big(W\big)\otimes\mathbb{Q}\Big)=1+r,
$$
which implies that $k\leqslant r-1$. For every
$i\in\{1,\ldots,r\}$, we see that
\begin{itemize}
\item either $E_{i}$ is an exceptional divisor of the~morphism $\alpha$,%
\item or $\alpha(E_{i})$ is a~divisor that lies in the~fibers of the~morphism $\pi$.%
\end{itemize}

Suppose that $F_{1}$ is not contracted by $\beta$. Then
$$
\Big\langle
F_{1},E_{1},\ldots,E_{r}\Big\rangle=\mathrm{Pic}^{G}\big(W\big)\otimes\mathbb{Q},
$$
where $\langle F_{1},E_{1},\ldots,E_{r}\rangle$ is a~linear span
of the~divisors $F_{1},E_{1},\ldots,E_{r}$. Then
$$
\mathrm{dim}_{\mathbb{Q}}\Bigg(\Big\langle
E_{1},\ldots,E_{r}\Big\rangle\bigcap \Big\langle
F_{1},\ldots,F_{k}\Big\rangle\Bigg)=k-1,
$$
because $F_{1},\ldots,F_{k}$ are linearly independent in
$\mathrm{Pic}^{G}(\bar{W})\otimes\mathbb{Q}$. Thus, we have
$$
\mathrm{dim}_{\mathbb{Q}}\Bigg(\Big\langle
\alpha\big(E_{1}\big),\ldots,\alpha\big(E_{r}\big)\Big\rangle\Bigg)=r-k+1=\mathrm{dim}_{\mathbb{Q}}\Big(\mathrm{Pic}^{G}\big(X\big)\otimes\mathbb{Q}\Big),
$$
where we assume that $\alpha(E_{i})=0$ if the~divisor $E_{i}$ is
contracted by $\alpha$. But
$$
\Big\langle
\alpha\big(E_{1}\big),\ldots,\alpha\big(E_{r}\big)\Big\rangle\ne\mathrm{Pic}^{G}\big(X\big)\otimes\mathbb{Q},
$$
because $\alpha(E_{i})$ lies in the~fibers of $\pi$ if $E_{i}$ is
not contracted by $\alpha$.

We see that $F_{1}$ is contracted by $\beta$. Similarly, we see
that all $F_{1},\ldots,F_{k}$ are contracted by $\beta$.

Without loss of generality we may assume that $F_{i}=E_{i}$ for
every $i\in\{1,\ldots,k\}$. Then
$$
\mathrm{dim}_{\mathbb{Q}}\Bigg(\Big\langle
\alpha\big(E_{k+1}\big),\ldots,\alpha\big(E_{r}\big)\Big\rangle\Bigg)=r-k=\mathrm{dim}_{\mathbb{Q}}\Big(\mathrm{Pic}^{G}\big(S\big)\otimes\mathbb{Q}\Big),
$$
and all divisors $\alpha(E_{k+1}),\ldots,\alpha(E_{r})$ lies in
the~fibers of $\pi$.

Let $M$ be a~general very ample divisor on $S$, and let $N$ be its
proper transform on $W$. Then
$$
\pi^{*}\big(M\big)\in\Big\langle\alpha\big(E_{k+1}\big),\ldots,\alpha\big(E_{r}\big)\Big\rangle,%
$$
and $N\sim(\pi\circ\alpha)^{*}(M)$. The divisor $N$ is not
contracted by $\beta$. But
$$
\big(\pi\circ\alpha\big)^{*}\big(M\big)\in\Big\langle E_{1},\ldots,E_{r}\Big\rangle,%
$$
which is a~contradiction. The obtained contradiction completes
the~proof.
\end{proof}

Suppose, in addition, that $\mathrm{dim}(S)=0$.

\begin{theorem}
\label{theorem:G-rigidity} The~following conditions are
equivalent:
\begin{itemize}
\item the~Fano variety  $X$ is $G$-birationally rigid,%

\item for every $G$-invariant linear system $\mathcal{M}$ on
the~variety $X$ that has no fixed components, there is a~
$G$-equivariant birational automorphism
$\xi\in\mathrm{Bir}^{G}(X)$ such that
$$
\Big(X,\lambda\,\xi\big(\mathcal{M}\big)\Big)
$$
has canonical singularities, where $\lambda\in\mathbb{Q}$ such
that $K_{X}+\lambda\,\xi(\mathcal{M})\equiv 0$.
\end{itemize}
\end{theorem}

\begin{proof}
See \cite[Theorem~1.26]{ChSh08c} (cf.
Lemma~\ref{lemma:rigidity-fibration} and
\ref{lemma:rigidity-Fano}).
\end{proof}

\begin{corollary}
\label{corollary:G-rigidity} The~following conditions are
equivalent:
\begin{itemize}
\item the~Fano variety  $X$ is $G$-birationally rigid,%

\item for every $G$-invariant linear system $\mathcal{M}$ on
the~variety $X$ that has no fixed components, there is a~
$G$-equivariant birational automorphism
$\xi\in\mathrm{Bir}^{G}(X)$ such that
$$
\kappa\Big(X,\ \lambda\,\xi\big(\mathcal{M}\big)\Big)\geqslant 0,
$$
where  $\lambda\in\mathbb{Q}$ such that
$K_{X}+\lambda\,\xi(\mathcal{M})\equiv 0$.
\end{itemize}
\end{corollary}

\begin{corollary}
\label{corollary:G-superrigidity} The~following conditions are
equivalent:
\begin{itemize}
\item the~Fano variety  $X$ is $G$-birationally superrigid,%

\item for every $G$-invariant linear system $\mathcal{M}$ on $X$
that has no fixed components, the~log pair
$$
\Big(X,\ \lambda\,\mathcal{M}\Big)
$$
has canonical singularities, where $\lambda\in\mathbb{Q}$ such
that $K_{X}+\lambda\,\mathcal{M}\equiv 0$.
\end{itemize}
\end{corollary}

Let us show how to apply
Corollary~\ref{corollary:G-superrigidity}.

\begin{lemma}
\label{lemma:orbits} Suppose that $X$ is a~smooth del Pezzo
surface such that
$$
\big|\Sigma\big|\geqslant K_{X}^{2}
$$
for any $G$-orbit $\Sigma\subset X$. Then $X$ is $G$-birationally
superrigid.
\end{lemma}

\begin{proof}
Suppose that $X$ is not $G$-birationally superrigid. Then it
follows from Corollary~\ref{corollary:G-superrigidity}~that there
is a~$G$-invariant linear system $\mathcal{M}$ on the~surface $X$
such that $\mathcal{M}$ does not have fixed curves, but the~log
pair $(X, \lambda\mathcal{M})$~is~not canonical at some point
$O\in X$, where $\lambda\in\mathbb{Q}$ such~that
$$
K_{X}+\lambda\mathcal{M}\equiv 0.
$$

Let $\Sigma$ be the~$G$-orbit of the~point $O$. Then
$\mathrm{mult}_{P}(\mathcal{M})>1/\lambda$ for every point $P\in
\Sigma$. Then
$$
\frac{K_{X}^{2}}{\lambda^{2}}=M_{1}\cdot M_{2}\geqslant\sum_{P\in\Sigma}\mathrm{mult}_{P}\Big(M_{1}\cdot M_{2}\Big)\geqslant\sum_{P\in\Sigma}\mathrm{mult}^{2}_{P}\big(\mathcal{M}\big)>\frac{|\Sigma|}{\lambda^{2}}\geqslant \frac{K_{X}^{2}}{\lambda^{2}},%
$$
where $M_{1}$ and $M_{2}$ are sufficiently general curves in
$\mathcal{M}$.
\end{proof}

Let us show how to apply Lemma~\ref{lemma:orbits}.

\begin{theorem}
\label{theorem:A6} Let $G$ be a finite subgroup in
$$
\mathrm{Aut}\big(\mathbb{P}^{2}\big)\cong\mathbb{PGL}\Big(3,\mathbb{C}\Big)
$$
such that $G\cong\mathbb{A}_{6}$. Then $\mathbb{P}^{2}$ is
$G$-birationally superrigid.
\end{theorem}

\begin{proof}
Let $\Sigma\subset \mathbb{P}^{2}$ be any~$G$-orbit. Then it
follows from \cite{YauYu93} that
$$
\big|\Sigma\big|\geqslant 12,
$$
which implies that $\mathbb{P}^{2}$ is $G$-birationally superrigid
by Lemma~\ref{lemma:orbits}.
\end{proof}

Let $\Gamma$ be a~subset in $\mathrm{Bir}^{G}(X)$.

\begin{definition}
\label{definition:untwisting} The subset $\Gamma$
untwists~all~$G$-maximal singularities if for every $G$-invariant
linear system $\mathcal{M}$~on the~variety $X$ that has no fixed
components, there is a~$\xi\in\Gamma$ such that~the~pair
$$
\Big(X,\ \lambda\,\xi\big(\mathcal{M}\big)\Big)
$$
has canonical singularities, where $\lambda$ is a~rational number
such that $K_{X}+\lambda\,\xi(\mathcal{M})\equiv 0$.
\end{definition}

Applying Lemma~\ref{lemma:rigidity-fibration} and
\ref{lemma:rigidity-Fano}, we obtain the~following corollary.

\begin{corollary}
\label{corollary:untwisting} Suppose that
$\Gamma$~untwists~all~$G$-maximal singularities. Then
\begin{itemize}
\item the~variety $X$ is $G$-birationally rigid,%
\item the~group $\mathrm{Bir}^{G}(X)$ is generated by $\Gamma$ and $\mathrm{Aut}^{G}(X)$.%
\end{itemize}
\end{corollary}

It follows from Theorem~\ref{theorem:G-rigidity} that
the~following conditions are equivalent:
\begin{itemize}
\item the~variety $X$ is $G$-birationally rigid,%
\item the~group $\mathrm{Bir}^{G}(X)$ untwists all $G$-maximal singularities.%
\end{itemize}

\begin{definition}
\label{definition:threshold} Global $G$-invariant log canonical
threshold of the~variety $X$ is the~number
$$
\mathrm{lct}\big(X,G\big)=\mathrm{sup}\left\{\lambda\in\mathbb{Q}\ \left|%
\aligned
&\text{the~log pair}\ \left(X, \lambda D\right)\ \text{has log canonical singularities}\\
&\text{for every $G$-invariant effective $\mathbb{Q}$-divisor}\ D\equiv -K_{X}\\
\endaligned\right.\right\}\geqslant 0.%
$$
\end{definition}

Let us show how to compute the~number $\mathrm{lct}(X,G)$ (see
\cite[Lemma~5.7]{Ch07b}).

\begin{theorem}
\label{theorem:quintic-A5} Let $X$ be a smooth del Pezzo surface
such that $K_{X}^{2}=5$. Then
$$
\mathrm{Aut}\big(X\big)\cong\mathbb{S}_{5},
$$
the~surface $X$ is $\mathbb{A}_{5}$-birationally superrigid, and
$\mathrm{lct}(X,\mathbb{A}_{5})=2$, where
$\mathrm{Pic}^{\mathbb{A}_{5}}(X)\cong\mathbb{Z}$.
\end{theorem}

\begin{proof}
The isomorphisms $\mathrm{Aut}(X)\cong\mathbb{S}_{5}$ and
$\mathrm{Pic}^{\mathbb{A}_{5}}(X)\cong\mathbb{Z}$ are well-known
(see \cite{RaSl02}, \cite{DoIs06}).

Let $\Sigma\subset X$ be any~$\mathbb{A}_{5}$-orbit. Suppose that
$|\Sigma|\leqslant 6$. Then $|\Sigma|\geqslant 5$, because
$\mathbb{A}_{5}$ is simple.

Let $P\in\Sigma$ be a~point, let $H\subset\mathbb{A}_{5}$ be
the~stabilizer of the~point $P$. Then
$$
\big|H\big|=\frac{|\mathbb{A}_{5}|}{|\Sigma|}=\frac{60}{|\Sigma|},
$$
and $H$ acts faithfully on the~tangent space of the~surface $X$ at
the~point $P$. Then
$$
\big|\Sigma\big|\ne 5,
$$
because $\mathbb{A}_{4}$ does not have faithful
two-dimensional representations.

We see that $|\Sigma|=6$. In particular, the~surface $X$ is
$G$-birationally superrigid~by~Lemma~\ref{lemma:orbits}, which
also follows from either from \cite{Blanc} or \cite{BaTo07}.

The surface $X$ contains $10$ smooth rational curves
$L_{1},L_{2},\ldots,L_{10}$ such that
$$
L_{1}\cdot L_{1}=L_{2}\cdot L_{2}=\cdots=L_{10}\cdot L_{10}=-1,
$$
and $\sum_{i=1}^{10}L_{i}\sim -2K_{X}$. Then
$\mathrm{lct}(X,\mathbb{A}_{5})\leqslant 2$, because the~divisor
$\sum_{i=1}^{10}L_{i}$ is $\mathbb{A}_{5}$-invariant.

It follows from \cite{Edge81a} that there is a birational morphism
$\chi\colon X\to \mathbb{P}^{2}$ that is a~blow~up~of~the~four
singular points of the~curve $W\subset\mathbb{P}^{2}$ that is
given by the~equation
$$
x^{6}+y^{6}+z^{6}+\Big(x^{2}+y^{2}+z^{2}\Big)\Big(x^{4}+y^{4}+z^{4}\Big)=12x^{2}y^{2}z^{2}\subset\mathbb{P}^{2}\cong\mathrm{Proj}\Big(\mathbb{C}[x,y,z]\Big).
$$

Let $Z$ be the~proper transform on $X$ of the~curve $W$. Then
$$
Z\sim -2K_{X},
$$
and the~curves $Z$ and $\sum_{i=1}^{10}L_{i}$ are the~only
$\mathbb{S}_{5}$-invariant curves in $|-2K_{X}|$.

Let $\mathcal{P}$ be the~pencil on $X$ that is generated by
the~curves $Z$ and $\sum_{i=1}^{10}L_{i}$. Then every curve~in
the~pencil $\mathcal{P}$ is $\mathbb{A}_{5}$-inva\-riant (see
\cite{Edge81a}).

Let $T$ be a curve in $\mathcal{P}$ such that $\Sigma\cap
T\ne\varnothing$. Then
$$
\Sigma\subset \mathrm{Sing}\big(T\big),
$$
because $H$ is not abelian and, hence, does not have faithful
one-dimensional representations.

It follows from \cite{Edge81a} that $\mathcal{P}$ contains five
singular curves, which are
\begin{itemize}
\item the~curve $\sum_{i=1}^{10}L_{i}$,%
\item two irreducible rational curves $R_{1}$ and $R_{2}$ that have $6$ nodes,%
\item two reduced curves $F_{1}$ and $F_{2}$ each consisting of $5$ smooth rational curves.%
\end{itemize}

Therefore, we see that $\Sigma=\mathrm{Sing}(R_{1})$ or
$\Sigma=\mathrm{Sing}(R_{2})$.

The representation induced by the~action of $\mathbb{A}_{5}$ on
$|-K_{X}|$ is~a~sum of two non-equivalent irreducible
three-dimen\-si\-o\-nal representations of $\mathbb{A}_{5}$ (see
\cite{RaSl02}), which induces two $\mathbb{A}_{5}$-equivariant
projections $\phi\colon X\dasharrow\mathbb{P}^{2}$ and~$\psi\colon
X\dasharrow\mathbb{P}^{2}$, respectively. It follows from
\cite{RaSl02} that
\begin{itemize}
\item the~maps $\phi$ and $\psi$ are are morphisms of degree $5$,%

\item the~actions of $\mathbb{A}_{5}$ on $\mathbb{P}^{2}$ induced
by $\phi$ and $\psi$ are induced by non-equivalent irreducible
three-dimen\-si\-o\-nal representations of $\mathbb{A}_{5}$ (cf.
the~proof of Lemma~\ref{lemma:A5-22}), respectively.
\end{itemize}

Suppose that $\mathrm{lct}(X,\mathbb{A}_{5})\ne 2$. There is
an~effective $\mathbb{A}_{5}$-invariant $\mathbb{Q}$-divisor $D$
on $X$ such~that
$$
\mathbb{LCS}\Big(X,\lambda D\Big)\ne\varnothing
$$
and $D\equiv-K_{X}$, where $\lambda$ is a positive rational number
such that $\lambda<2$. Then
\begin{equation}
\label{equation:vanishing}
H^{1}\Bigg(X,\ \mathcal{O}_{X}\Big(-K_{X}\Big)\otimes\mathcal{I}\Big(X,\lambda D\Big)\Bigg)=0%
\end{equation}
by Theorem~\ref{theorem:Shokurov-vanishing}, where
$\mathcal{I}(X,\lambda D)$ is the~multiplier ideal sheaf (see
Definition~\ref{definition:log-canonical-subscheme}).

Suppose that there is a $\mathbb{A}_{5}$-invariant reduced curve
$C$ on the~surface $X$ such that
$$
\lambda D=\mu C+\Omega,
$$
where $\mu$ is a~positive rational number such that $\mu\geqslant
1$, and $\Omega$ is an effective $\mathbb{Q}$-divisor on
the~surface $X$, whose support does not contain any component of
the~curve $C$. Then
$$
C\sim -mK_{X},
$$
where $m$ is a natural number. Therefore, we have
$$
5=-K_{X}\cdot D\geqslant\frac{5\mu}{\lambda}>\frac{5m}{2},%
$$
which is implies that $m=1$. But the~equality $m=1$ is impossible.

Therefore, we see that $\mathbb{LCS}(X,\lambda D)$ does not
contain curves (see
Definition~\ref{definition:log-canonical-center}).

It follows from the~vanishing~\ref{equation:vanishing} that
$$
\Big|\mathrm{LCS}\Big(X,\lambda D\Big)\Big|\leqslant 6,%
$$
but the~set $\mathrm{LCS}(X,\lambda D)$ is
$\mathbb{A}_{5}$-invariant. Then $|\mathrm{LCS}(X,\lambda D)|=6$.

Thus, we see that either $\mathrm{LCS}(X,\lambda
D)=\mathrm{Sing}(R_{1})$ or $\mathrm{LCS}(X,\lambda
D)=\mathrm{Sing}(R_{2})$.

We may assume that $\mathrm{Supp}(D)$ does not contain $R_{1}$ and
$R_{2}$ by Remark~\ref{remark:convexity}.

We may assume that $\mathrm{LCS}(X,\lambda
D)=\mathrm{Sing}(R_{1})=\Sigma$. Put
$$
\mathrm{Sing}\big(R_{1}\big)=\Big\{O_{1},O_{2},O_{3},O_{4},O_{5},O_{6}\Big\},
$$
where $O_{i}$ is a~singular point of the~curve~$R_{1}$. We may
assume that $P=O_{1}$. Then
$$
10=R_{1}\cdot D\geqslant\sum_{i=1}^{6}\mathrm{mult}_{O_{i}}\big(D\big)\mathrm{mult}_{O_{i}}\big(R_{1}\big)\geqslant 12\mathrm{mult}_{P}\big(D\big).%
$$
which implies that $\mathrm{mult}_{P}(D)\leqslant 5/6$.

Let $\pi\colon U\to X$ be a blow up of the~points
$O_{1},O_{2},O_{3},O_{4},O_{5},O_{6}$.~Then
$$
K_{U}+\lambda\bar{D}+\sum_{i=1}^{6}\Big(\lambda\mathrm{mult}_{O_{i}}\big(D\big)-1\Big)E_{i}\equiv \pi^{*}\Big(K_{X}+\lambda D\Big),%
$$
where $E_{i}$ is the~exceptional curve of the~blow up $\pi$ such
that $\pi(E_{i})=O_{i}$, and $\bar{D}$ is the~proper transform
of~the~divisor $D$~on the~surface $U$. Then the~log pair
$$
\Bigg(U,\lambda\bar{D}+\sum_{i=1}^{6}\Big(\lambda\mathrm{mult}_{O_{i}}\big(D\big)-1\Big)E_{i}\Bigg)%
$$
is not log canonical at some point $Q\in E_{1}$. Then it follows
from Theorem~\ref{theorem:connectedness} that
$$
\mathrm{LCS}\Bigg(U,\lambda\bar{D}+\sum_{i=1}^{6}\Big(\lambda\mathrm{mult}_{O_{i}}\big(D\big)-1\Big)E_{i}\Bigg)\bigcap E_{1}=Q,%
$$
which is impossible, because the~$\mathbb{A}_{5}$-orbit of
the~point $Q$ contains at least $2$ points of the~exceptional
curve $E_{1}$, because $H$ acts faithfully on the~tangent space of
$X$ at the~point $P$.
\end{proof}

The number $\mathrm{lct}(X,G)$ plays an~important role in geometry
(see Theorem~\ref{theorem:KE}). We put
$$
\mathrm{lct}\big(X\big)=\mathrm{lct}\big(X,G\big)
$$
if the~group $G$ is trivial. Then
$\mathrm{lct}(\mathbb{P}^{n})=1/(n+1)$ (see \cite{ChSh08c}).

\begin{example}
\label{example:hypersurfaces-lct} Let $X$ be a general
hypersurface in $\mathbb{P}^{n}$ of degree $n\geqslant 6$. Then
$\mathrm{lct}(X)=1$~by~\cite{Pu04d}.
\end{example}

Note that Definition~\ref{definition:rigidity} still make sense if
$X$ is defined over a~perfect field.

\begin{definition}
\label{definition:birational-rigidity-universal} The variety $X$
is said to be universally $G$-birationally rigid if the~variety
$$
X\otimes\mathrm{Spec}\Big(\mathbb{C}\big(U\big)\Big)
$$
is $G$-birationally rigid for any variety $U$, where
$\mathbb{C}(U)$ is the~field of rational functions
of~the~variety~$U$, and we identify $G$ with a subgroup of
the~group $\mathrm{Aut}(X\otimes\mathrm{Spec}(\mathbb{C}(U)))$.
\end{definition}

It follows from Definition~\ref{definition:superrigidity} that if
the~variety $X$ is $G$-birationally superrigid, then
\begin{itemize}
\item the~variety $X$ is universally $G$-birationally rigid,%

\item for any variety $U$, the~variety
$$
X\otimes\mathrm{Spec}\Big(\mathbb{C}\big(U\big)\Big)
$$
is $G$-birationally superrigid, where $\mathbb{C}(U)$ is the~field
of rational functions of the~variety~$U$, and we identify $G$ with
a subgroup of the~group
$\mathrm{Aut}(X\otimes\mathrm{Spec}(\mathbb{C}(U)))$.
\end{itemize}

\begin{remark}
\label{remark:Kollar} Suppose that $\mathrm{dim}(X)\ne 1$ and
$X$~is~$G$-bira\-ti\-onally rigid. Then
\begin{itemize}
\item the~group $\mathrm{Aut}^{G}(X)$ is finite (see \cite{Ros56}),%
\item the~variety $X$ is universally $G$-birationally rigid if $\mathrm{Bir}^{G}(X)$ is countable (see \cite{Ko08}).%
\end{itemize}
\end{remark}

Take $r\in\mathbb{N}$ such that $r\geqslant 2$. For every
$i\in\{1,2,\ldots,r\}$, let $X_{i}$ be a~Fano variety, and let
$G_{i}$ be a finite subgroup in $\mathrm{Aut}(X_{i})$ such that
$$
\mathrm{dim}_{\mathbb{Q}}\Bigg(\mathrm{Pic}^{G_{i}}\big(X_{i}\big)\otimes\mathbb{Q}\Bigg)=1,
$$
and every $G_{i}$-invariant Weil divisor on the~variety $X_{i}$ is
a~$\mathbb{Q}$-Cartier divisor. Suppose, in addition, that $X_{i}$
has~at most terminal singularities, and the~variety $X_{i}$ is
$G_{i}$-birationally rigid. Put
\begin{itemize}
\item $\mathcal{X}=X_{1}\times X_{2}\times\cdots\times X_{r}$ and $\mathcal{G}=G_{1}\times G_{2}\times\cdots\times G_{r}$,%
\item $\mathcal{S}_{1}=X_{2}\times\cdots\times X_{r}$ and $\mathcal{S}_{r}=X_{1}\times\cdots\times X_{r}$,%
\item $\mathcal{S}_{i}=X_{1}\times\cdots\times X_{i-1}\times X_{i+1}\times\cdots\times X_{r}$ for every $i\in\{2,\ldots,r-1\}$.%
\end{itemize}

Let $\pi_{i}\colon \mathcal{X}\to\mathcal{S}_{i}$ be a~natural
projection. Then $\pi_{i}$ is a~$\mathcal{G}$-Mori fibration and
$$
\big\{\pi_{1},\ldots,\pi_{r}\big\}\subseteq \mathcal{P}\big(\mathcal{X},\mathcal{G}\big),%
$$
where $\mathcal{P}(\mathcal{X},\mathcal{G})$ is
the~$\mathcal{G}$-pliability of the~variety $\mathcal{X}$ (see
Definition~\ref{definition:pliablity}).

\begin{theorem}
\label{theorem:II-Pukhlikov} For every $i\in\{1,2,\ldots,r\}$,
suppose that $X_{i}$ is universally $G_{i}$-birationally rigid and
the~inequality $\mathrm{lct}(X_{i}, G_{i})\geqslant 1$ holds.~Then
$\mathcal{P}(\mathcal{X},\mathcal{G})=\{\pi_{1},\ldots,\pi_{r}\}$.
\end{theorem}

\begin{proof}
Suppose that there is a~$\mathcal{G}$-Mori fibration
$\bar{\pi}\colon\bar{X}\to\bar{S}$ such that there is
a~$\mathcal{G}$-equivariant birational~map
$\nu\colon\bar{X}\dasharrow\mathcal{X}$. We must show that there
is $i\in\{1,2,\ldots,r\}$ such that $\mathrm{dim}(\bar{S})\ne 0$,
and there exists a commutative diagram
\begin{equation}
\label{equation:pliavility}
\xymatrix{
&\bar{X}\ar@{->}[d]_{\bar{\pi}}\ar@{-->}[rr]^{\nu}&&\mathcal{X}\ar@{-->}[rr]^{\rho}&&\mathcal{X}\ar@{->}[d]^{\pi_{i}}&\\
&\bar{S}\ar@{-->}[rrrr]^{\sigma}&&&&\mathcal{S}_{i},&}
\end{equation}
where $\rho$ is a~$\mathcal{G}$-equivariant birational
automorphism, and $\sigma$ is a birational map.

Arguing as in the~proof of \cite[Theorem~1]{Pu04d} and
\cite[Theorem~6.5]{Ch07a} and using
Lemma~\ref{lemma:rigidity-Fano}~instead of using
\cite[Proposition~2]{Pu04d}, we see that $\mathrm{dim}(\bar{S})\ne
0$ (cf. the~proof of Theorem~\ref{theorem:A5-A5}).

Arguing as in the~proof of \cite[Theorem~1]{Pu04d} and
\cite[Theorem~6.5]{Ch07a} and using
 Lemma~\ref{lemma:rigidity-fibration}~instead of using
\cite[Proposition~2]{Pu04d}, we obtain the~existence of
the~commutative diagram~\ref{equation:pliavility}.
\end{proof}

Let us show how to apply Theorem~\ref{theorem:II-Pukhlikov} (cf.
Examples~\ref{example:hypersurfaces} and
\ref{example:hypersurfaces-lct}).

\begin{example}
\label{example:Valentiner} The simple group $\mathbb{A}_{6}$ is a~
group of automorphisms of the~sextic
$$
10x^{3}y^{3}+9zx^{5}+9zy^{5}+27z^{6}=45x^{2}y^{2}z^{2}+135xyz^{4}\subset\mathbb{P}^{2}\cong\mathrm{Proj}\Big(\mathbb{C}\big[x,y,z\big]\Big),
$$
which induces a~monomorphism
$\mathbb{A}_{6}\times\mathbb{A}_{6}\to\mathrm{Aut}(\mathbb{P}^{2}\times\mathbb{P}^{2})$,
which induces a~monomorphism
$$
\phi\colon\mathbb{A}_{6}\times\mathbb{A}_{6}\longrightarrow\mathrm{Bir}\big(\mathbb{P}^{4}\big),%
$$
because the~variety $\mathbb{P}^{2}\times\mathbb{P}^{2}$ is
rational. Then it immediately follows from
Theorem~\ref{theorem:II-Pukhlikov} that the~subroup
$\phi(\mathbb{A}_{6}\times\mathbb{A}_{6})$ is not conjugate to
any~subgroup in $\mathrm{Aut}(\mathbb{P}^{4})$,~because
\begin{itemize}
\item the~surface $\mathbb{P}^{2}$ is $\mathbb{A}_{6}$-birationally superrigid by Theorem~\ref{theorem:A6},%
\item the~equality $\mathrm{lct}(\mathbb{P}^{2},\mathbb{A}_{6})=2$ holds (see \cite{Ch07b}).%
\end{itemize}
\end{example}

For further applications of Theorem~\ref{theorem:II-Pukhlikov} see
\cite{Pu04d} and \cite{Ch07a}.

\section{Cremona group}
\label{section:simple-subgroups}

Let $G$ be a~finite group. Put
$\mathrm{Cr}_{2}(\mathbb{C})=\mathrm{Bir}(\mathbb{P}^{2})$. Let
$\phi\colon G\to\mathrm{Cr}_{2}(\mathbb{C})$ be a~monomorphism.

\begin{problem}
\label{problem:1} Up to conjugations, find all subgroups in
$\mathrm{Cr}_{2}(\mathbb{C})$ that are isomorphic to $G$.
\end{problem}

The purpose of this section is to solve Problem~\ref{problem:1} in
the~case when
$$
G\in\Big\{\mathbb{A}_{5},\mathbb{A}_{6},\mathbb{PSL}\Big(2,\mathbb{F}_{7}\Big)\Big\},
$$
which solves Problem~\ref{problem:1} for all simple non-abelian
subgroups in $\mathrm{Cr}_{2}(\mathbb{C})$ (see
Theorem~\ref{theorem:simple-subgroups}).

\begin{theorem}
\label{theorem:conjugacy-classes} Up to conjugations, the~group
$\mathrm{Cr}_{2}(\mathbb{C})$ contains $3$, $1$, $2$ subgroups
somorphic~to
$$
\mathbb{A}_{5},\mathbb{A}_{6},\mathbb{PSL}\Big(2,\mathbb{F}_{7}\Big),
$$
respectively. The description of this subgroups follows
Theorems~\ref{theorem:Cremona-Klein},
\ref{theorem:Cremona-Valentiner},
\ref{theorem:Cremona-Icosahedron}.
\end{theorem}

\begin{proof}
The required assertion (probably known) follows  from
Theorems~\ref{theorem:Cremona-Klein},
\ref{theorem:Cremona-Valentiner},
\ref{theorem:Cremona-Icosahedron}.
\end{proof}

Let $\phi^{\prime}\colon G\to\mathrm{Cr}_{2}(\mathbb{C})$ be
a~monomorphism. If $\phi(G)=\phi^{\prime}(G)$, then
$\phi^{\prime}=\phi\circ\chi$, where $\chi\in\mathrm{Aut}(G)$.

\begin{definition}
\label{definition:conjugate} We say that the~pairs $(G,\phi)$ and
$(G,\phi^{\prime})$ are conjugate if there exists a~birational
automorphism $\epsilon\in\mathrm{Cr}_{2}(\mathbb{C})$ such that
there is a~commutative diagram
$$
\xymatrix{
&&G\ar@{->}[dl]_{\phi}\ar@{->}[dr]^{\phi^{\prime}}&&\\
&\mathrm{Bir}\big(\mathbb{P}^{2}\big)\ar@{->}[rr]_{\omega_{\epsilon}}&&\mathrm{Bir}\big(\mathbb{P}^{2}\big),&}
$$
where $\omega_{\epsilon}$ is an~inner isomorphism such that
$\omega_{\epsilon}(g)=\epsilon \circ g\circ\epsilon^{-1}$ for
every $g\in\mathrm{Cr}_{2}(\mathbb{C})$.
\end{definition}

If $\phi^{\prime}=\phi\circ\chi$ for some inner isomorphism
$\chi\in\mathrm{Aut}(G)$, then $(G,\phi)$ and $(G,\phi^{\prime})$
are conjugate.

\begin{problem}
\label{problem:2} For every $\sigma\in\mathrm{Aut}(G)$, decide
whether $(G,\phi)$ and $(G,\phi\circ\sigma)$ are conjugate~or~not.
\end{problem}

It follows from \cite[Lemma~3.5]{DoIs06} that we can find a~smooth
rational surface $X$, a~mono\-mor\-phism $\upsilon\colon G\to
\mathrm{Aut}(X)$, and a~birational map $\xi\colon
X\dasharrow\nolinebreak \mathbb{P}^2$ such that
$$
\phi\big(g\big)=\xi\circ\upsilon\big(g\big)\circ
\xi^{-1}\in\mathrm{Bir}\big(\mathbb{P}^{2}\big)
$$
for every element $g\in G$. The triple $(X,\xi,\upsilon)$ is not
uniquely defined by the~pair $(G,\phi)$.

\begin{definition}
\label{definition:regularization} We say that the~triple
$(X,\xi,\upsilon)$ is a~regularization of the~pair $(G,\phi)$.
\end{definition}

Let $(X^{\prime},\xi^{\prime},\upsilon^{\prime})$ be a~
regularization of the~pair $(G,\phi^{\prime})$.

\begin{theorem}
\label{theorem:regularization} The~following assertions are
equivalent:
\begin{itemize}
\item the~pairs $(G,\phi)$ and $(G,\phi^{\prime})$ are conjugate,%

\item there is a~birational map $\rho\colon X\dasharrow
X^{\prime}$ such that
$$
\upsilon^{\prime}\big(g\big)=\rho\circ\upsilon\big(g\big)\circ
\rho^{-1}\in\mathrm{Aut}\big(X^{\prime}\big)
$$
for every element $g\in G$.
\end{itemize}
\end{theorem}

\begin{proof}
See \cite[Lemma~3.4]{DoIs06}.
\end{proof}

It follows from Theorem~\ref{theorem:regularization} that, to
solve Problems~\ref{problem:1} and \ref{problem:2}, we may assume
that there~is a~morphism $\pi\colon X\to S$ that is a~
$\upsilon(G)$-Mori fibration. Then either $S\cong\mathbb{P}^{1}$
or $S$ is a~point.

\begin{theorem}
\label{theorem:simple-subgroups} Suppose that $G$ be a~simple
nonabelian group. Then
$$
G\in \Big\{\mathbb{A}_{5},\mathbb{A}_{6},\mathbb{PSL}\Big(2,\mathbb{F}_{7}\Big)\Big\}.%
$$
\end{theorem}

\begin{proof}
The required assertion follows from \cite{DoIs06}. Indeed, suppose
that $S\cong\mathbb{P}^{1}$. Then
$$
G\subset\mathrm{Aut}\big(\mathbb{P}^{1}\big)\cong\mathbb{PGL}\Big(2,\mathbb{C}\Big),
$$
because $\upsilon(G)$ acts nontrivially either on the~fiber of
$\pi$ or on $\pi(X)\cong\mathbb{P}^{1}$. Then
$G\cong\mathbb{A}_{5}$ (see \cite{Spr77}).

We may assume that $S$ is a~point. Then either
$X\cong\mathbb{P}^{2}$ or $K_{X}^{2}\in\{2,5\}$ (see
\cite{DoIs06}).

In the~case when $K_{X}^{2}=5$, we have
$\mathrm{Aut}(X)\cong\mathbb{S}_{5}$ and $G\cong\mathbb{A}_{5}$
(see \cite{DoIs06}). Similarly, we have
$$
G\cong\mathbb{PSL}\Big(2,\mathbb{F}_{7}\Big)
$$
in the~case when $K_{X}^{2}=2$ (see \cite{DoIs06}). In the~case
when $X\cong\mathbb{P}^{2}$, it is well-known (see \cite{DoIs06})
that
$$
G\in\Big\{\mathbb{A}_{5},\mathbb{A}_{6},\mathbb{PSL}\Big(2,\mathbb{F}_{7}\Big)\Big\},
$$
which completes the~proof.
\end{proof}

Let us identify the~subgroup $\upsilon(G)\subset\mathrm{Aut}(X)$
with the~group $G$.

\begin{theorem}
\label{theorem:Cremona-Klein} Suppose that
$G\cong\mathbb{PSL}(2,\mathbb{F}_{7})$. Then
\begin{itemize}
\item the~variety $S$ is a~point,%

\item the~surface $X$ is $G$-birationally superrigid,%

\item one of the~following two possibilities holds:
\begin{itemize}
\item $X\cong\mathbb{P}^{2}$, and $G$ is conjugate to a~subgroup
that leaves invariant the~curve
\begin{equation}
\label{equation:quartic-curve}
x^{3}y+y^{3}z+z^{3}x=0\subset\mathrm{Proj}\Big(\mathbb{C}\big[x,y,z\big]\Big)\cong\mathbb{P}^{2},
\end{equation}

\item $X$ is a~double cover of $\mathbb{P}^{2}$ branched along
the~curve~\ref{equation:quartic-curve},
\end{itemize}
\item for every  $\sigma\in\mathrm{Aut}(G)$, the~following
conditions are equivalent:
\begin{itemize}
\item the~pairs $(G,\phi)$ and $(G,\phi\circ\sigma)$ are conjugate,%
\item the~isomorphism $\sigma$ is an inner isomorphism.%
\end{itemize}
\end{itemize}
\end{theorem}

\begin{proof}
Arguing as in the~proof of Theorem~\ref{theorem:simple-subgroups},
we see that
\begin{itemize}
\item the~variety $S$ is a~point,%

\item one of the~following two possibilities holds:
\begin{itemize}
\item $X\cong\mathbb{P}^{2}$, and $G$ is conjugate to a~subgroup
that leaves invariant the~curve~\ref{equation:quartic-curve},

\item $X$ is a~double cover of $\mathbb{P}^{2}$ branched along
the~curve~\ref{equation:quartic-curve}.
\end{itemize}
\end{itemize}

Any $G$-orbit in $X$ consists of least $2$ points. Indeed, the~
group $G$ is simple, which implies that
\begin{itemize}
\item the~group $G$ does not have faithful two-dimensional representations,%
\item the~group $G$ does not have subgroups of index two.%
\end{itemize}

If $X\cong\mathbb{P}^{2}$, then any $G$-orbit in $X$ contains at
least $12$ points (see \cite{Spr77} or \cite{YauYu93}).

It follows from Lemma~\ref{lemma:orbits} that $X$ is
$G$-birationally superrigid.

Take any $\sigma\in\mathrm{Aut}(G)$ such that the~pairs $(G,\phi)$
and $(G,\phi\circ\sigma)$ are conjugate. Then
there~is~a~birational map $\rho\colon X\dasharrow X$ such that
$$
\upsilon\circ\sigma\big(g\big)=\rho\circ\upsilon\big(g\big)\circ\rho^{-1}%
$$
for every $g\in G$ by Theorem~\ref{theorem:regularization}. Then
$\rho\in\mathrm{Aut}(X)$, because $X$ is $G$-birationally
superrigid.~Put
$$
\hat{G}=\big\langle G,\rho\big\rangle,
$$
where we identify the~subgroup $\upsilon(G)$ with the~group $G$.
Then $\hat{G}$ is a finite subgroup in $\mathrm{Aut}^{G}(X)$,
which implies that $\hat{G}=G$, because $\mathrm{Aut}^{G}(X)=G$ if
$X\cong\mathbb{P}^{2}$, and
$$
\mathrm{Aut}\big(X\big)=\mathrm{Aut}^{G}\big(X\big)\cong\mathbb{PSL}\Big(2,\mathbb{F}_{7}\Big)\times\mathbb{Z}_{2}
$$
if $X\not\cong\mathbb{P}^{2}$. Then $\rho\in G$, which means that
$\sigma$ is an~inner isomorphism of the~group $G$.
\end{proof}

\begin{theorem}
\label{theorem:Cremona-Valentiner} Suppose that
$G\cong\mathbb{A}_{6}$. Then
\begin{itemize}
\item $X\cong\mathbb{P}^{2}$ and $X$ is $G$-birationally superrigid,%

\item the~subgroup $G$ is conjugate to a~subgroup that leaves
invariant the~curve
\begin{equation}
\label{equation:sextic-curve}
10x^{3}y^{3}+9zx^{5}+9zy^{5}+27z^{6}=45x^{2}y^{2}z^{2}+135xyz^{4}\subset\mathbb{P}^{2}\cong\mathrm{Proj}\Big(\mathbb{C}\big[x,y,z\big]\Big),
\end{equation}
\item for every $\sigma\in\mathrm{Aut}(G)$, the~following
conditions are equivalent:
\begin{itemize}
\item the~pairs $(G,\phi)$ and $(G,\phi\circ\sigma)$ are conjugate,%
\item the~isomorphism $\sigma$ is an inner isomorphism.%
\end{itemize}
\end{itemize}
\end{theorem}

\begin{proof}
Arguing as in the~proof of Theorem~\ref{theorem:Cremona-Klein}, we
see that
$$
X\cong\mathbb{P}^{2},
$$
and $G$ is conjugate to a~subgroup that leaves invariant
the~curve~\ref{equation:sextic-curve}.

It follows from Theorem~\ref{theorem:A6} that $X$ is
$G$-birationally superrigid.

Let $C$ be the~curve~\ref{equation:quartic-curve}. Then we may
assume that $C$ is $G$-invariant.

Let $\rho$ be any element in $\mathrm{Aut}^{G}(X)$, and let $g$ be
any element in $G$. Then
$$
g\Big(\rho\big(C\big)\Big)=\rho\Big(g^{\prime}\big(C\big)\Big)=\rho\big(C\big)%
$$
for some $g^{\prime}\in G$. But $C$ is the~only sextic curve in
$\mathbb{P}^{2}$ that is $G$-invariant. Then
$$
\rho\big(C\big)=C,
$$
which implies that $\rho\in G$, because $G\cong\mathrm{Aut}(C)$.

Now arguing as in the~proof of
Theorem~\ref{theorem:Cremona-Klein}, we conclude the~proof.
\end{proof}

\begin{theorem}
\label{theorem:Cremona-Icosahedron} Suppose that
$G\cong\mathbb{A}_{5}$. Then
\begin{itemize}
\item for every $\sigma\in\mathrm{Aut}(G)$, the~pairs $(G,\phi)$ and $(G,\phi\circ\sigma)$ are conjugate,%
\item one of the~following possibilities holds:
\begin{itemize}
\item $X$ is a~blow up of $\mathbb{P}^{2}$ at any four points in
general~position, the~variety $S$ is a~point, the~surface
$X$ is $G$-birationally superrigid,~and $\mathrm{Aut}(X)\cong\mathbb{S}_{5}$,%

\item $X\cong\mathbb{P}^{2}$, the~variety $X$ is $G$-birationally
rigid, and
$$
\mathbb{A}_{5}\cong\mathrm{Aut}^{\mathbb{A}_{5}}\Big(\mathbb{P}^{2}\Big)\subsetneq\mathrm{Bir}^{\mathbb{A}_{5}}\Big(\mathbb{P}^{2}\Big)\cong\mathbb{S}_{5},
$$
\item $X\cong\mathbb{F}_{n}$, where $n\in\mathbb{N}\cup\{0\}$ and
$n$ is even, there is a~birational map $\rho\colon
X\dasharrow\mathbb{P}^{1}\times\mathbb{P}^{1}$ that induces
the~monomorphism $\bar{\upsilon}\colon
G\to\mathrm{Aut}(\mathbb{P}^{1}\times\mathbb{P}^{1})$ such that
$$
\bar{\upsilon}\big(g\big)=\rho\circ\upsilon\big(g\big)\circ\rho^{-1}
$$
for every $g\in G$, and $\bar{\upsilon}$ is induced by
the~product action of $\mathbb{A}_{5}\times \mathrm{id}_{\mathbb{P}^{1}}$ on $\mathbb{P}^{1}\times\mathbb{P}^{1}$.%
\end{itemize}
\end{itemize}
\end{theorem}

Note that the~assertions of Theorems~\ref{theorem:Cremona-Klein},
\ref{theorem:Cremona-Valentiner} and
\ref{theorem:Cremona-Icosahedron} solves Problems~\ref{problem:1}
and \ref{problem:2} for
$$
G\in\Big\{\mathbb{A}_{5},\mathbb{A}_{6},\mathbb{PSL}\Big(2,\mathbb{F}_{7}\Big)\Big\}.
$$

In the~rest of this section we prove
Theorem~\ref{theorem:Cremona-Icosahedron}. Suppose that
$G\cong\mathbb{A}_{5}$.

\begin{lemma}
\label{label:A5-1} One of the~following possibilities~holds:
\begin{itemize}
\item the~variety $S$ is a point and
\begin{itemize}
\item either $K_{X}^{2}=5$,
\item or $X\cong\mathbb{P}^{2}$,%
\end{itemize}
\item $S\cong\mathbb{P}^{1}$ and $X\cong\mathbb{F}_{n}$, where $n\in\mathbb{N}\cup\{0\}$.%
\end{itemize}
\end{lemma}

\begin{proof}
If the~variety $S$ is a point, then either $K_{X}^{2}=5$ or
$X\cong\mathbb{P}^{2}$ (see \cite{DoIs06}). If
$S\cong\mathbb{P}^{1}$, then
$$
X\cong\mathbb{F}_{n}
$$
by \cite[Lemma~5.6]{DoIs06}, where
$n\in\mathbb{N}\cup\{0\}$.
\end{proof}

Note that if the~equality $K_{X}^{2}=5$ holds and  the~variety $S$
is a~point, then
$$
\mathrm{Aut}\big(X\big)\cong\mathbb{S}_{5}
$$
and $X$ is a~blow up of $\mathbb{P}^{2}$ at four points in
general~position  (see \cite{RaSl02}, \cite{DoIs06}).

\begin{lemma}
\label{lemma:A5-2} Suppose that $K_{X}^{2}=5$ and  $S$ is a~point.
Then
\begin{itemize}
\item for every $\sigma\in\mathrm{Aut}(G)$, the~pairs $(G,\phi)$ and $(G,\phi\circ\sigma)$ are conjugate,%
\item the~surface $X$ is $G$-birationally superrigid.%
\end{itemize}
\end{lemma}

\begin{proof}
The surface $X$ is $G$-birationally superrigid by
Theorem~\ref{theorem:quintic-A5}. It is well-known that
$$
\mathrm{Aut}\big(G\big)\cong\mathbb{S}_{5},
$$
and every element in $\mathrm{Aut}(G)$ is induced by an~inner
isomorphism~of~the~group~$\mathbb{S}_{5}$, which implies that
the~pairs $(G,\phi)$ and $(G,\phi\circ\sigma)$ are conjugate for
every $\sigma\in\mathrm{Aut}(G)$.
\end{proof}

Note that if $X\cong\mathbb{P}^{2}$, then the~embedding
$$
\mathbb{A}_{5}\cong G\subset\mathrm{Aut}\big(X\big)\cong\mathbb{PSL}\Big(3,\mathbb{C}\Big)%
$$
is induced by a~three-dimensional representation of the~group
$\mathbb{A}_{5}$.

\begin{lemma}
\label{lemma:A5-22} Suppose that $X\cong\mathbb{P}^{2}$. Then
\begin{itemize}
\item the~variety $X$ is $G$-birationally rigid, and
$$
\mathbb{A}_{5}\cong\mathrm{Aut}^{G}\big(X\big)\subsetneq\mathrm{Bir}^{G}\big(X\big)\cong\mathbb{S}_{5},
$$
\item for every $\sigma\in\mathrm{Aut}(G)$, the~pairs $(G,\phi)$ and $(G,\phi\circ\sigma)$ are conjugate.%
\end{itemize}
\end{lemma}

\begin{proof}
It follows from~\cite{YauYu93} that we can assume that
the~embedding $G\subset\mathrm{Aut}(\mathbb{P}^{2})$ induces
the~action of the~group~$G$~on the~ring~$\mathbb{C}[x,y,z]$ such
that the~ring $\mathbb{C}[x,y,z]^{G}$ is generated by
$$
f_{2}\big(x,y,z\big)=x^{2}+yz,\
f_{6}\big(x,y,z\big)=8x^{4}yz-2x^2y^2z^2-x\Big(y^{5}+z^{5}\Big)+y^{3}z^{3},
$$
$$
f_{10}\big(x,y,z\big)=320x^6y^2z^2-160x^4y^3z^3+20x^2y^4z^4+6y^5z^5-4x\Big(y^5+z^5\Big)\Big(32x^4-20x^2yz+5y^2z^2\Big)+y^{10}+z^{10},
$$
$$
\left.\aligned
&f_{15}\big(x,y,z\big)=x\Big(y^{10}-z^{10}\Big)\Big(352x^4-160x^2yz-10y^2z^2\Big)+\Big(y^{5}-z^{5}\Big)\Big(3840x^8yz-1024x^{10}\Big)-\\
&-\Big(y^{5}-z^{5}\Big)\Big(3840x^6y^2z^2+1200x^4y^3z^3-100x^2y^4z^4+y^{10}+z^{10}+2y^5z^5\Big).\\
\endaligned
\right.
$$

It follows from \cite{YauYu93} that the~forms $f_{2}$, $f_{6}$,
$f_{10}$ and $f_{15}$ are related by~the~equation
$$
f_{15}^{4}=-1728f_{6}^{5}+f_{10}^{3}+720f_{2}f_{6}^{3}f_{10}-80f_{2}^{2}f_{6}f_{10}^{2}+64f_{2}^{3}\Big(5f_{6}^{2}-f_{2}f_{10}\Big)^{2}.
$$

Note that every $G$-semi-invariant polynomial in
$\mathbb{C}[x,y,z]$ must be also $G$-invariant, because the~group
$G$ is simple. In particular, there are no $G$-invariant curves in
$\mathbb{P}^{2}$ of degree $1$, $3$ and $5$.

Let $C\subset\mathbb{P}^{2}$ be the~curve that is given
by~the~equation
$$
f_{2}\big(x,y,z\big)=0\subset\mathbb{P}^{2}\cong\mathrm{Proj}\Big(\mathbb{C}\big[x,y,z\big]\Big),
$$
and let $\mathcal{P}$ be the~pencil on $\mathbb{P}^{2}$ that is
given by~the~equation
$$
\lambda f^{3}_{2}\big(x,y,z\big)+\mu f_{6}\big(x,y,z\big)=0,
$$
where $[\lambda:\mu]\in\mathbb{P}^{1}$. Then the~following
assertions hold:
\begin{itemize}
\item the~curve $C$ is unique $G$-invariant conic on $\mathbb{P}^{2}$,%
\item the~action of $G\cong\mathbb{A}_{5}$ on $C\cong\mathbb{P}^{1}$ induces an~embedding $C\subset\mathrm{Aut}(C)$,%
\item every $G$-invariant sextic curve on $\mathbb{P}^{2}$ belongs to~the~pencil $\mathcal{P}$,%
\item all curves in the~pencil  $\mathcal{P}$ are irreducible and
reduced except the~following curves:
\begin{itemize}
\item the~non-reduced curve that is given by the~equation $f_{2}^{3}(x,y,z)=0$,%
\item a~reduced reducible curve consisting of $6$ lines
$L_{1},\ldots,L_{6}$ such that the~set
$$
C\cap\Big(L_{1}\cup L_{2}\cup L_{3}\cup L_{4}\cup L_{5}\cup L_{6}\Big)%
$$
is the~unique $G$-orbit contained in $C$ that consists of $12$ points (cf. \cite{Spr77}),%
\end{itemize}
\item every $G$-orbit in $\mathbb{P}^{2}$ consists of at least $6$ points,%
\item there is a~unique curve $R\in\mathcal{P}$ such that
$|\mathrm{Sing}(R)|=6$, and the~subset
$$
\mathrm{Sing}\big(R\big)\subset\mathbb{P}^{2}
$$
is the~only $G$-orbit consisting of $6$ points.
\end{itemize}

Take any $\sigma\in\mathrm{Aut}(G)$. Then $\sigma$ is induced by
an~inner isomorphism~of $\mathbb{S}_{5}$. Thus, if
$$
\mathrm{Bir}^{G}\big(X\big)\cong\mathbb{S}_{5},
$$
then the~pairs
$(G,\phi)$ and $(G,\phi\circ\sigma)$ are conjugate.

Let $\rho$ be any element in $\mathrm{Aut}^{G}(X)$, and let $g$ be
any element in $G$. Then
$$
g\circ\rho=\rho\circ g^{\prime}
$$
for some $g^{\prime}\in G$. Then
$$
g\Big(\rho\big(C\big)\Big)=\rho\Big(g^{\prime}\big(C\big)\Big)=\rho\big(C\big),%
$$
which implies that $\rho(C)=C$. There is unique $G$-orbit
$\Lambda\subset C$ such that $|\Lambda|=12$. Then
$$
g\Big(\rho\big(\Lambda\big)\Big)=\rho\Big(g^{\prime}\big(\Lambda\big)\Big)=\rho\big(\Lambda\big),%
$$
which implies that $\rho(\Lambda)=\Lambda$. The subgroup in
$\mathrm{Aut}(C)$ that leaves invariant the~set $\Lambda$ is
finite, which implies that $\rho\in G$. Thus, we see that
$\mathrm{Aut}^{G}(X)=G$.

To complete the~proof, we must show that $X$ is $G$-birationally
rigid and $\mathrm{Bir}^{G}(X)\cong\mathbb{S}_{5}$.

Let $\Sigma\subset\mathbb{P}^{2}$ be a~$G$-orbit  such that
$|\Sigma|\leqslant 9$. Then it follows from \cite{Spr77} or
\cite{YauYu93} that
$$
\Sigma\cap C=\varnothing,
$$
and $|\Sigma|=6$ (cf. Section~\ref{section:icosahedron}). Then
$\Sigma$ is uniquely defined by the~equality $|\Sigma|=6$.

Let $\gamma\colon W\to X$ be a blow up of all points in $\Sigma$.
Then it follows from \cite[Proposition~1]{Hi08} that
$$
\mathrm{Aut}\big(W\big)\cong\mathbb{S}_{5},
$$
and $W$ is isomorphic to the~Clebsch cubic surface (see
\cite{DoIs06}). Put
$$
\tau=\gamma\circ\theta\circ\gamma^{-1},
$$
where $\theta$ is an~odd involution in
$\mathrm{Aut}(W)\cong\mathbb{S}_{5}$. Then
$\tau\not\in\mathrm{Aut}(\mathbb{P}^{2})$.

The~involution $\tau$ induces the~monomorphism
$\upsilon^{\prime}\colon G\to\mathrm{Aut}(\mathbb{P}^{2})$ such
that
$$
\upsilon^{\prime}\big(g\big)=\tau\circ\upsilon\big(g\big)\circ\tau^{-1}
$$
for every $g\in G$. Then $\upsilon^{\prime}$ is induced by
a~three-dimensional representation of the~group $\mathbb{A}_{5}$
that is different from the~representation that induces
the~monomorphism $\upsilon$ (cf. \cite[Section~9]{DoIs06}).

Let $E$ be the~reduced $\gamma$-exceptional divisor such that
$\gamma(E)=\Sigma$. Then there is a~commutative diagram
$$
\xymatrix{
&&W\ar@{->}[dl]_{\gamma}\ar@{->}[dr]^{\psi}&&\\
&\mathbb{P}^{2}\ar@{-->}[rr]_{\tau}&&\mathbb{P}^{2},&}
$$
where $\psi$ is a~blow down of the~curve $\theta(E)$ to the~set
$\Sigma$ (cf. proof of Theorem~\ref{theorem:quintic-A5}).

If the~group generated by $\tau$ untwists all $G$-maximal
singularities (see Definition~\ref{definition:untwisting}), then
$$
\mathrm{Bir}^{G}\big(X\big)=\Big\langle\mathrm{Aut}^{G}\big(X\big),\tau\Big\rangle\cong\mathbb{S}_{5}
$$
by Corollary~\ref{corollary:untwisting}, and $X$ is
$G$-birationally rigid by Corollary~\ref{corollary:G-rigidity}.
Hence, to complete the~proof, we must show that the~group
generated by $\tau$ untwists all $G$-maximal singularities.

Let  $\mathcal{M}$ be a~$G$-invariant linear system on the~surface
$X$ such that
\begin{itemize}
\item the~linear system $\mathcal{M}$ does not have fixed curves,%
\item the~log pair $(X, \mu\mathcal{M})$~is~not canonical at some
point $O\in X$, where $\mu\in\mathbb{Q}$ such~that
$$
K_{X}+\mu\mathcal{M}\equiv 0.
$$
\end{itemize}

Let $\Delta$ be the~$G$-orbit of the~point $O$. Then
$$
\mathrm{mult}_{P}\big(\mathcal{M}\big)>\frac{1}{\mu}
$$
for every point $P\in \Delta$. Let $M_{1}$ and $M_{2}$ be general
curves in $\mathcal{M}$. Then
$$
\frac{9}{\mu^{2}}=\frac{K_{X}^{2}}{\mu^{2}}=M_{1}\cdot M_{2}\geqslant\sum_{P\in\Delta}\mathrm{mult}_{P}\Big(M_{1}\cdot M_{2}\Big)\geqslant\sum_{P\in\Delta}\mathrm{mult}_{P}\big(M_{1}\big)\mathrm{mult}_{P}\big(M_{2}\big)=\sum_{P\in\Delta}\mathrm{mult}^{2}_{P}\big(\mathcal{M}\big)>\frac{|\Delta|}{\mu^{2}},%
$$
which implies that $|\Delta|<9$. Then $\Delta=\Sigma$. Put
$H=\gamma^{*}(\mathcal{O}_{\mathbb{P}^{2}}(1))$. Then
\begin{equation}
\label{equation:involution} \left\{\aligned
&\theta^{*}\big(H\big)\sim 5H-2E,\\
&\theta^{*}\big(E\big)\sim 12H-5E,\\
\endaligned
\right.
\end{equation}
because the~involution $\theta$ acts non-trivially on
$\mathrm{Pic}(W)$.

Put $\mathcal{M}^{\prime}=\tau(\mathcal{M})$. Let $\mu^{\prime}$
be a~positive rational number such that the~equivalence
$$
K_{X}+\mu^{\prime}\mathcal{M}^{\prime}\equiv 0
$$
holds. Then it follows from
the~equivalences~\ref{equation:involution} that
$$
\mu^{\prime}=\frac{3}{15\slash\mu-12\mathrm{mult}_{O}\big(\mathcal{M}\big)},
$$
which implies that $\mu^{\prime}>\mu$. Similarly, it follows from
the~equivalences~\ref{equation:involution} that
$$
\mathrm{mult}_{P}\big(\mathcal{M}^{\prime}\big)=\frac{6}{\mu}-5\mathrm{mult}_{O}\big(\mathcal{M}\big),
$$
for every point $P\in\Sigma$. Then the~log pair
$(X,\mu^{\prime}\mathcal{M}^{\prime})$ is canonical at every point
of the~set $\Sigma$.

Arguing as above, we see that the~singularities of the~log pair
$(X,\mu^{\prime}\mathcal{M}^{\prime})$ are canonical everywhere,
which implies that the~group generated by $\tau$ untwists all
$G$-maximal singularities.
\end{proof}

To complete the~proof of
Theorem~\ref{theorem:Cremona-Icosahedron}, we may assume that
$X\cong\mathbb{F}_{n}$, where $n\in\mathbb{N}\cup\{0\}$.

\begin{lemma}
\label{lemma:A5-3}  Suppose that there are monomorphism
$\iota\colon G\to\mathrm{Aut}(\mathbb{P}^{1})$, a~non-inner
isomorphism $\nu\in\mathrm{Aut}(G)$, and a~birational map
$\chi\colon X\dasharrow\mathbb{P}^{1}\times\mathbb{P}^{1}$ such
that
$$
\chi\circ\upsilon\big(g\big)\circ\chi^{-1}\Big(a,b\Big)=\Bigg(\iota\big(g\big)\Big(a\Big),\iota\circ\nu\big(g\big)\Big(b\Big)\Bigg)
$$
for every $g\in G$ and
$(a,b)\in\mathbb{P}^{1}\times\mathbb{P}^{1}$. Then  $(G,\phi)$ and
$(G,\phi\circ\sigma)$ are conjugate for any
$\sigma\in\mathrm{Aut}(G)$.
\end{lemma}

\begin{proof}
The~birational map $\chi$ induces the~monomorphism
$\grave{\upsilon}\colon
G\to\mathrm{Aut}(\mathbb{P}^{1}\times\mathbb{P}^{1})$ such that
$$
\grave{\upsilon}\big(g\big)=\chi\circ\upsilon\big(g\big)\circ\chi^{-1}
$$
for every $g\in G$. Then $\grave{\upsilon}$ is induced by
the~twisted diagonal action of the~group $\mathbb{A}_{5}$ on
$\mathbb{P}^{1}\times\mathbb{P}^{1}$.

Let us identify the~subgroup $\grave{\upsilon}(G)$ with the~group
$G$.

Take $\tau\in\mathrm{Aut}(\mathbb{P}^{1}\times\mathbb{P}^{1})$
such that  $\tau(a,b)=(b,a)$ for any
$(a,b)\in\mathbb{P}^{1}\times\mathbb{P}^{1}$. Then
$$
\big\langle G,\tau\big\rangle\cong\mathbb{S}_{5},
$$
which implies that the~pairs $(G,\phi)$ and $(G,\phi\circ\sigma)$
are conjugate  for every $\sigma\in\mathrm{Aut}(G)$, because every
isomorphism of the~group $\mathbb{A}_{5}$ is induced by an~inner
isomorphism~of~the~group $\mathbb{S}_{5}$.
\end{proof}

Therefore, to complete the~proof of
Theorem~\ref{theorem:Cremona-Icosahedron}, we may ignore any
difference between the~monomorphism $\upsilon$ and
$\upsilon\circ\sigma$, where $\sigma\in\mathrm{Aut}(G)$.

\begin{lemma}
\label{lemma:A5-4} Suppose that $n\ne 0$. Then $n$ is even, there
is a~commutative diagram
$$
\xymatrix{
&X\ar@{->}[d]_{\pi}\ar@{-->}[rr]^{\rho}&&\mathbb{P}^{1}\times\mathbb{P}^{1}\ar@{->}[d]^{\bar{\pi}}\\
&\mathbb{P}^{1}\ar@{=}[rr]&&\mathbb{P}^{1},&}
$$
where $\rho$ is a~birational map that induces the~monomorphism
$\bar{\upsilon}\colon
G\to\mathrm{Aut}(\mathbb{P}^{1}\times\mathbb{P}^{1})$ such that
$$
\bar{\upsilon}\big(g\big)=\rho\circ\upsilon\big(g\big)\circ\rho^{-1}
$$
for every $g\in G$, and $\bar{\upsilon}$ is induced by
the~product action of the~group $\mathbb{A}_{5}\times \mathrm{id}_{\mathbb{P}^{1}}$ on $\mathbb{P}^{1}\times\mathbb{P}^{1}$.%
\end{lemma}

\begin{proof}
Let $Z$ be the~section of $\pi$ such that $Z^{2}=-n$. Then $Z$ is
$G$-invariant.

The curve $Z$ contains $G$-orbits consisting of $12$, $20$ and
$30$ points (see \cite{Spr77}).

Let $\Sigma$ be the~$G$-orbit such that $\Sigma\subset Z$ and
$|\Sigma|=30$. Then there is a~commutative diagram
$$
\xymatrix{
&&U\ar@{->}[dl]_{\alpha}\ar@{->}[dr]^{\beta}&&\\
&X\ar@{->}[d]_{\pi}\ar@{-->}[rr]^{\psi}&& X_{1}\ar@{->}[d]^{\pi_{1}}\\
&\mathbb{P}^{1}\ar@{=}[rr]&&\mathbb{P}^{1},&}
$$
where $\psi$ is a birational map, $\pi_{1}$ is
a~$\mathbb{P}^{1}$-bundle, $\alpha$ is the~blow up of the~set
$\Sigma$, and $\beta$ is the~blow down of the~proper transforms of
the~fibers of $\pi$ that pass through the~points of the~set
$\Sigma$.

The birational map $\psi$ induces the~monomorphism
$\upsilon_{1}\colon G\to\mathrm{Aut}(X_{1})$ such that
$$
\upsilon_{1}\big(g\big)=\psi\circ\upsilon\big(g\big)\circ\psi^{-1}
$$
for every $g\in G$. Let us identify the~subgroup $\upsilon_{1}(G)$
with the~group $G$. Put $Z_{1}=\psi(Z)$. Then
$$
Z_{1}\cdot Z_{1}=Z\cdot Z-30=-n-30<0,
$$
which implies that $X_{1}\cong\mathbb{F}_{n+30}$ and $Z_{1}$ is
$G$-invariant.

The curve $Z_{1}$ contains $G$-orbits consisting of $12$, $20$ and
$30$ points (see \cite{Spr77}).

Let $\Sigma_{1}$ be the~$G$-orbit such that $\Sigma_{1}\subset
Z_{1}$ and
$$
\big|\Sigma_{1}\big|=20,
$$
let $P_{1}$ be a point in $\Sigma_{1}$, let $H_{1}$ be
the~stabilizer in $G$ of the~point $P_{1}$. Then
$H_{1}\cong\mathbb{Z}_{3}$.

Let $L_{1}$ be the~fiber of $\pi_{1}$ such that $P_{1}\in L_{1}$.
Then
$$
h_{1}\big(L_{1}\big)=L_{1}
$$
for every $h_{1}\in H_{1}$. Then there is a point $Q_{1}\in
L_{1}\setminus P_{1}$ such that $h_{1}(Q_{1})=Q_{1}$ for every
$h_{1}\in H_{1}$.

Let $\Lambda_{1}$ be the~$G$-orbit of the~point $Q_{1}$. Then
$|\Lambda_{1}|=20$ and
$$
\Lambda_{1}\cap Z_{1}=\varnothing,
$$
because $Z_{1}$ is $G$-invariant and $Q_{1}\not\in Z_{1}$. Thus,
there is a~commutative diagram
$$
\xymatrix{
&&U_{1}\ar@{->}[dl]_{\alpha_{1}}\ar@{->}[dr]^{\beta_{1}}&&\\
&X_{1}\ar@{->}[d]_{\pi_{1}}\ar@{-->}[rr]^{\psi_{1}}&& X_{2}\ar@{->}[d]^{\pi_{2}}\\
&\mathbb{P}^{1}\ar@{=}[rr]&&\mathbb{P}^{1},&}
$$
where $\psi_{1}$ is a birational map, $\pi_{2}$ is
a~$\mathbb{P}^{1}$-bundle, $\alpha_{1}$ is the~blow up of the~set
$\Lambda_{1}$, and $\beta_{1}$ is the~blow down of the~proper
transforms of the~fibers of $\pi_{1}$ that pass through the~points
of the~set $\Lambda_{1}$.

The birational map $\psi_{1}$ induces the~monomorphism
$\upsilon_{2}\colon G\to\mathrm{Aut}(X_{2})$ such that
$$
\upsilon_{2}\big(g\big)=\psi_{1}\circ\upsilon_{1}\big(g\big)\circ\psi_{1}^{-1}
$$
for every $g\in G$. Let us identify the~subgroup $\upsilon_{2}(G)$
with the~group $G$. Put $Z_{2}=\psi_{1}(Z_{1})$. Then
$$
Z_{2}\cdot Z_{2}=Z_{1}\cdot Z_{1}+20=Z\cdot Z-30+20=-n-10<0,
$$
which implies that $X_{2}\cong\mathbb{F}_{n+10}$ and $Z_{2}$ is
$G$-invariant.

The curve $Z_{2}$ contains $G$-orbits consisting of $12$, $20$ and
$30$ points (see \cite{Spr77}).

Let $\Sigma_{2}$ be the~$G$-orbit such that $\Sigma_{2}\subset
Z_{2}$ and
$$
\big|\Sigma_{1}\big|=12,
$$
let $P_{2}$ be a point in $\Sigma_{2}$, let $H_{2}$ be
the~stabilizer in $G$ of the~point $P_{2}$. Then
$H_{2}\cong\mathbb{Z}_{5}$.

Let $L_{2}$ be the~fiber of $\pi_{2}$ such that $P_{2}\in L_{2}$.
Then
$$
h_{2}\big(L_{2}\big)=L_{2}
$$
for every $h_{2}\in H_{2}$. Then there is a point $Q_{2}\in
L_{2}\setminus P_{2}$ such that $h_{2}(Q_{2})=Q_{2}$ for every
$h_{2}\in H_{2}$.

Let $\Lambda_{2}$ be the~$G$-orbit of the~point $Q_{2}$. Then
$|\Lambda_{2}|=12$ and
$$
\Lambda_{2}\cap Z_{2}=\varnothing,
$$
because $Z_{2}$ is $G$-invariant and $Q_{2}\not\in Z_{2}$. Thus,
there is a~commutative diagram
$$
\xymatrix{
&&U_{2}\ar@{->}[dl]_{\alpha_{2}}\ar@{->}[dr]^{\beta_{2}}&&\\
&X_{2}\ar@{->}[d]_{\pi_{2}}\ar@{-->}[rr]^{\psi_{2}}&& X^{\prime}\ar@{->}[d]^{\pi^{\prime}}\\
&\mathbb{P}^{1}\ar@{=}[rr]&&\mathbb{P}^{1},&}
$$
where $\psi_{2}$ is a birational map, $\pi^{\prime}$ is
a~$\mathbb{P}^{1}$-bundle, $\alpha_{2}$ is the~blow up of the~set
$\Lambda_{2}$, and $\beta_{2}$ is the~blow down of the~proper
transforms of the~fibers of $\pi_{2}$ that pass through the~points
of the~set $\Lambda_{2}$.

The birational map $\psi_{2}$ induces the~monomorphism
$\upsilon^{\prime}\colon G\to\mathrm{Aut}(X^{\prime})$ such that
$$
\upsilon^{\prime}\big(g\big)=\psi_{2}\circ\upsilon_{3}\big(g\big)\circ\psi_{2}^{-1}
$$
for every $g\in G$. Let us identify the~subgroup
$\upsilon^{\prime}(G)$ with the~group $G$. Put
$Z^{\prime}=\psi_{2}(Z_{2})$. Then
$$
Z^{\prime}\cdot Z^{\prime}=Z_{2}\cdot Z_{2}+12=Z_{1}\cdot Z_{1}+20+12=Z\cdot Z-30+20+12=-n+2\leqslant 0,%
$$
and the~curve $Z^{\prime}$ is a~$G$-invariant section of
$\pi^{\prime}$. Note that $X^{\prime}\cong\mathbb{F}_{n-2}$.

Put $\nu=\psi_{2}\circ\psi_{1}\circ\psi$. Then
$(X^{\prime},\xi\circ\nu,\upsilon^{\prime})$ is a~regularization
of the~pair $(G,\phi)$.

So far, we constructed the~commutative diagram
$$
\xymatrix{
&X\ar@{->}[d]_{\pi}\ar@{-->}[rr]^{\nu}&& X^{\prime}\ar@{->}[d]^{\pi^{\prime}}\\
&\mathbb{P}^{1}\ar@{=}[rr]&&\mathbb{P}^{1},&}
$$
where $\nu$ is a~birational map, $\pi^{\prime}$ is
a~$\mathbb{P}^{1}$-bundle, $X^{\prime}\cong\mathbb{F}_{n-2}$,
there is a~section $Z^{\prime}$ of $\pi^{\prime}$ such that
$$
Z^{\prime}\cdot Z^{\prime}=-n+2
$$
and $Z^{\prime}$ is $G$-invariant. If $n=2$, then we are done.
Similarly, we see that $n\ne 3$. Then $n\geqslant 4$.

Repeating the~above construction $\lfloor n/2\rfloor$ times, we
conclude the~proof.
\end{proof}

Thus, to complete the~proof of
Theorem~\ref{theorem:Cremona-Icosahedron}, we may assume that
$n=0$. Then
$$
X\cong\mathbb{P}^{1}\times\mathbb{P}^{1},
$$
and we may assume that $\pi\colon
\mathbb{P}^{1}\times\mathbb{P}^{1}\to\mathbb{P}^{1}$ is
a~projection to the~first factor.

\begin{lemma}
\label{lemma:A5-5} Suppose that there is a~$G$-invariant section
$Z$ of the~fibration $\pi$. Then there is a~commutative diagram
$$
\xymatrix{
&X\ar@{->}[d]_{\pi}\ar@{-->}[rr]^{\rho}&&\mathbb{P}^{1}\times\mathbb{P}^{1}\ar@{->}[d]^{\bar{\pi}}\\
&\mathbb{P}^{1}\ar@{=}[rr]&&\mathbb{P}^{1},&}
$$
where $\rho$ is a~birational map that induces the~monomorphism
$\bar{\upsilon}\colon
G\to\mathrm{Aut}(\mathbb{P}^{1}\times\mathbb{P}^{1})$ such that
$$
\bar{\upsilon}\big(g\big)=\rho\circ\upsilon\big(g\big)\circ\rho^{-1}
$$
for every $g\in G$, and $\bar{\upsilon}$ is induced by
the~product action of the~group $\mathbb{A}_{5}\times \mathrm{id}_{\mathbb{P}^{1}}$ on $\mathbb{P}^{1}\times\mathbb{P}^{1}$.%
\end{lemma}

\begin{proof}
Let $\Sigma$ be a~$G$-orbit such that $\Sigma\subset Z$ and
$|\Sigma|=60$. Then there is a~commutative diagram
$$
\xymatrix{
&&U\ar@{->}[dl]_{\alpha}\ar@{->}[dr]^{\beta}&&\\
&X\ar@{->}[d]_{\pi}\ar@{-->}[rr]^{\psi}&& X^{\prime}\ar@{->}[d]^{\pi^{\prime}}\\
&\mathbb{P}^{1}\ar@{=}[rr]&&\mathbb{P}^{1},&}
$$
where $\psi$ is a birational map, $\pi^{\prime}$ is
a~$\mathbb{P}^{1}$-bundle, $\alpha$ is the~blow up of the~set
$\Sigma$, and $\beta$ is the~blow down of the~proper transforms of
the~fibers of $\pi$ that pass through the~points of the~set
$\Sigma$.

The birational map $\psi$ induces the~monomorphism
$\upsilon^{\prime}\colon G\to\mathrm{Aut}(X^{\prime})$ such that
$$
\upsilon^{\prime}\big(g\big)=\psi\circ\upsilon\big(g\big)\circ\psi^{-1}
$$
for every $g\in G$. Let us identify the~subgroup
$\upsilon^{\prime}(G)$ with the~group $G$. Put
$Z^{\prime}=\psi(Z)$. Then
$$
Z^{\prime}\cdot Z^{\prime}=Z\cdot Z-60,
$$
the~curve $Z^{\prime}$ is a section of $\pi^{\prime}$, and
the~curve $Z^{\prime}$ is $G$-invariant.

Put $m=-Z^{\prime}\cdot Z^{\prime}$. If $m\geqslant 0$, then
$X^{\prime}\cong\mathbb{F}_{m}$ and we are done by
Lemma~\ref{lemma:A5-4}.

If $m<0$, then we can repeat the~above construction $\lceil
m/60\rceil$ times to complete the~proof.
\end{proof}

Let $\tau$ be the~biregular involution of
$\mathbb{P}^{1}\times\mathbb{P}^{1}$ such that
$$
\tau\big(a,b\big)=\big(b,a\big)
$$
for every $(a,b)\in\mathbb{P}^{1}\times\mathbb{P}^{1}$. Then
$\tau$ induces the~monomorphism $\upsilon^{\prime}\colon
G\to\mathrm{Aut}(\mathbb{P}^{1}\times\mathbb{P}^{1})$ such that
$$
\upsilon^{\prime}\big(g\big)=\tau\circ\upsilon\big(g\big)\circ\tau
$$
for every $g\in G$. Then $(\mathbb{P}^{1}\times\mathbb{P}^{1},
\xi\circ\tau,\upsilon^{\prime})$ is a~regularization of the~pair
$(G,\phi)$.

\begin{lemma}
\label{lemma:A5-7} The following assertions are equivalent:
\begin{itemize}
\item the~monomorphism $\upsilon$ is induces by the~product action of $\mathrm{id}_{\mathbb{P}^{1}}\times \mathbb{A}_{5}$,%
\item the~monomorphism $\upsilon^{\prime}$ is induces by the~product action of $\mathbb{A}_{5}\times\mathrm{id}_{\mathbb{P}^{1}}$.%
\end{itemize}
\end{lemma}

\begin{proof}
The required assertion is obvious.
\end{proof}

Let us fix a~monomorphism of~groups $\iota\colon
G\to\mathrm{Aut}(\mathbb{P}^{1})$, and let us fix
some~non-inner~iso\-mor\-phism $\nu\in\mathrm{Aut}(G)$. It follows
from Lemmas~\ref{lemma:A5-3}, \ref{lemma:A5-5} and
\ref{lemma:A5-7} that we may assume~that
$$
\upsilon(g)\big(a,b\big)=\Big(\iota(g)\big(a\big),\iota\circ\nu(g)\big(b\big)\Big)
$$
for every $g\in G$ and every
$(a,b)\in\mathbb{P}^{1}\times\mathbb{P}^{1}$.

\begin{lemma}
\label{lemma:A5-6} There is a~$G$-invariant section of
the~projection $\pi$.
\end{lemma}

\begin{proof}
Explicit computations shows that there is a~$G$-invariant curve
$Z\subset\mathbb{P}^{1}\times\mathbb{P}^{1}$ such that
$$
Z\sim\pi^{*}\Big(\mathcal{O}_{\mathbb{P}^{1}}\big(7\big)\Big)\otimes\big(\pi\circ\tau\big)^{*}\Big(\mathcal{O}_{\mathbb{P}^{1}}\big(1\big)\Big),
$$
which implies that $Z$  is a~$G$-invariant section of
the~projection $\pi$.
\end{proof}

The assertion of Theorem~\ref{theorem:Cremona-Icosahedron} is
proved.

\end{document}